\documentclass{amsart}

\usepackage[utf8]{inputenc}
\usepackage[english]{babel}
\usepackage{graphicx}
\usepackage{amssymb}
\usepackage{amsmath,mathtools}
\usepackage{tikz-cd}
\usepackage{adjustbox}
\usepackage{mathabx}
\usepackage{framed}
\usepackage{hyperref}
\usepackage{csquotes}
\usepackage{bm}
\usepackage[backend=bibtex,
style=alphabetic,
bibencoding=ascii
]{biblatex}
\usepackage{rotating}

\newcommand{\GU}{\mathrm{GU}}
\newcommand{\mF}{\mathrm{F}}
\newcommand{\G}{\mathrm{G}}
\newcommand{\T}{\mathrm{T}}
\newcommand{\SU}{\mathrm{SU}}
\newcommand{\U}{\mathrm{U}}
\newcommand{\Qp}{\mathbb{Q}_p}
\newcommand{\Z}{\mathbb{Z}}
\newcommand{\J}{\mathrm{J}}
\newcommand{\R}{\mathbb{R}}
\newcommand{\C}{\mathbb{C}}

\newcommand{\Res}{\mathrm{Res}}
\newcommand{\Irr}{\mathrm{Irr}}
\newcommand{\tr}{\mathrm{tr}}
\newcommand{\im}{\mathrm{im}}
\newcommand{\ind}{\mathrm{ind}}
\newcommand{\Stab}{\mathrm{Stab}}
\newcommand{\Sh}{\mathrm{Sh}}
\newcommand{\GL}{\mathrm{GL}}

\newcommand{\id}{\mathrm{id}}
\newcommand{\Id}{\mathrm{Id}}
\newcommand{\End}{\mathrm{End}}
\newcommand{\Ext}{\mathrm{Ext}}
\newcommand{\Vol}{\mathrm{Vol}}
\newcommand{\Groth}{\mathrm{Groth}}
\newcommand{\Gm}{{\mathbb{G}_m}}

\newcommand{\mc}{\mathcal}
\newcommand{\mf}{\mathfrak}
\newcommand{\Q}{\mathbb{Q}}
\newcommand{\X}{\mathbb{X}}
\newcommand{\F}{\mathbb{F}}
\newcommand{\mR}{\mathcal{R}}

\newcommand{\Hom}{\mathrm{Hom}}
\newcommand{\Int}{\mathrm{Int}}
\newcommand{\Ig}{\mathrm{Ig}}
\newcommand{\diag}{\mathrm{diag}}
\newcommand{\Lie}{\mathrm{Lie}}

\newcommand{\WD}{\mathrm{WD}}
\newcommand{\LL}{{}^L}
\newcommand{\Spf}{\mathrm{Spf}}
\newcommand{\re}{\mathrm{re}}
\newcommand{\Gr}{\mathrm{Groth}}
\newcommand{\Red}{\mathrm{Red}}
\newcommand{\Trans}{\mathrm{Trans}}
\newcommand{\der}{\mathrm{der}}
\newcommand{\mb}{\mathbf}
\newcommand{\D}{\mathbb{D}}
\newcommand{\ov}{\overline}
\newcommand{\Jac}{\mathrm{Jac}}
\newcommand{\Mant}{\mathrm{Mant}}
\newcommand{\A}{\mathbb{A}}
\newcommand{\Gal}{\mathrm{Gal}}

\newcommand{\Leta}{{}^L\eta}
\newcommand{\disc}{\mathrm{disc}}
\newcommand{\bas}{\mathrm{bas}}
\newcommand{\temp}{\mathrm{temp}}
\newcommand{\unit}{\mathrm{unit}}
\newcommand{\iso}{\mathrm{iso}}
\newcommand{\mH}{\mathrm{H}}
\newcommand{\SL}{\mathrm{SL}}
\newcommand{\LM}{{}^LM}
\newcommand{\rad}{\mathrm{rad}}
\newcommand{\LU}{{}^L\mathrm{U}}
\newcommand{\LGU}{{}^L\mathrm{GU}}

\newcommand{\ab}{\mathrm{ab}}
\newcommand{\N}{\mathbb{N}}

\newcommand{\inv}{\mathrm{inv}}

\newcommand{\el}{\mathrm{ell}}
\newcommand{\M}{\mathrm{M}}
\newcommand{\an}{\mathrm{an}}
\newcommand{\eff}{\mathrm{eff}}
\newcommand{\LJ}{{}^L\J}

\newtheorem{theorem}{Theorem}[section]
\newtheorem{lemma}[theorem]{Lemma}
\newtheorem{proposition}[theorem]{Proposition}
\newtheorem{corollary}[theorem]{Corollary}
\newtheorem{conjecture}[theorem]{Conjecture}
\theoremstyle{definition}
\newtheorem{definition}[theorem]{Definition}

\theoremstyle{remark}
\newtheorem{remark}[theorem]{Remark}

\numberwithin{equation}{section}

\title{The Kottwitz conjecture for unitary PEL type Rapoport--Zink spaces}
\author{Alexander Bertoloni Meli, Kieu Hieu Nguyen}

\begin{document}
\begin{abstract}
    In this paper we study the cohomology of PEL-type Rapoport--Zink spaces associated to unramified unitary similitude groups over $\Q_p$ in an odd number of variables. We extend the results of Kaletha--Minguez--Shin--White to construct a local Langlands correspondence for these groups and prove an averaging formula relating the cohomology of Rapoport--Zink spaces to this correspondence. We use this formula to prove the Kottwitz conjecture for the groups we consider.
\end{abstract}

\maketitle

\tableofcontents

\section{Introduction}

Shimura varieties play an important role in the global Langlands program, which predicts a link between automorphic representations of linear algebraic groups and Galois representations. Rapoport and Zink (\cite{RZ1}) introduced $p$-adic analogues of Shimura varieties defined as moduli spaces of $p$-divisible groups with additional structures. The $\ell$-adic $ (\ell \neq p) $ cohomology of these spaces should provide local incarnations of the Langlands correspondences and this is the subject of the Kottwitz conjecture (\cite[Conjecture 7.3]{RV1}). The goal of this paper is to prove the Kottwitz conjecture in the case of PEL type Rapoport--Zink spaces associated to unramified unitary similitude groups over $\Q_p$ in an odd number of variables. Prior to our work, the conjecture was proven for Lubin-Tate spaces by \cite{Boyer1}, \cite{HT1} and \cite{Boyer2}.  By duality \cite{Fal}, \cite{FGL}, \cite{SW17}, the conjecture is also known in the Drinfeld case. The  case of basic unramified EL type Rapoport--Zink spaces was proven by \cite{Shi1}, \cite{Far1} and the case of basic unramified PEL of unitary type of signature $(1, n-1)$ by \cite{KH1}. Kaletha and Weinstein (\cite{KW}) have proven, for all local Shimura varieties, a weakened form of the Kottwitz conjecture where, in particular, they do not consider the $W_{E_{\mu}}$-action.

We now describe our results in more detail. One considers triples $(\G, b, \mu)$ such that $\G$ is a connected reductive group over $\Q_p$ and $\mu$ is a minuscule cocharacter of $\G$ and $b$ is an element of the Kottwitz set $\mb{B}(\Q_p, \G, -\mu)$. Then Rapoport--Zink attach to triples $(\G, b, \mu)$ of \emph{PEL}-type a tower of rigid spaces $\mc{M}_{K_p}$ indexed by compact open subgroups $K_p \subset \G(\Q_p)$. 

Attached to the group $\G$ and the element $b$ is a connected reductive group $\J_b$ that is an inner form of a Levi subgroup of $\G$. The element $b$ is said to be \emph{basic} when $\J_b$ is in fact an inner form of $\G$. The tower $ (\mc{M}_{K_p})_{K_p \subset \G(\Q_p)} $ carries an action of $ \G(\Q_p) \times \J_b(\Q_p) \times W_{E_{\mu}} $ where $E_{\mu}$ is the field of definition of the conjugacy class of $\mu$. For each $i \geq 0$ one can take the compactly supported $\ell$-adic cohomology $H^i_c(\mc{M}_{K_p}, \ov{\Q}_{\ell})$ of $\mc{M}_{K_p}$ and hence consider the cohomology space 
\[
H^{i,j} (\G, b, \mu)[\rho] := \mathop{\mathrm{lim}}_{\overrightarrow{K_p}} \Ext^j_{J_b(\Q_p)}( H_c^{i}(\mathcal{M}_{K_p}, \overline{\Q}_{\ell}), \rho).
\]

Then the Kottwitz conjecture describes the homomorphism of Grothendieck groups $\Mant_{\G, b, \mu}: \Groth(\J_b(\Q_p)) \to \Groth(G(\Q_p) \times W_{E_{\mu}})$ given by
\[
\Mant_{\G, b, \mu} (\rho) := \sum_{i,j} (-1)^{i+j} H^{i,j} (\G, b, \mu)[\rho](- \dim \mc{M}^{\an}),
\]
in the case when $b$ is basic and $\rho$ is an irreducible admissible representation of $\J_b(\Q_p)$ with supercuspidal $L$-parameter. This means that under the local Langlands correspondence, the $L$-parameter $\psi_{\rho}: W_{\Q_p} \times \SL_2(\C) \to \LL\J_b$ is trivial when restricted to the $\SL_2(\C)$ factor and $\psi_{\rho}$ does not factor through a proper Levi subgroup of $\LL\J_b$.  

The Kottwitz conjecture states that
\begin{conjecture}{\label{Kottwitzconj}}
For irreducible admissible representations $\rho$ of $\J_b(\Q_p)$ with supercuspidal $L$-parameter, we have the following equality in $\Groth(\G(\Q_p) \times W_{E_{\mu}})$:
\begin{equation*}
  \mathrm{Mant}_{\G, b, \mu}( \rho )=
 \sum\limits_{\pi \in \Pi_{\psi_{\rho}}(\G)} [\pi][ \Hom_{\ov{\mathcal{S}}_{\psi_{\rho}}}(\iota_{\mf{w}}(\rho) \otimes \iota_{\mf{w}}(\pi)^{\vee}, r_{-\mu} \circ \psi_{\rho}) \otimes | \cdot |^{-\langle \rho_{\G}, \mu \rangle}],   
\end{equation*}
where $\Pi_{\psi_{\rho}}(\G)$ is the $L$-packet of irreducible admissible representations of $\G(\Q_p)$ attached to $\psi_{\rho}$. 
\end{conjecture}
We have not defined all the notation appearing in this conjecture, but this is described in detail in \S 5.

The main goal of this paper is to prove Conjecture \ref{Kottwitzconj} when $\G = \GU$ is an unramified unitary similitude group over $\Q_p$ in an odd number of variables and the datum $(\GU, b, \mu)$ is basic and of PEL-type. Of course, to make sense of the Kottwitz conjecture for $\GU$, one needs to establish the local Langlands correspondence for this group and show it satisfies an expected list of desiderata. In particular, one needs to check that the $L$-packet $\Pi_{\psi_{\rho}}$ has the expected structure determined by finite group $\ov{\mc{S}}_{\psi_{\rho}}$ related to the centralizer group of $\psi_{\rho}$ in $\widehat{\J_b}$ and satisfies the \emph{endoscopic character identities}. 

Prior to this work, such a local Langlands correspondence was known for unitary groups by the works \cite[Theorem 2.5.1, Theorem 3.2.1]{C.P.Mok} and \cite[Theorem 1.6.1]{KMSW}. These authors work with the \emph{arithmetic normalization} of the local Langlands correspondence whereby the Artin map is normalized so that uniformizers correspond to arithmetic Frobenius morphisms. However, it is more convenient for us to work with the opposite normalization. In Theorem \ref{itm: local} we use Kaletha's results in \cite{KalContra} on the compatibility of local Langlands correspondence and the contragredient to define a local Langlands correspondence for unitary groups under the \emph{geometric normalization} whereby the Artin map takes uniformizers to geometric Frobenius morphisms.

We next construct a local Langlands correspondence for our groups $\GU$ by lifting the result for unitary groups to the group $\U \times Z(\GU)$ and then descending it to $\GU$. We can carry out such an analysis because the map $\U \times Z(\GU) \to \GU$ is a surjection on $\Q_p$ points for odd unitary groups. This property fails in the even case and is in fact the main reason we consider odd unitary similitude groups. We get
\begin{theorem}[Theorem \ref{itm: local}, Theorem \ref{localECIpsi+}, \S \ref{ECIforGU}]
The local Langlands correspondence for odd unitary similitude groups is known and satisfies the properties of \cite[Theorem 1.6.1]{KMSW}, in particular, the endoscopic character identities.
\end{theorem}

With the local Langlands correspondence in hand, we can describe our proof of Conjecture \ref{Kottwitzconj} for the groups we consider. Our method of proof is similar to that of \cite{Shi1} and crucially uses the \emph{endoscopic averaging formulas} of \cite{BM2}. We briefly describe these formulas. Suppose that $\mc{H}^{\mf{e}} = (\mH, s, \Leta)$ is an elliptic endoscopic datum for $\GU$. Then there exists a complicated map
\begin{equation*}
    \Red^{\mc{H}^{\mf{e}}}_b: \Groth^{st}(\mH(\Q_p)) \to \Groth(J_b(\Q_p)),
\end{equation*}
whose precise definition is given in \S5.2. We remark that $\Groth^{st}(\mH(\Q_p))$ denotes the subgroup of $\Groth(\mH(\Q_p))$ with stable virtual character. Associated to each tempered $L$-parameter $\psi^{\mH}$ of $\mH$, we have a stable character denoted by $S\Theta_{\psi^{\mH}}$. Suppose that $\psi$ is an $L$-parameter of $\GU$ with parameter $\psi^{\mH}$ of $\mH$ such that $\psi = \Leta \circ \psi^{\mH}$. Then the endoscopic averaging formula is the following identity in $\Groth(\GU(\Q_p) \times W_{E_{\mu}}$:
\begin{equation}{\label{EAF}}
\sum\limits_{b \in \mb{B}(\Q_p, \GU, - \mu)} \Mant_{\GU,b, \mu}(\Red^{\mc{H}^{\mf{e}}_b}(S\Theta_{\psi^{\mH}}))=
\end{equation}
\begin{equation*}
 \sum\limits_{\rho} \sum\limits_{\pi_p \in
 \Pi_{\psi}(\GU, \varrho)} \langle \pi_p, \eta(s) \rangle \frac{\tr( \eta(s) \mid V_\rho)}{\dim \rho} \pi_p \boxtimes [\rho \otimes | \cdot |^{-\langle \rho_{\GU}, \mu \rangle}], 
\end{equation*}
where the first sum on the right-hand side is over irreducible factors of the representation $r_{- \mu} \circ \psi$ and $V_{\rho}$ is the $\rho$-isotypic part of $r_{- \mu} \circ \psi$. The averaging formula is derived in \cite{BM2} under a substantial list of assumptions. In this paper, we verify these assumptions for supercuspidal parameters and hence prove:
\begin{theorem}
For supercuspidal parameters $\psi$ of $\GU$, the endoscopic averaging formulas hold. 
\end{theorem}

 For the sake of completeness, we briefly recall the strategy of the proof of this result as well as explain the important assumptions. The proof is via global methods. Thus we consider a global unitary similitude group $ \mb{GU} $ defined over $\Q$ and a Shimura variety $ \Sh $ attached to $ \mb{GU} $ which ``globalizes" our Rapoport--Zink space. In particular, we have $ \mb{GU}_{\Q_p} = \GU $. We deduce the averaging formula by combining the Mantovan formula (\cite[Theorem 22]{Man2}, \cite[Theorem 6.26]{LS2018})
\begin{equation} \phantomsection \label{itm : Mantovan's equality} 
 H^*_c (\Sh, \mathcal{L}_{\xi}) = \sum\limits_{b \in \mb{B}(\Q_p, \GU, - \mu)} \Mant_{\GU,b, \mu} (H_c^*(\Ig_b, \mathcal{L}_{\xi}))   
\end{equation}
and the trace formulas for Shimura and Igusa varieties (\cite[Theorem 7.2]{kot7}, \cite[Theorem 13.1]{Shi4}, \cite[Theorem 7.2]{Shi3}). We denote respectively by $ H^*_c (\Sh, \mathcal{L}_{\xi}) $ and $ H_c^*(\Ig_b, \mathcal{L}_{\xi}) $ the alternating sums of the compactly supported cohomology of Shimura and Igusa varieties evaluated at the $\ell$-adic sheaf $ \mathcal{L}_{\xi} $ associated to some irreducible algebraic representation $\xi$ of $\mb{GU}$.

To carry out this approach, we need to define global $A$-parameters of $\mb{GU}$ without referring to the conjectural global Langlands group. We do so by adapting Arthur's approach (also used in \cite{C.P.Mok} and \cite{KMSW}) where global parameters correspond to self-dual formal sums of cuspidal automorphic representations of $\GL_n$. For us, a parameter $\mb{\psi}$ of $\mb{GU}$ consists of a pair $(\dot{\psi}, \chi)$ such that $\dot{\psi}$ is a global parameter of $\mb{U}^*$ in the sense of \cite{C.P.Mok} and $\chi$ is an automorphic character of $Z(\mb{GU})(\A)$. We attach global $A$-packets to these parameters in the generic case and prove they satisfy the \emph{global multiplicity formula} (Proposition \ref{global multiplicity formula}).

One important step in the proof of the averaging formula is the process of stabilisation and destabilisation of the trace formula for the cohomology of Shimura and Igusa varieties following \cite{kot7} and \cite{Shi3}. The goal is to relate both sides of the equality (\ref{itm : Mantovan's equality}) to the \emph{global multiplicity formula}. In order to achieve this, we need to prove a technical hypothesis concerning stable orbital integrals. More precisely, let $\mH$ be an endoscopic group of $\GU$ and $f^{\mH}$ a test function satisfying some local ``cuspidality'' conditions. We want to show that $ST^{\mH}_{\el} (f^{\mH}) = ST^{\mH}_{\disc} (f^{\mH}) $ where $ST^{\mH}_{\el} (f^{\mH})$ is a sum of stable orbital integrals of $\mH$ with respect to $f^{\mH}$ and $ST^{\mH}_{\disc} (f^{\mH})$ is, loosely speaking, the traces of all automorphic representations of $ \mH(\A) $ evaluated against $f^{\mH}$. This hypothesis is proven in Section \S \ref{STell=STdisc}.  

Once we have done the destabilisation step, we can put everything into Equation (\ref{itm : Mantovan's equality}) and derive the averaging formula. However, at this point the equality (\ref{itm : Mantovan's equality}) is still quite complicated and we need to solve a lifting problem in order to extract the desired information. More precisely, for our choice of connected reductive group $\mb{GU}$ over $\Q$ such that $\mb{GU}_{\Q_p} = \GU$ and a cuspidal $L$-parameter $ \mb{\psi} $ of $\mb{GU}_{\Q_p}$, we need to construct global $L$-parameters $ \Dot{\mb{\psi}}$ lifting $\psi$ and satisfying a number of conditions. For instance, we need to precisely control the centralizer group of $\mb{\psi}$ in $\widehat{\mb{GU}}_{\Q_p}$. These lifting problems are studied in \cite{ArthurBook}, \cite{KMSW} and we adapt their arguments to the unitary similitude case (Section \S \ref{global liftings}). 

With the endoscopic averaging formula in hand, we prove the Kottwitz conjecture in \S6. To do so, we observe that $\Red^{\mc{H}^{\mf{e}}}_b(S\Theta_{\psi^{\mH}})=0$ whenever $b$ is non-basic and $\psi$ is supercuspidal. Hence, in this case, the only term on the left-hand side of the endoscopic averaging formula is the one for $b$ basic. We then combine the formulas for each elliptic $\mc{H}^{\mf{e}}$ to deduce the conjecture.

\subsection{Acknowledgements}
We would like to thank Tasho Kaletha for helpful discussions related to the arithmetic and geometric normalizations of the local Langlands correspondence. We are also grateful to Pascal Boyer, Laurent Fargues, and Sug Woo Shin for many helpful conversations regarding this paper. The first author was partially supported by NSF RTG grants DMS-1646385 and DMS-1840234.  The second author was supported by ERC Consolidator Grant 770936: NewtonStrat.

 \section{Automorphic representations}{\label{section1}}
\subsection{The groups}
Let $F$ be a field of characteristic $0$, $E$ a quadratic extension of $F$ and fix an algebraic closure $\ov{F}$. Let $J \in \GL_n(F)$ be the anti-diagonal matrix defined by $J=(J_{i,j})$ such that $J_{i,j} =(-1)^{i+1}\delta_{i, n+1-j}$. We now define quasi-split groups $\U^*_{E/F}(n)$ and $\GU^*_{E/F}(n)$ over $F$ as follows. We define $\U^*_{E/F}(n)(\ov{F}) = \GL_n(\ov{F})$ and $\GU^*_{E/F}(n)(\ov{F})=\GL_n(\ov{F}) \times \GL_1(\ov{F})$. Then we give $\GU^*_{E/F}(n)(\ov{F})$ an action of $\Gamma_F :=\Gal(\ov{F}/F)$ whereby $\sigma \in \Gal(\ov{F}/F)$ acts by
\begin{equation*}
    \sigma_{\GU} : \begin{cases} (g, c) \mapsto (\sigma(g), \sigma(c)) & \sigma \in \Gamma_E \\ (g, c) \mapsto (cJ \sigma(g)^{-t}J^{-1}, \sigma(c)) & \sigma \notin \Gamma_E \end{cases}.
\end{equation*}
We get an action of $\Gamma_F$ on $\U^*_{E/F}(n)(\ov{F})$ by restriction.

We also need to define slightly more general groups $\G(\U_1 \times \dots \times \U_k)^*$ defined by
\begin{equation*}
   \G(\U^*(n_1) \times ... \times \U^*(n_k)) := \left\{ (g_1, \cdots, g_k) \in \GU^*_1 \times \cdots \times \GU^*_k | c(g_1) = \cdots =c(g_k) \right\}.
\end{equation*}

In this paper, we only need to consider the case where $F$ is one of $\Q_v$ or $\Q$. We now fix for once and for all a prime $p$ and a quadratic imaginary extension $E/\Q$ that is inert at $p$. At each place $v$ of $\Q$ we get a rank two etale algebra $E_v$ over $\Q_v$. Since we will not change $E$, we can unambiguously use the notations $\U^*(n)$ and $\GU^*_n$ for the global groups we have defined and $\U^*_{\Q_v}(n)$ and $\GU^*_{\Q_v}(n)$ for the local groups (for $v$ that do not split over $E$). 

The global groups we consider in this paper will be inner forms of $\GU^*_n$ coming from hermitian forms. Namely, let $V$ be an $n$-dimensional $E$-vector space equipped with a hermitian form $\langle  \bullet , \bullet  \rangle$. Let $\GU(V)$, (resp. $\U(V) $) be the algebraic groups defined over $\Q$ by 
\[
\GU(V)(R) = \{ (g, c(g)) \in \GL(V \otimes_{\Q} R) \times \Gm(R) \mid \langle gx,gy \rangle  = c(g) \langle x, y \rangle, x,y \in  V \otimes_{\Q} R \}
\]
\[
\U(V)(R) = \{ g \in \GL(V \otimes_{\Q} R) \mid  \langle gx,gy \rangle  = \langle x, y \rangle , x,y \in  V \otimes_{\Q} R \}
\]
for any $ \Q $-algebra $R$.

In this paper we will assume that $n$ is an odd number and that the localization $\GU(V)_{\Q_v}$ at every finite place $v$ is quasi-split. Such groups exist and the quasi-split condition we impose at the finite places does not constrain the isomorphism class of the group at the archimedean place. Indeed we can define
\begin{equation*}
    I_{r,s} := \begin{pmatrix} I_r & 0 \\
    0 & -I_s \end{pmatrix},
\end{equation*}
where $I_r$ is the $r \times r$ identity matrix. Then for $V$ an $n$-dimensional $E$-vector space, 
\begin{equation*}
    \langle x , y \rangle := \sigma(x)^tI_{r,s}y,
\end{equation*}
for $r+s$ odd and $\sigma \in \Gamma_{E/\Q}$ the nontrivial element, gives a unitary similitude group of type $(r,s)$ at the archimedean place that is quasi-split at the finite places.

We need to verify that the groups $\GU(V)$ and $\U(V)$ arise as extended pure inner twists of $\GU^*(n)$ and $\U^*(n)$ respectively. Recall that a group $G$ over $F$ arises as an extended pure inner twist of $G^*$ if there exists a tuple $(\varrho, z)$ such that $\varrho: G^* \to G$ is an isomorphism over some finite extension $K/F$ and $z \in Z^1_{\bas}(\mc{E}_3(K/F), G^*(K))$ is such that for each $\sigma \in \Gamma_{K/F}$ and each $e \in \mc{E}_3(K/F)$ projecting to $\sigma$, we have
\begin{equation*}
    \varrho^{-1} \circ \sigma(\varrho) = \Int(z(e)).
\end{equation*}
The set $Z^1_{\bas}(\mc{E}_3(K/F), G^*(K))$ is defined as in \cite{Kot9}. In the case that $G^*$ has connected center, it is known by \cite[Proposition 10.4]{Kot9} that all inner twists of $G^*$ come from extended pure inner twists. In our case, we have $Z(\U^*(n))=\U(1)$ and $Z(\GU^*(n))=\Res_{E/\Q}\Gm$ so this is indeed the case. We can also consider extended pure inner twists for connected reductive groups over $F=\Q_v$. The definition is the same except for we have $z \in Z^1_{\bas}(\mc{E}_{\iso}(K/F), G^*(K))$ (where $\mc{E}_{\iso}(K/F)$ is the local gerb $\mc{E}(K/F)$ in \cite{Kot9}). As in \cite{Kot9}, we define:
\begin{equation*}
    \mb{B}(F, G) := \varprojlim\limits_K H^1_{\bas}(\mc{E}_3(K/F), G(K))
\end{equation*}
for $F$ a number field and
\begin{equation*}
    \mb{B}(F, G) :=  \varprojlim\limits_K H^1_{\bas}(\mc{E}_{\iso}(K/F), G(K)),
\end{equation*}
for $F$ a finite extension of $\Q_v$.
% In this paper, we use the presentations
% \[
% \GU(\Q_p) = \left\{ g \in \GL_n(E_p)  \mid \prescript{t}{}{g^*} J g = c(g) J, c(g) \in \Q^{\times} \right\},  
% \]
% \[
% \U(\Q_p) = \left\{ g \in \GL_n(E_p) \mid \prescript{t}{}{g^*} J g = J \right\},
% \]
% where $g^*$ denotes the conjugate of $g$ by the non-trivial element of $\Gal(E_p/\Q_p)$ and $J$ is the anti-diagonal matrix defined by $J=(J_{i,j})$ such that $J_{i,j} = (-1)^{i+1} \delta_{i, n+1-j}$.
%\begin{Bertie}
%I like to pick $J$ as in KMSW so that it has alternating signs on the anti-diagonal. What do you think? If we do so, then the $w$ below can be identified with $J$. Also it doesn't matter but maybe we should change $w{}^tg^{-1}w$ to $w{}^tgw^{-1}$ since that is more natural to read.

%Also probably it's easier to define $GU$ as a set of pairs $(g,c)$ satisfying a relation, right? I guess it doesn't matter too much.

%Finally, the above is the definition of the $F$-points of $GU$. We should probably make that clear.

%\textcolor{blue}{The definition using the set of pairs $(g, c)$ is good for me. It is a good idea to choose $J$ as in KMSW because we use his result all the time. So we need to change the notations... I agree with you. It is just the $F$-points, we should make it clear.}

%I changed the notation to $w$ above. I actually think I like your notation because we use $J$s a lot for things like $J_b$. You can delete these comments if you are happy with this.
%\end{Bertie}

A maximal torus $T$ defined over $\Q_v$ of $\GU^*_{\Q_v}(n)$ and with maximal split rank is given by the diagonal subgroup. We have
\[
T(\Q_v) = \left\{ (t_1, \cdots, t_n) \in (E_v^{\times})^{n} |  \exists c \in \Q_v^{\times}, \forall i \in \{ 1, \cdots, n \}, t_{i} \sigma(t_{ n + 1 - i}) = c \right\}.
\]

The  maximal split subtorus of $T$ is isomorphic to $(\Q_v^{\times})^{\frac{n-1}{2}} \times \Q_v^{\times}$. The relative Weyl group is $W_{\re} = (\Z/2\Z)^{\frac{n-1}{2}} \rtimes S_{\frac{n-1}{2}}$ where $S_{\frac{n-1}{2}}$ is the permutation group of $\{ 1, \cdots , \frac{n-1}{2}\}$. The normalizer of $A$ inside $\GU^*(n)(\Q_v)$ is generated by $A$ and the following elements: 
\[
S_{i,j} = \left( \begin{array}{ccc}
I^{i,j}_{\frac{n-1}{2}} & & \\
& 1 & \\
& & I^{\frac{n+1}{2} - i,\frac{n+1}{2} - j}_{\frac{n-1}{2}}
\end{array} \right), \qquad A_k = I^{\frac{n+1}{2} - k, \frac{n+1}{2} + k}_n
\]
where $ 1 \leq i,j,k \leq \frac{n-1}{2} $ and $I_n^{i,j}$ is the matrix with $1$ in the positions $(i,j), (j,i)$ and $(k,k)$ for $k \neq i,j$ and $0$ elsewhere.  

A minimal parabolic subgroup of $\GU^*_{\Q_v}(n)$ is
\[
P_{\text{min}} = \left\{ \left( \begin{array}{ccccccc}
    t_1 & &&&& &*   \\
    & \ddots & & &  &  \\
     & & t_{\frac{n-1}{2}} & & &  \\
     & & & x \\
     & & & & c \sigma(t_{\frac{n-1}{2}}^{-1}) \\
     & & & & & \ddots \\
     0 & & & & & & c\sigma(t_1^{-1})
\end{array} \right) | t_i, x \in E_v^{\times}, x\sigma(x) = c \right\} \bigcap  \GU^*(n)(\Q_v).
\]

From the description of unitary similitude groups, we see that there is an embedding $ E_v^{\times} \hookrightarrow Z({\GU})(\Q_v) $ given by $ t \longmapsto \diag( t, \cdots, t ) $. The tuple $(P_{\min}, T, \{E_{i,i+1}\}_{1 \leq i \leq n-1})$ gives a $\Gamma_{\Q_v}$-stable splitting of $\U^*_{\Q_v}(n)$.

We can identify $\widehat{\GU^*_{\Q_v}(n)}$ with  $\GL_n(\C) \times \C^{\times}$ and $\widehat{\U^*_{\Q_v}(n)}$ with $\GL_n(\C)$. We fix the standard $F$-splittings of $\GL_n(\C) \times \C^{\times}$ and $\GL_n(\C)$ consisting of the $(\widehat{T}, \widehat{B}, \{E_{i, i+1}\}_{1 \leq i \leq n-1})$ where $\widehat{T}$ and $\widehat{B}$ are the diagonal subgroup and upper triangular subgroup respectively. The action of the Weil group $W_{\Q_v}$ on these dual groups factors through $\Gal(E_v/\Q_v)$ and the non-trivial element $\sigma$ of $W_{E_v/\Q_v}$ acts via 
\begin{equation*}
    \sigma((g,c)) = (J {g}^{-t} J^{-1}, c\det(g))
\end{equation*}
and
\begin{equation*}
  \sigma(g) = (J {g}^{-t} J)  
\end{equation*}
respectively (see \cite[pg 38]{Mo} for details).
 
A maximal torus defined over $\Q_v$ of $\G( \U_1 \times \cdots \times \U_k )^*_{\Q_v}$ with maximal split rank is given by
\begin{equation*}
T = \left\{ \begin{matrix}  ((t_{1,1}, \cdots, t_{1, n_1}), \cdots, (t_{k, 1}, \cdots, t_{k, n_k})) \\ \in (E_v^{\times})^{n_1 + \cdots n_k} \end{matrix} \biggm\vert \begin{matrix}
\exists c \in \Q_v^{\times}, \forall i \in \{ 1, \cdots, k \}, \forall j \in \{ 1, \cdots, n_i \} \\ t_{i,j} \sigma(t_{i, n_i + 1 - j}) = c \end{matrix} \right\}.
\end{equation*}

If we denote $I$, resp. $J$ the set of indexes $i$ such that $ n_i $ is odd, resp. even, then a maximal split sub-torus of $T$ is isomorphic to
\[
\displaystyle A = \Q_p^{\times} \times (\prod_{i \in I} (\Q_p^{\times})^{\frac{n_i - 1}{2}}) \times (\prod_{j \in J} (\Q_p^{\times})^{\frac{n_j}{2}}).
\]
And the relative Weyl group is 
\begin{equation*}
    W_{\re} = \displaystyle (\prod_{i \in I} (\Z/2\Z)^{\frac{n_i-1}{2}} \rtimes S_{\frac{n_i-1}{2}}) \times (\prod_{j \in J} (\Z/2\Z)^{\frac{n_j}{2}} \rtimes S_{\frac{n_j}{2}}).
\end{equation*}
%where $S_{\frac{n-1}{2}}$ is the permutation group of $\{ 1, \cdots \frac{n-1}{2}\}$.
 
\begin{lemma} \phantomsection \label{itm : GU = ZxU}
We have the equality $ \G( \U^*(n_1) \times \cdots \times \U^*(n_k) )(\Q_v) = (\U^*(n_1) \times \cdots \times \U^*(n_k))(\Q_v) E_v^{\times}$ where $E^{\times}_v$ embeds into $\G(\U^*(n_1) \times ... \times \U^*(n_k))^*$ via the diagonal embedding. 
\end{lemma}
\begin{proof}
For simplicity, we prove the equality when $k = 1$. The general case follows by the same argument.

We just need to show that $c(E_v^{\times}) = c(\GU^*(n)(\Q_v))$. Because $\GU^*(n)(\Q_v)$ is quasi-split, we have the Bruhat decomposition $ \displaystyle \GU^*(n)(\Q_v) = \coprod_{w \in W_{\text{re}}} P_{\text{min}} \cdot  w \cdot P_{\text{min}} $. We see that $ c( P_{\text{min}} \cdot  w \cdot P_{\text{min}} ) = c(P_{\text{min}} \cdot  w) $ and $c(w) = 1$ by the above description of the normalizer of  $A$. Hence, $c(\GU^*(n)(\Q_p)) = c(P_{\text{min}})$ and then $ c(\GU^*(n)(\Q_p)) = c(T)$ since $c(\U_{P_{\text{min}}}) = 1$ where $\U_{P_{\text{min}}}$ is the unipotent radical of ${P_{\text{min}}}$. By the assumption $n$ is odd and the description of $T$, we have $c(\GU^*(n)(\Q_v)) = \{ x\sigma(x) | x \in E_v^{\times} \}$. Moreover, by the above injection $ E_v^{\times} \hookrightarrow Z({\GU^*(n)})(\Q_v) $, we also see that $ c(E_v^{\times}) = \{ x\sigma(x) | x \in E_v^{\times} \}$. Therefore $c(E_v^{\times}) = c(\GU^*(n)(\Q_v))$.
\end{proof}

We now recall some facts from the theory of endoscopy.
\begin{definition}{(cf. \cite[Definition 2.1]{BM2})}
    A \emph{refined endoscopic datum} for $G$ a connected reductive group over $F$ is a triple $(\mathrm{H},s,\eta)$ such that
    \begin{itemize}
        \item $\mH$ is a quasisplit reductive group over $F$.\\
        \item $s \in Z(\widehat{\mH})^{\Gamma_F}$\\
        \item $\eta: \widehat{\mH} \to \widehat{\G}$ such that the conjugacy class of $\eta$ is $\Gamma_F$-stable and $\eta(\widehat{\mH})=Z_{\widehat{G}}(\eta(s))^{\circ}$.
    \end{itemize}
    
    Suppose that $(\mH,s,\eta), (\mH', s', \eta')$ are refined endoscopic data. Then we say that an isomorphism $\alpha: \mH \to \mH'$ is an isomorphism of endoscopic data if $\widehat{\alpha}(s')=s$ and $\eta \circ \widehat{\alpha}$ and $\eta'$ are conjugate in $\widehat{\G}$.
    
    We say that a refined endoscopic datum $(\mH,s,\eta)$ is \emph{elliptic} if $(Z(\widehat{\mH})^{\Gamma_F})^{\circ} \subset Z(\widehat{\G})$.
    
    We denote the set of isomorphism classes of refined endoscopic data of $\G$ by $\mc{E}^r(\G)$.
\end{definition}
We record a set of representatives for the isomorphism classes of refined elliptic endoscopic data for $\U^*(n_1) \times ... \times \U^*(n_k)$ and $\G(\U^*(n_1) \times ... \times \U^*(n_k))$. The description will be analogous in the local case. Compare with \cite[Proposition 2.3.1]{Mo} but note that we have more isomorphism classes because we consider refined endoscopic data. For each $i$, choose non-negative natural numbers $n^+_i$ and $n^-_i$ such that $n^+_i + n^-_i = n_i$. 

Then, in the unitary case, let $\mathrm{H}$ be the group $\U^*(n^+_1) \times \U^*(n^-_1) \times ... \times \U^*(n^+_k) \times \U^*(n^-_k)$, let $\eta$ be the block diagonal embedding of dual groups and let\\ $s= ( I_{n^+_1}, -I_{n^-_1}, ..., I_{n^+_k}, -I_{n^-_k})$. These elliptic endoscopic data are all non-isomorphic and give a representative of each elliptic isomorphism class. 

In the unitary similitude case we let $\mathrm{H}$ be $G(\U^*(n^+_1) \times \U^*(n^-_1) \times ... \times \U^*(n^+_k) \times \U^*(n^-_k))$, let $\eta$ be the block diagonal embedding of dual groups, and let $s= (I_{n^+_1}, -I_{n^-_1}, ..., I_{n^+_k}, -I_{n^-_k}, 1)$. We further require that $n^-_1+ ... + n^-_k$ is even.

In each case, we can extend $\eta$ to get a map $\Leta$ of $L$-groups. This is done explicitly in \cite[Proposition 2.3.2]{Mo} (c.f \cite[pg52]{KMSW}).

%we have the Iwasawa's decomposition $GU_{E/F}(n) = C \cdot A \cdot N$ where $C$ is compact, $N$ is unipotent and $A$ is the maximal split torus. In particular, $\val_F( c(GU_{E/F}(n))) = \val_F(c(A))$. By the description of $A$ in the beginning of the section, we see that $A = \left\{ \diag (t_1, \cdots, t_{\frac{n-1}{2}}, x, c(t_1^{-1})^{*}, \cdots c(t_{\frac{n-1}{2}}^{-1})^*) | t_i, x \in F^{\times}, c = x^2  \right\}$ then $\val_F(c(A)) = 2\Z$ and $c(A) = \left\{ x^2 | x \in F^{\times} \right\}$. 

%Moreover $c(E^{\times}) = \Norm_{E/F}(E^{\times})$. %We see that $ \left\{ x^2 | x \in F^{\times} \right\} = \Norm_{E/F}(F^{\times}) \subset \Norm_{E/F}(E^{\times}) $ . In other word $ c(A) \subset c(E^{\times}) $
%If the extension $E/F$ is unramified, by Lang theorem, $\Norm_{E/F}(E^{\times}) = \theta^{2 \Z} \mathcal{O}_{F}^{\times} $ where $\theta$ is the uniformizer of $F$, hence $c(GU_{E/F}(n)) \subset c(E^{\times})$. %If the extension $E/F$ is ramified then $\Norm_{E/F}(t) = tt^* = t^2$ if $t \in F^{\times}$, in this case we also see that $c(A) \subset c(E^{\times}) = \Norm_{E/F}(E^{\times})$ 

\subsection{Automorphic representations of unitary groups}
We start by considering a local field  ${\Q_v}$ for $v$ any place of $\Q$. The local Langlands group is defined by $\mc{L}_{\Q_v} : = W_{\R} $ if $v = \infty$ and by $ W_{\Q_p} \times \SU(2)$ if $v = p$ is a prime. We also set $ \prescript{L}{}{\G} = \widehat{\G} \rtimes W_{\Q_v} $ as a topological group where $ \widehat{\G} $ is the Langlands dual group of $\G$. In our case we see that $\prescript{L}{}{\U^*_{\Q_v}(n)} = \GL_n(\C) \ltimes W_{\Q_v} $ and the group $W_{E_v}$ acts trivially on $\GL_n(\C)$.
\begin{definition}
A local $ L $-parameter for a connected reductive group $ \G $ defined over $ \Q_v $ is a continuous morphism $ \phi: \mc{L}_{\Q_v} \longrightarrow \prescript{L}{}{\G}$ which commutes with the canonical projections of $\mc{L}_{\Q_v} $ and $ \prescript{L}{}{\G} $ to $ W_{\Q_v} $ and such that $ \phi $ sends semisimple elements to semisimple elements.
\end{definition}
We denote $\Phi(\G)$ the set of $\widehat{\G}$-conjugacy classes of $L$-parameters. An $L$-parameter $\phi$ is called bounded, resp. discrete, if its image in $\prescript{L}{}{\G}$ projects to a relatively compact subset of $\widehat{\G}$, resp. if its image is not contained in any proper parabolic subgroup of $\prescript{L}{}{\G}$. We denote by $\Phi_{\text{bdd}}(\G)$, resp. $\Phi_2(\G)$ the subsets of $\Phi(\G)$ consisting of bounded $L$-parameters (resp. discrete parameters).

For global classifications, we will also need the notion of a local Arthur parameter.
\begin{definition}
A local $A$-parameter for a connected reductive group $\G$ defined over $\Q_v$ is a continuous morphism $ \psi: \mc{L}_{\Q_v} \times \SU(2) \longrightarrow \prescript {L}{}{\G}$ such that the image of $ \psi|_{\mc{L}_{\Q_v}} $ is a bounded $L$ -parameter.
\end{definition}
We denote by $ \Psi(\G) $ the set of equivalence classes of $A$-parameters. We also denote the set $ \Psi^{+}(\G) $ of the equivalence classes of continuous morphisms $ \psi $ as above but where $ \psi|_{\mc{L}_{\Q_v}} $ is not necessarily bounded . An $A$-parameter $ \psi $ (or $ \psi \in \Psi^{+}(\G) $) is said to be generic if $ \psi|_{\SU(2)} $ is trivial.

We also have a ``base change'' morphism of $L$-groups (\cite[page 9]{C.P.Mok}):
\[
\eta_B : \prescript{L}{}{\U^*_{\Q_v}(n)} \longrightarrow  \prescript{L}{}{\Res_{E_v/\Q_v}(\GL_{n, E_v})},	
\]
which allows us to identify the $L$-parameters of $\U^*_{\Q_v}(n)$ with self-dual $L$-parameters of $\GL_{n, E_v}$. In particular, if $ \phi \in \Phi(\G) $ is an $L$-parameter then its image by the base change map is just $ \Phi_{|\mc{L}_{E_v}} $.

For each $A$-parameter $ \psi \in \Psi^{+}(\G) $ we define centralizer groups as below, which play an important role in the local and global theory:
\[
S_{\psi} := \text{Cent}(\text{Im}\psi, \widehat{\G}), \quad \overline{S}_{\psi} := S_{\psi} / Z(\widehat{\G})^{\Gamma}, \quad \overline{\mathcal{S}}_{\psi} := \pi_0 (\overline{S}_{\psi}),
\]  
\[
 S_{\psi}^{\text{rad}} := (S_{\psi} \cap \widehat{\G}_{\text{der}})^{\circ}, \quad S_{\psi}^{\natural} := S_{\psi} / S_{\psi}^{\text{rad}}.
\]

We also need to introduce some notation for representations. We denote the set of isomorphism classes of irreducible admissible representations of a connected reductive group $\G$ by $\Pi(\G)$. We denote the set of tempered, essentially square integrable, and unitary representations by $\Pi_{\temp}(\G)$, $\Pi_2(\G)$, and $\Pi_{\unit}(\G)$ respectively. We denote $\Pi_{\temp}(\G) \cap \Pi_{2}(\G)$ by $\Pi_{2, \temp}(\G)$.
 
The following theorem gives the local Langlands correspondence for extended pure inner twists of $\U^*_{\Q_v}(n)$. We first fix some more notation.

%Consider an extended pure inner form $(\varrho, z)$ of the unitary group $U_{E/F}(n)$. 

%The data $ (\varrho, z) $ defines a character $ \chi_z \in X^*(Z (\widehat{U}_{E/F}(n))^{\Gamma}) $. Consider a parameter $ \psi \in \Psi(U_{E/F}(n)) $ and the associated centralizers
%\begin{equation} \phantomsection \label{character}
%Z(\widehat{U}_{E/F}(n))^{\Gamma} \hookrightarrow S_{\psi} \twoheadrightarrow S^{\natural}_{\psi}
%\end{equation}

%We denote by $ \text{Irr}(S^{\natural}_{\psi}, \chi_z)$ the set of characters in $S^{\natural}_{\psi} $ whose pull back via (\ref{character}) induces the character $ \chi_z $.
 
Fix an odd natural number $n$ and let $\U$ and $(\varrho, z)$ be an extended pure inner twist of $\U^*_{\Q_v}(n)$. Fix a $\Gamma_{\Q_v}$ invariant splitting of $\widehat{\U}$.  Then $(\varrho, z)$ induces a unique isomorphism $\widehat{\U} \cong \widehat{\U^*_{\Q_v}(n)}$ preserving the chosen $\Gamma_{\Q_v}$-splittings. The cocycle $z$ and the map
\begin{equation*}
    \mb{B}(\Q_v, \U) \to X^*(Z(\widehat{\U})^{\Gamma_{\Q_v}})
\end{equation*}
defines a character $\chi_z \in X^*(Z(\widehat{\U})^{\Gamma_{\Q_v}})$ by $z \mapsto \chi_z$.
Fix a non-trivial character $\varphi: \Q_v \to \C^{\times}$. Together with our chosen splitting of $\U^*_{\Q_v}(n)$, this gives a Whittaker datum $\mf{w}$ of $\U^*_{\Q_v}(n)$. Attached to each refined endoscopic datum $(\mathrm{H}, s, \eta)$ of $\U$ we have a canonical local transfer factor $\Delta[\mf{w}, \varrho, z]$ normalized as in \cite[\S4.1]{BM2}. These transfer factors correspond to the $\Delta_D$ factors of \cite[\S 5]{KS2}.  Since $\U$ is simply connected, we can extend $\eta$ to a map $\Leta$ of $L$-groups.

We stress that in this paper, we are using the \emph{geometric normalization} of the Langlands correspondence. This means that our Artin map is normalized to map a geometric Frobenius morphism to a uniformizer and explains why we normalize our transfer factors using the $\Delta_D$ normalization. This normalization is consistent with \cite{HT1} and \cite{BM2} but is the inverse of the normalization in \cite{KMSW}.

\begin{theorem}(\cite[Theorem. 1.6.1]{KMSW}, \cite[Theorem. 2.5.1, Theorem. 3.2.1]{C.P.Mok}) \phantomsection \label{itm: local}

Fix an odd natural number $n$ and let $\U$ and $(\varrho, z)$ be an extended pure inner twist of $\U^*_{\Q_v}(n)$. Fix a non-trivial character $\varphi: \Q_v \to \C^{\times}$. Together with our chosen splitting of $\U^*_{\Q_v}(n)$, this gives a Whittaker datum $\mf{w}$ of $\U^*_{\Q_v}(n)$. Then,
\begin{enumerate}
    \item For each generic $\psi \in \Psi(\U^*_{\Q_v}(n))$, there exists a finite set $\Pi_{\psi}(\U, \varrho)$ endowed with a morphism to $ \Pi_{\text {unit}} (\U)$. %The set $ \Pi_{\psi} (U)$ does not depend on $z$ and has an application
    Our choice of $\mf{w}$ defines a map
\[
\iota_{\mf{w}}: \Pi_{\psi} (\U, \varrho) \longrightarrow \Irr (S^{\natural}_{\psi}, \chi_z), \qquad \pi \mapsto \langle \pi, - \rangle,
\]
where $\Irr (S^{\natural}_{\psi}, \chi_z)$ denotes the set of irreducible representations of $S^{\natural}_{\psi}$ restricting on $Z(\widehat{U})^{\Gamma_{\Q_v}}$ to $\chi_z$. 
%The set $ \Pi_{\psi} (U) $ and $ \langle- | - \rangle $ depend only on the equivalence class $\Xi$ of $ (\varrho, z) $. 
\item The morphism $ \Pi_{\psi}(\U, \varrho) \longrightarrow \Pi_{\text{unit}}(\U) $ is injective and its image is contained in $ \Pi_{\text{temp}}(\U) $. If $ \Q_v $ is non-Archimedean then the map $\Pi_{\psi} (\U, \varrho) \longrightarrow \Irr(S^{\natural}_{\psi})$ is a bijection.
\item For each $\pi \in \Pi_{\text{unit}}(\U)$ in the image of $\Pi_{\psi}(\U, \varrho)$, the central character $\omega_{\pi} :  Z(\U) \longrightarrow \C^{\times} $ has a Langlands parameter given by the composition
\[
\mc{L}_{\Q_v} \xrightarrow{\phi_{\psi}} \prescript{L}{}{\U} \xrightarrow{\det \rtimes \id} \C^{\times} \rtimes W_{{\Q_v}}.
\]
\item Let $(\mathrm{H}, s, \Leta)$ be a refined endoscopic datum and let $\psi^{\mH} \in \Psi(\mH)$ be a generic parameter such that $\Leta \circ \psi^{\mH} = \psi$. If $f^{\mH} \in \mathcal{H}(\mathrm{H})$ and $f \in \mathcal{H}(\U)$ are two $\Delta[\mf{w}, \varrho, z]$-matching functions then we have
\begin{equation*}
 \displaystyle \sum_{\pi^{\mH} \in \Pi_{\psi^{\mH}}(\mH)} \langle \pi^{\mH}, s_{\psi^{\mH}} \rangle\tr(\pi^{\mH} \mid f^{\mH})   = e(\U) \sum_{\pi \in \Pi_{\psi}(\U, \varrho)} \langle \pi, \Leta(s) \cdot s_{\psi}\rangle \tr(\pi \mid f),
 \end{equation*}
 where $e( \cdot)$ is the Kottwitz sign.
\item We have 
\begin{equation*}
    \Pi_{\text{temp}}(\U) = \coprod_{\psi \in \Phi_{\text{bdd}} (\U^*_{\Q_v}(n))} \Pi_{\psi} (\U, \varrho)
\end{equation*}
and
\begin{equation*}
    \Pi_{\text{2, temp}} (\U) = \coprod_{\psi \in \Phi_{\text{2, bdd}} (\U^*_{\Q_v}(n))} \Pi_{\psi} (\U, \varrho).
\end{equation*}
\end{enumerate}
\end{theorem}
\begin{proof}
The contents of this theorem appear in the works of Mok (\cite[Theorem 2.5.1]{C.P.Mok}) and Kaletha--Minguez--Shin--White (\cite[Theorem 1.6.1]{KMSW}) except using the \emph{arithmetic normalization} of the Langlands correspondence. Hence our main task is to explain how we can use these arithmetically normalized correspondences to define a geometrically normalized correspondence.

For $\psi \in \Psi(\U^*_{\Q_v}(n))$ a generic parameter, we let $\Pi^A_{\psi}(\U, \varrho)$ denote the packet of representations assigned to $\psi$ by \cite[Theorem 1.6.1]{KMSW} (the letter $A$ stands for \emph{arithmetic normalization}) and define $\Pi_{\psi}(\U, \varrho)$ to consist of the contragredients of the representations in $\Pi^A_{\psi}(\U, \varrho)$. By the compatibility of the local Langlands correspondence with contragredients (proven in our case in Proposition \ref{itm : trace - dual}, cf. \cite[Equation (1.2)]{KalContra}), this is the same as saying that the packet $\Pi_{\psi}(\U, \varrho)$ of \cite{KMSW} is assigned to the parameter $\LL C \circ \psi$ where $\LL C$ is the extension to $\LU^*_{\Q_v}$ of the Chevalley involution, $\widehat{C}$, of $\widehat{G}$ as described in \cite[p. 3-4]{KalContra}.

We now define $\iota_{\mf{w}}$. For convenience, we will denote by $\iota^A_{\mf{w}}$ the maps given by the arithmetic normalization. Then we define for $\pi^{\vee} \in \Pi_{\psi}(\U, \varrho)$ that 
\begin{equation*}
    \iota_{\mf{w}}(\pi^{\vee}) = \iota^A_{\mf{w}^{-1}}(\pi)^{\vee},
\end{equation*}
where if $\mf{w}$ is the Whittaker datum $(B, \psi_{\Q_v})$, then $\mf{w}^{-1}$ is the datum $(B, \psi^{-1}_{\Q_v})$. Equivalently by taking the contragredient, we have
\begin{equation*}
    \iota_{\mf{w}}(\pi) = \iota^A_{\mf{w}}(\pi) \circ \widehat{C}^{-1}.
\end{equation*}

We now verify the endoscopy character identity which is (4) in the theorem. Fix $f \in \mc{H}(\U)$ and $f^{\mH} \in \mc{H}(\mH)$ a $\Delta[\mf{w}, \varrho, z]$-matching function. 

By Lemma \ref{arithmetic - geometric normalization}, we have that if $i_{\U}: \U(\Q_p) \to \U(\Q_p)$ is the inverse map, then $f^{\mH} \circ i_{\mH}$ and $f \circ i_{\U}$ are matching for the transfer factors $\Delta'[\mf{w}^{-1}, \varrho, z]$ with respect to the endoscopic datum $(\mH, s^{-1}, \Leta)$.  We use the letter $\Delta$ ($=\Delta_D$) resp. $ \Delta' $ to denote the transfer factors that are compatible with the geometric normalization resp. arithmetic normalization of the local Artin reciprocity map. 
Then we will show in Proposition \ref{itm : distribution-dual} that
\begin{align*}
       \sum_{\pi^{\mH} \in \Pi_{\psi^{\mH}}(\mH)} \langle \pi^{\mH}, s_{\psi^{\mH}} \rangle\tr(\pi^{\mH} \mid f^{\mH}) & = \sum_{\pi^{\mH} \in \Pi^A_{\LL C_{\mH} \circ \psi^{\mH}}(\mH)} \langle \pi^{\mH}, s_{\psi^{\mH}} \rangle\tr(\pi^{\mH} \mid f^{\mH})\\
       & =\sum_{\pi^{\mH} \in \Pi^A_{\psi^{\mH}}(\mH)} \langle \pi^{\mH}, s_{\psi^{\mH}} \rangle\tr(\pi^{\mH} \mid f^{\mH} \circ i_{\mH}).
\end{align*}

We now apply the endoscopic character identity proven in \cite[Theorem 1.6.1]{KMSW} to get that the above equals
\begin{align*}
  &e(\U) \sum_{\pi \in \Pi^A_{\psi}(\U, \varrho)} \tr(\iota^A_{\mf{w}^{-1}}(\pi) \mid  \Leta(s^{-1}) \cdot s_{\psi}) \tr(\pi \mid f \circ i) \\
  =&e(\U) \sum_{\pi \in \Pi_{\psi}(\U, \varrho)} \tr(\iota_{\mf{w}}(\pi)^{\vee} \mid  \Leta(s^{-1}) \cdot s_{\psi}) \tr(\pi^{\vee} \mid f \circ i). 
\end{align*}
Now, since $\tr(\pi^{\vee} \mid f) = \tr(\pi \mid f \circ i)$ (by Lemma \ref{traceeq}), we get that the above equals
\begin{equation*}
 e(\U) \sum_{\pi \in \Pi_{\psi}(\U, \varrho)} \tr( \iota_{\mf{w}}(\pi) \mid  \Leta(s) \cdot s_{\psi}) \tr(\pi \mid f),
\end{equation*}
as desired.
\end{proof}

\begin{proposition} \phantomsection \label{itm : distribution-dual}
Let $ \psi \in \Psi(\U^*_{\Q_v}(n)) $ be a tempered Langlands parameter. Then we have an equality of linear forms on $ \mathcal{H}(\U^*_{\Q_v}(n)) $
  \[
  \sum_{\pi \in \Pi^A_{\psi}(\U)} \langle \pi, s_{\psi} \rangle\tr(\pi \mid f) = \sum_{\pi \in \Pi^A_{\LL C_{\U} \circ \psi}(\U)} \langle \pi, s_{\psi} \rangle\tr(\pi \mid f \circ i_{\U}).
  \]
\end{proposition}
\begin{proof}
Thanks to the results in \cite{C.P.Mok}, \cite{KMSW}, the arguments in \cite[Theorem 4.8]{KalContra} also work in our case. Indeed, the group $\U^*_{\Q_v}(n)$ can be extended to an endoscopic datum $ \mf{e} = ( \U^*_{\Q_v}(n), \mathcal{U}, s, \xi ) $ of the triple $ (\GL_n, \theta, 1) $ for a suitable outer automorphism of $\GL_n$ preserving the standard splitting. Then $ \xi \circ \psi $ is a Langlands parameter of $\GL_n$ and denote by $\rho$ the representation of $\GL_n(\Q_p)$ assigned to  $ \xi \circ \psi $ by the local Langlands correspondence. There is a unique isomorphism $ X : \rho \longrightarrow \rho \circ \theta $ which preserves the $\mf{w}$-Whittaker functionals. Then we have the distribution
\[
f^n \longmapsto T\Theta^{\mf{w}}_{\xi \circ \psi} (f^n) = \tr \Big( v \longmapsto \int_{\GL_n(\Q_p)} f^n(g) \rho (g) X v dg \Big).
\]

Then by \cite[Theorem 3.2.1]{C.P.Mok} and \cite[Theorem 1.6.1 (4)]{KMSW} the linear form $ \displaystyle f \longmapsto S\Theta_{\psi} (f) = \sum_{\pi \in \Pi^A_{\psi}(\U)} \langle \pi, s_{\psi} \rangle\tr(\pi \mid f) $ is the unique distribution on $ \mathcal{H}(\U^*_{\Q_v}(n)) $ having the properties that 
\[
S\Theta_{\psi} (f) = T\Theta^{\mf{w}}_{\xi \circ \psi} (f^n)
\]
for all $ f \in \mathcal{H}(\U^*_{\Q_v}(n)) $ and $ f^n \in \mathcal{H}(\GL_n(\Q_p)) $ such that the $ ( \theta, \omega ) $-twisted orbital integrals of $f^n$ match the stable integrals of $f$ with respect to $ \Delta'[ \mf{w}, \mf{e}, z_{\mf{e}}  ] $. 

Once we have this characterization, the proof of \cite[Theorem 4.8]{KalContra} works without any change since Proposition 4.4, Corollary 4.5 and Corollary 4.7 in loc. cit. are valid for quasi-split unitary groups.
\end{proof}

\begin{proposition} \phantomsection \label{itm : trace - dual}
 Let $ \psi \in \Psi(\U^*_{\Q_v}(n)) $ be a tempered Langlands parameter and $\mf{w}$ be a Whittaker datum. Let $\pi$ be a representation in $\Pi^A_{\psi}(\U)$ and denote $ \iota^A_{\mf{w}} (\pi) = \rho $. Then
\begin{enumerate}
    \item[$\bullet$] The contragredient representation $ \pi^{\vee} $ belongs to the $L$-packet $\Pi^A_{ \prescript{L}{}{C} \circ \psi}(\U)$,
    \item[$\bullet$] $ \iota^A_{\mf{w}^{-1}}(\pi^{\vee}) = (\rho \circ C^{-1})^{\vee} $.
\end{enumerate}
\end{proposition}
\begin{proof}
These results are completely analogous to \cite[Theorem 4.9]{KalContra}. The same arguments carry over to our case since an analogue of Theorem 4.8 in loc.cit. is still valid for unitary groups (Proposition \ref{itm : distribution-dual}).
\end{proof}
We also have the following basic fact.
\begin{lemma}{\label{traceeq}}
    For $(\pi, V)$ an admissible representation of $\G(\Q_p)$ for $\G$ a reductive group and $f \in \mc{H}(\G)$, we have
\begin{equation*}
    \tr(\pi^{\vee} \mid f) = \tr(\pi \mid f \circ i).
\end{equation*}
\end{lemma}
\begin{proof}
Pick some compact open $K \subset \G(\Q_p)$ such that $f$ is $K$-bi-invariant.  Let $(\pi^{\vee}, V^{\vee})$ denote the contragredient of $\pi$ so that $V^{\vee} \subset V^*$ is the subspace of smooth vectors in the dual vector space $V^*$ of $V$. Then we note that $(V^{\vee})^K = (V^K)^*$ since each vector in $(V^*)^K$ is by definition smooth. 

The operator $\pi^{\vee}(f)$ acts on $(V^K)^*$ as the dual of the operator $\pi(f \circ i)$. Indeed for $v^* \in (V^K)^*$ and $w \in V^K$
\begin{align*}
    \pi^{\vee}(f)v^*(w) & = \int\limits_{\G(\Q_p)} f(g)\pi^{\vee}(g)v^*(w)dg\\
    &= \int\limits_{\G(\Q_p)} (f \circ i)(g^{-1}) v^*(\pi(g^{-1})w)dg\\
    &= v^* \left(\int\limits_{\G(\Q_p)} (f \circ i)(g) \pi(g)wdg \right)\\
    &= \pi(f \circ i)^*v^*(w),
\end{align*}
where the third equality uses that $\G$ is unimodular. This implies the desired equality of traces.
\end{proof}

When we consider global parameters, we will also need a version of Theorem \ref{itm: local} for $\psi \in \Psi^+(\U^*_{\Q_v}(n))$. The following theorem is essentially contained within the union of the remarks in \cite[pg33]{C.P.Mok} and \cite[\S 1.6.4]{KMSW}.
\begin{theorem}{\label{localECIpsi+}}
Fix an odd natural number $n$ and let $\U$ and $(\varrho, z)$ be an extended pure inner twist of $\U^*_{\Q_v}(n)$. Fix a non-trivial character $\varphi: \Q_v \to \C^{\times}$. Together with our chosen splitting of $\U^*_{\Q_v}(n)$, this gives a Whittaker datum $\mf{w}$ of $\U^*_{\Q_v}(n)$. Then,
\begin{enumerate}
    \item For each  generic $\psi \in \Psi^+(\U^*_{\Q_v}(n))$, there exists a finite set $\Pi_{\psi}(\U, \varrho)$ of possibly reducible or non-unitary representations of $\U$.
    Our choice of $\mf{w}$ defines a map
\[
\iota_{\mf{w}}: \Pi_{\psi} (\U, \varrho) \longrightarrow \Irr (S^{\natural}_{\psi}, \chi_z), \qquad \pi \mapsto \langle \pi, - \rangle,
\]
where $\Irr (S^{\natural}_{\psi}, \chi_z)$ denotes the set of irreducible representations of $S^{\natural}_{\psi}$ with central character $\chi_z$. Each $\pi \in \Pi_{\psi}(\U, \varrho)$ has a central character $\omega_{\pi}$, these characters are the same for each element of $\Pi_{\psi}(\U, \varrho)$.
%The set $ \Pi_{\psi} (U) $ and $ \langle- | - \rangle $ depend only on the equivalence class $\Xi$ of $ (\varrho, z) $. 
\item Let $(\mathrm{H}, s, \Leta)$ be a refined endoscopic datum and let $\psi^{\mH} \in \Psi^+(\mH)$ be a generic parameter such that $\Leta \circ \psi^{\mH} = \psi$. If $f^{\mH} \in \mathcal{H}(\mathrm{H})$ and $f \in \mathcal{H}(\U)$ are two $\Delta[\mf{w}, \varrho, z]$-matching functions then we have
\begin{equation*}
 \displaystyle \sum_{\pi^{\mH} \in \Pi_{\psi^{\mH}}(\mH)} \langle \pi^{\mH}, s_{\psi^{\mH}} \rangle\tr(\pi^{\mH} \mid f^{\mH})   = e(\U) \sum_{\pi \in \Pi_{\psi}(\U, \varrho)} \langle \pi, \Leta(s) \cdot s_{\psi}\rangle \tr(\pi \mid f),
 \end{equation*}
 where $e( \cdot)$ is the Kottwitz sign.
\end{enumerate}
\end{theorem}
\begin{proof}
We sketch the proof following ideas in \cite{KMSW} and \cite{C.P.Mok}. The proof of (1) is in \cite[\S 1.6.4]{KMSW}. They choose a standard parabolic subgroup $P^*=M^*N_{P^*}$ of $\U^*_{\Q_v}(n)$ that transfers to $\U$, a parameter $\psi_{M^*} \in \Psi(M^*)$ and a character $\lambda \in \Hom_{\Q_v}(M^*, \Gm)$ such that the induced central parameter $\phi_{\lambda}  : W_{\Q_v} \to \LM^*$ satisfies that $\psi_{M^*} \cdot \phi_{\lambda}$ agrees with $\psi$ under the $L$-embedding $\LM^* \hookrightarrow \LU^*_{\Q_v}(n)$. Choose a representative $(\varrho, z)$ in its equivalence class so that the restriction $(\varrho_{M^*}, z_{M^*})$ to $M^*$ is also an extended pure inner twist.

They then define $\Pi_{\psi}(\U, \varrho_{M^*})$ by
\begin{equation*}
    \Pi_{\psi}(\U, \varrho) := \{ I^{\U}_{P}(\pi_M \otimes \chi_{\lambda}) \mid \pi_M \in \Pi_{\psi_{M^*}}(M, \varrho_{M^*})\},
\end{equation*}
where $I^G_P$ denotes normalized parabolic induction and $\chi_{\lambda}$ is the character of $M(\Q_v)$ corresponding to $\lambda$. Note that by definition of parabolic induction, if $\pi_M$ has central character $\omega_{\pi_M}$, then $I^{\U}_P(\pi_M)$ will have central character $\delta^{\frac{1}{2}}_P \omega_{\pi_M}$. Since each element of $\Pi_{\psi_{M^*}}(M, \varrho_{M^*})$ has the same central character, this will also be true of $\Pi_{\psi}(\U, \varrho)$.

From the explicit description of $S_{\psi}$ given in \cite[pg 62]{KMSW}, it follows that $S_{\psi}=S_{\psi_{M^*}}$. In \cite[\S 1.6.4]{KMSW} they show that $S^{\rad}_{\psi}Z(\widehat{\U^*_{\Q_v}(n)})^{\Gamma_{\Q_v}} = S^{\rad}_{\psi}Z(\widehat{M^*})^{\Gamma_{\Q_v}}$ and that $\chi_z$ and $\chi_{z_{M^*}}$ both extend uniquely to give the same character $\tilde{\chi}_z$ of $S^{\rad}_{\psi}Z(\widehat{\U^*_{\Q_v}(n)})^{\Gamma_{\Q_v}}$ that is trivial on $S^{\rad}_{\psi}$. Now, we have an identification
\begin{equation*}
    \Irr(S^{\natural}_{\psi}, \chi_z) = \Irr(S^{\natural}_{\psi_{M^*}}, \chi_{z_{M^*}}),
\end{equation*}
since both parametrize irreducible representations of $S_{\psi}$ that restrict to $\tilde{\chi}_z$ on $S^{\rad}_{\psi}Z(\widehat{\U^*_{\Q_v}(n)})^{\Gamma_{\Q_v}}$. One can now define
\begin{equation*}
    \langle \pi , s \rangle  = \langle \pi_M , s_{M^*} \rangle
\end{equation*}
for $s \in S^{\natural}_{\psi}$.

It remains to verify the endoscopic character identity. Fix a refined endoscopic datum $(\mH,s, \Leta)$ for $\U^*_{\Q_v}(n)$ such that $\psi = \Leta \circ \psi^{\mH}$ for some $\psi^{\mH} \in \Psi^+(\mH)$. Then $\Leta(s) \in S_{\psi} \subset \widehat{M}$. By  \cite[Proposition 3.10]{BM2}, there exists a refined endoscopic datum $(\mH_{M^*}, s_{M^*}, \Leta_{M^*})$ and parameter $\psi^{\mH_{M^*}} \in \Psi(\mH_{M^*})$ corresponding to the pair $(\psi_{M^*}, \Leta(s))$. It is clear from construction that under the map $Y: \mc{E}^r(M^*) \to \mc{E}^r(\U^*_{\Q_v})$ of \cite[\S 2.5]{BM2}, the image of the class of $(\mH_{M^*}, s_{M^*}, \Leta_{M^*})$ equals the class of  $(\mH,s, \Leta)$. Now by \cite[Proposition 2.20]{BM2}, we can choose a refined datum equivalent to $(\mH_{M^*}, s_{M^*}, \Leta_{M^*})$ fitting into an embedded datum $(\mH, \mH_{M^*}, s, \Leta)$. We observe that $\mH_{M^*}$ is a Levi subgroup of $\mH$.

Now, $\Leta|_{M^*}$ induces a map $Z(\widehat{M^*}) \hookrightarrow Z(\widehat{\mH_{M^*}})$ and hence $\phi_{\lambda}$ yields a central parameter $\psi_{\lambda_{\mH}}$ of $\LL\mH_{M^*}$. It is easy to see that by definition $\psi^{\mH_{M^*}} \cdot \phi_{\lambda_{\mH}} = \psi^{\mH}$ under the natural inclusion $\LL\mH_{M^*} \hookrightarrow \LL\mH$. Hence, we can define a packet $\Pi_{\psi^{\mH}}(\mH)$ and pairing
\begin{equation*}
    \langle \cdot, \cdot \rangle : \Pi_{\Psi^{\mH}}(\mH) \times S^{\natural}_{\psi^{\mH}} \to \C^{\times},
\end{equation*}
using the above procedure.

We need to verify that the resulting pairing satisfy the endoscopic character identity. Let $f, f^{\mH}$ be $\Delta[\mf{w},\varrho,z]$-matching functions. Let $f_P \in \mc{H}(M^*)$ and $f^{\mH}_{P^*_{\mH}} \in \mc{H}(\mH_{M^*})$ be the corresponding constant term functions. By the paragraph at the top of page $237$ and the remark on page $239$ of \cite{vanDijk} it follows that
\begin{equation*}
    \tr(I^{\U}_P(\pi_M) \mid f) =  \tr(\pi_M \mid f_P),
\end{equation*}
and similarly for $f^{\mH}_{P^*_{\mH}}$. We can restrict the splitting of $\U^*_{\Q_v}(n)$ to $M^*$ and together with the character $\phi$, this gives a Whittaker datum $\mf{w}_{M^*}$. By \cite[Proposition 5.3]{BM2}, the corresponding canonical transfer factor $\Delta[\mf{w}_{M^*}, \varrho_{M^*}, z_{M^*}]$ satisfies
\begin{equation*}
    \Delta[\mf{w}_{M^*}, \varrho_{M^*}, z_{M^*}](\gamma_{\mH}, \gamma) = |D^{\U^*_{\Q_v}(n)}_{M^*}(\gamma)|^{-\frac{1}{2}} |D^{\mH}_{\mH_{M^*}}(\gamma^{\mH})|^{\frac{1}{2}} \Delta[\mf{w}, \varrho, z](\gamma_{\mH}, \gamma)
\end{equation*}
for regular $\gamma \in M^*(\Q_v), \gamma^{\mH} \in \mH_{M^*}(\Q_v)$ and where we recall that  $D^G_{M}(\gamma)$ is defined to equal $\mathrm{det}(1- \mathrm{ad}(m))|_{\mathrm{Lie}(G) \setminus \mathrm{Lie}(M)}$.

We now claim that $f_P$ and $f^{\mH}_{P^*_{\mH}}$ are $\Delta[\mf{w}_{M^*}, \varrho_{M^*}, z_{M^*}]$-matching. If we can show this then we will have
\begin{align*}
\sum_{\pi^{\mH} \in \Pi_{\psi^{\mH}}(\mH)} \langle \pi^{\mH}, s_{\psi^{\mH}} \rangle\tr(\pi^{\mH} \mid f^{\mH})   & = \sum_{\pi_{\mH_{M^*}} \in \Pi_{\psi^{\mH_{M^*}}}(\mH_{M^*})} \langle \pi_{\mH_{M^*}}, s_{\psi} \rangle\tr(\pi_{\mH_{M^*}} \mid f^{\mH}_{P^*_{\mH}})\\
&= e(M) \sum_{\pi_M \in \Pi_{\psi_{M^*}}(M, \varrho_{M^*})} \langle \pi_M, \Leta(s) \cdot s_{\psi_{M^*}}\rangle \tr(\pi_M \mid f_P)\\
&=e(\U) \sum_{\pi \in \Pi_{\psi}(\U, \varrho)} \langle \pi, \Leta(s) \cdot s_{\psi}\rangle \tr(\pi \mid f),
\end{align*}
as desired. Note that in the above we use that $e(M)=e(\U)$ which is part of the proposition on page $292$ of \cite{Kot4}.

Suppose $\gamma_{\mH} \in \mH_{M^*}(\Q_v)$ and $\gamma \in M(\Q_v)$ are strongly regular elements that transfer to each other. Then by \cite[Lemma 9]{vanDijk}, we have the following equality of orbital integrals (and analogously for $f^{\mH}$):
\begin{equation*}
    O^{\U}_{\gamma}(f)=|D^{\U}_M(\gamma)|^{- \frac{1}{2}} O^{M}_{\gamma}(f_P),
\end{equation*}
and hence, since $f$ and $f^{\mH}$ are $\Delta[\mf{w}, \varrho, z]$-matching:
\begin{align*}
    SO^{\mH_{M^*}}_{\gamma_{\mH}}(f^{\mH}_{P^*_{\mH}}) & = |D^{\mH}_{\mH_{M^*}}(\gamma_{\mH})|^{ \frac{1}{2}}SO^{\mH}_{\gamma_{\mH}}(f^{\mH})\\
    & = |D^{\mH}_{\mH_{M^*}}(\gamma_{\mH})|^{ \frac{1}{2}} \sum\limits_{\gamma' \sim_{st} \gamma} \Delta[\mf{w}, \varrho, z](\gamma_{\mH}, \gamma') O^{\U}_{\gamma'}(f)\\
    &  = |D^{\mH}_{\mH_{M^*}}(\gamma_{\mH})|^{ \frac{1}{2}} |D^{\U}_M(\gamma)|^{- \frac{1}{2}} \sum\limits_{\gamma' \sim_{st, \U} \gamma} \Delta[\mf{w}, \varrho, z](\gamma_{\mH}, \gamma') O^{M}_{\gamma'}(f_P)\\
    & = \sum\limits_{\gamma' \sim_{st, M} \gamma} \Delta[\mf{w}_{M^*}, \varrho_{M^*}, z_{M^*}](\gamma_{\mH}, \gamma') O^{M}_{\gamma'}(f_P),
\end{align*}
as desired. Note that we use that the number of conjugacy classes of $\gamma$ in the stable class is the same for $\U$ and $M$ (this follows from the injection $H^1(\Q_v, M) \hookrightarrow H^1(\Q_v, \U)$).

\end{proof}
\textbf{The global classification} \text{} 

We now consider the global situation. Recall that we have fixed a quadratic imaginary extension $E/\Q$ and are considering global unitary groups $\U=\U(V)$ that are quasi-split at the finite places and with fixed quasi-split inner form $\U^*(n)$. Due to the lack of global $L$-group, we rely on the cuspidal automorphic representations of $\GL_n(\mathbb{A}_E)$ to define the notion of global parameters as in \cite{ArthurBook} (cf. \cite{KMSW}). Let $\Psi(N)$ denote the set of all formal sums 
\[
\psi^n = \ell_1 (\pi_1 \boxtimes \nu_1) \boxplus \cdots \boxplus \ell_r (\pi_r \boxtimes \nu_r),
\]
where $ \ell_i $ are positive integers, $\pi_i$ are cuspidal automorphic representations of $\GL_n(\mathbb{A}_E)$ and $\nu_i$ are algebraic representations of $\SL_n(\C)$ such that $\pi_i \boxtimes \nu_i$ are pairwise disjoint.

We denote $(\pi \boxtimes \nu)^* = \pi^* \boxtimes \nu$ where $\pi^* = (\pi^c)^{\vee}$ the conjugate dual representation of $\pi$. Now for $\psi = \ell_1 (\pi_1 \boxtimes \nu_1) \boxplus \cdots \boxplus \ell_r (\pi_r \boxtimes \nu_r) $, we say that $\psi$ is self-dual if there exists an involution $ i \longmapsto i^{*} $ of $ \{ 1, \cdots , r \} $ such that $(\pi_i \boxtimes \nu_i)^{*} = \pi_{i^{*}} \boxtimes \nu_{i^{*}}$ and $\ell_i = \ell_{i^*}$. From a self-dual formal sum $ \psi^n $, we can construct a group $\mathcal{L}_{\psi^n}$ and a map (\cite[p. 22, 23; definition 2.4.3]{C.P.Mok})
\[
	\widetilde{\psi}^n : \mathcal{L}_{\psi^n} \times \SL_2(\C) \longrightarrow  \prescript{L}{}{\GL_{n,E}}. 
\]

Recall that we have a base change map $ \eta_B : \prescript{L}{}{\U^*(n)} \longrightarrow \prescript{L}{}{\GL_{n,E}} $.
\begin{definition}
The set of global $L$-parameters $ \Psi (\U^*(n), \eta_B)$ of $\U^*(n)$ is the set consisting of pairs $(\psi^n, \widetilde{\psi})$ where $ \psi^n $ is a self-dual formal sum and $\widetilde{\psi} : \mathcal{L}_{\psi^n} \times \SL_2(\C) \longrightarrow \prescript{L}{}{\U^*(n)} $ is a map such that $\widetilde{\psi^n} = \eta_B \circ \widetilde{\psi}$. 
\end{definition}	

We remark that $\widetilde{\psi}$ is determined by the base change map $\eta_B$ and $\widetilde{\psi}^n$, and %so that roughly speaking, a global $L$-parameter for $U_{E/F}$ is a self-dual formal sum of cuspidal representations of $GL(\mathbb{A}_E)$. 
as in the local case, from the map $\widetilde{\psi}$, we can define various centralizer groups $ S_{\widetilde{\psi}} $, $ \ov{S}_{\widetilde{\psi}} $, $ \overline{\mathcal{S}}_{\widetilde{\psi}} $, $ S^{\natural}_{\widetilde{\psi}} $. 
%Furthermore, there are localization maps of $L$-parameters and localization maps of centralizer groups (\textcolor{red}{REF}). \\

There is a localization morphism $\psi \longmapsto \psi_v $ from $\Psi(\U^*(n))$ to $\Psi^{+}(\U^*_{\Q_v}(n))$ \cite[p. 18,19]{C.P.Mok}. More precisely, if $v$ is a place of $\Q$ that splits in $E$ then $E_v = E_w \times E_{\overline{w}}$ and $\U^*_{\Q_v}(n)  \cong  \GL_{n, E_w}$ where $w$, $\overline{w}$ are the primes of $E$ above $v$. Moreover, there is an identification $\Q_v = E_w$ and therefore we can define $\psi_v = \psi^n_w$. If $v$ is a place of $\Q$ that does not split in $E$ then $E_v$ is a quadratic extension of $\Q_v$. By \cite[Theorem. 2.4.10]{C.P.Mok} the localization $\psi^n_v$ of $\psi^n$ factors through the base change map $\eta_B$ so that it defines a parameter $\psi_v$ in $\Psi^{+}(\U^*_{\Q_v}(n))$. 

According to Theorem \ref{itm: local} and Theorem \ref{localECIpsi+}, for each $\psi_v \in \Psi^+(\U_v)$ we have a packet $\Pi_{\psi_v}(\U_v)$ together with a map
\[
\Pi_{\psi_v} (\U_v) \longrightarrow \text{Irr}(S^{\natural}_{\psi_v}), \qquad \pi_v \mapsto \langle \pi_v, - \rangle.
\]

We denote
\[
\Pi_{\psi}(\U, \varrho) := \left\lbrace  \bigotimes_v \pi_v : \pi_v \in \Pi_{\psi_v}(\U_v, \varrho_v) \ | \ \langle\pi_v, - \rangle = 1 \quad \text{for almost all $v$} \right\rbrace .
\]

Since the localization maps $ \psi \longmapsto \psi_v $ induce the localization maps $ S^{\natural}_{\psi} \longrightarrow S^{\natural}_{\psi_v} $ for centralizer groups (\cite[pg 71]{KMSW}), we can associate to each $ \bigotimes_v \pi_v \in \Pi_{\psi}(\U, \varrho) $ a character of $S^{\natural}_{\psi}$ by the following formula
\[
\langle \pi,s \rangle := \prod_v \langle \pi_v, s_v \rangle, \quad s \in S^{\natural}_{\psi}
\]
where $s_v$ is the image of $s$ by the localization morphism $S^{\natural}_{\psi} \longrightarrow S^{\natural}_{\psi_v}$. The global pairing $\langle \pi, \cdot \rangle$ descends to a character of $\ov{\mathcal{S}}_{\psi}$ (see \cite[pg 89]{KMSW}).

\begin{definition} \phantomsection \label{itm : paquets}
	Let $ \Pi_{\psi}(\U, \varrho, \epsilon_{\psi}) := \left\lbrace \pi \in \Pi_{\psi} ( \U, \varrho ) : \langle \pi , - \rangle_{\varrho} = \epsilon_{\psi} \right\rbrace $ where $\epsilon_{\psi}$ is the Arthur character of $\ov{\mathcal{S}}_{\psi}$. If $\psi$ is a generic parameter then $\epsilon_{\psi} \equiv 1 $. 
\end{definition} 

\begin{theorem}(\cite{KMSW} theorem $ 1.7.1) $ \phantomsection \label{itm: global}
 There is an isomorphism of $\U(\mathbb{A})$-modules
\[
L^2_{\text{disc}}(\U(\Q)\setminus \U(\mathbb{A})) \simeq \bigoplus_{\psi \in \Psi(\U^*(n))} L^2_{\text{disc}, \psi}(\U(\Q) \setminus \U (\mathbb{A})).
\]

If $ \psi = \phi $ is generic then
\begin{enumerate}
\item [$\bullet$] $L^2_{\text{disc}, \phi} (\U(\Q) \setminus \U (\mathbb{A})) = 0 $ if $ \psi \notin \Psi_2 (\U^*(n)) $.
\item [$\bullet$] $L^2_{\text{disc}, \phi} (\U(\Q) \setminus \U (\mathbb{A})) \simeq \bigoplus_{\pi \in \Pi_{\phi} (\U, \epsilon_{\psi})} \pi $ if $ \psi \in \Psi_2(\U^*(n)) $.
\end{enumerate}
In particular, if $ \pi $ is an automorphic representation of the unitary group $ \U $ belonging to a generic global packet then $ m_{\pi} = 1$.
\end{theorem}
\begin{remark}
We remark that the theorem as we have stated it here is proved unconditionally in \cite{KMSW}. The careful reader will note that the theorem of \cite{KMSW} requires that $\U$ arises as a pure inner twist of $\U^*(n)$. Indeed, this will be true since we are assuming $\U$ comes from a hermitian form (see \cite[pg 18]{KMSW}).
\end{remark}
\subsection{Automorphic representations of unitary similitude groups}
In this section, we want to transfer the results about automorphic representations from unitary groups to unitary similitude groups (with an odd number of variables). We begin with the local case.

Let $v$ be a finite place of $\Q$ that does not split over $E$, let $n$ be an odd positive integer, and let $\GU$ be an inner form of $\GU^*_{\Q_v}(n)$ and denote the corresponding unitary group by $\U$. Fix a  $\Gamma_{\Q_v}$-invariant splitting of $\GU$ and restrict to get a $\Gamma_{\Q_v}$-invariant splitting of $\U$. Fix also a character $\varphi: \Q_v \to \C^{\times}$. This data gives us Whittaker data $\mf{w}_{\GU}$ and $\mf{w}_{\U}$ of $\GU$ and $\U$ respectively.

We give $\GU$ the structure of an extended pure inner twist $(\varrho_{\GU}, z_{\GU})$. We can restrict $\varrho$ to $\U$ to get an inner twist $\varrho_{\U} : \U^*_{\Q_v}(n) \to \U$. We give this the structure of an extended pure inner twist $(\varrho_{\U}, z_{\U})$. Note that since we are assuming $n$ is odd, $\GU$ will automatically be quasi-split. By Lemma \ref{itm : GU = ZxU}, we have $\GU(\Q_v) = \U(\Q_v)E_v^{\times}$ and then the following result:
\begin{corollary} \label{itm : GU = U E}
There is a natural bijection between the set $\Pi(\GU)$ and the set of pairs $(\pi, \chi)$ such that $\pi \in \Pi(\U)$ and $\chi$ is a character of $E^{\times}_v$ such that $\chi|_{E_v^{\times} \bigcap \U(\Q_v)} = \omega_{\pi}|_{E_v^{\times} \bigcap \U(\Q_v)}$ where $\omega_{\pi}$ is the central character of $\pi$.
\end{corollary}

We use this corollary to define $A$-packets of representations for $\GU$ and the associated $A$-parameters. Fix $\chi$ a character of $Z(\GU)$ corresponding to a morphism $ \tilde{\chi} : \mc{L}_{\Q_v} \longrightarrow \widehat{\GU^*_{\Q_v}(n)}_{\ab} \rtimes W_{\Q_v} = (\C^{\times} \times \C^{\times}) \rtimes W_{\Q_v} $, and a parameter $\psi \in \Psi^+(\U^*_{\Q_v}(n))$ given by $\psi: \mc{L}_{\Q_v} \times \SU(2) \to \LU = \GL_n(\C) \rtimes W_{\Q_v}$ such that $\omega_\pi|_{E_v^{\times} \bigcap \U(\Q_v)} = \chi|_{E_v^{\times} \bigcap \U(\Q_v)}$ for each $\pi \in \Pi_{\psi}(\U, \varrho_{\U})$. 

We can view $\LGU = \GL_n(\C) \times \C^{\times} \rtimes W_{\Q_v} $ as a product of $\LU = \GL_n(\C) \rtimes W_{\Q_v} $ and $\widehat{\GU^*_{\Q_v}(n)}_{\ab} \rtimes W_{\Q_v} = (\C^{\times} \times \C^{\times}) \rtimes W_{\Q_v} $ over $ \C^{\times} \rtimes W_{\Q_v} $ where the first projection is given by $ g \rtimes \alpha \longmapsto \det g \rtimes \alpha $ and the second is given by $ (x, y) \rtimes \alpha \longmapsto x \rtimes \alpha $. The above pair $ (\psi, \chi) $ then defines a unique morphism $ \Tilde{\psi} : \mc{L}_{\Q_v} \times \SU(2)  \longrightarrow \GL_n(\C) \times \C^{\times} \rtimes W_{\Q_v} $. Conversely, each  $\tilde{\psi} \in \Psi^+(\GU^*_{\Q_v}(n))$ gives rise to a pair $(\psi, \chi)$. We summarize these relationships in the following commutative diagram:

\begin{center}
 \begin{tikzpicture}[scale = 1]
\draw (0.4, 2.6) node {$\mc{L}_{\Q_v} \times \SU(2)$};
\draw (8, -1.5) node {$ \square $};
\draw (5, 0) node {$\GL_n(\C) \times \C^{\times} \rtimes W_{\Q_v}$};
\draw (11, 0) node {$\C^{\times} \times \C^{\times} \rtimes W_{\Q_v}$};
\draw (5, -3) node {$\GL_n(\C) \rtimes W_{\Q_v}$};
\draw (11, -3) node {$\C^{\times} \rtimes W_{\Q_v}$};
\draw[->] (0.4, 2.35) -- (5, 0.3) node[midway, below]{$ \Tilde{\psi} $} ;
\draw[->] (0.4, 2.35) -- (4.75, -2.5) node[midway, below]{$\psi$};
\draw[->] (0.4, 2.35) -- (10.75, 0.3) node[midway, above]{$\tilde{\chi}$};
\draw[->] (6.75, 0) -- (9.5, 0) node[midway, below]{$(\det \times \Id) \rtimes \Id$};
\draw[->] (6.45, -3) -- (10.1, -3) node[midway, above]{$\det \rtimes \Id$};
\draw[->] (5, -0.5) -- (5, -2.5) node[midway, right]{$\text{pr}_1$};
\draw[->] (11, -0.5) -- (11, -2.5) node[midway, right]{$\text{pr}_1$};
\end{tikzpicture}   
\end{center}

We define the associated $A$-packet of representations for $ \GU({\Q_v}) $ to be the set 
\begin{equation*}
    \Pi_{\tilde{\psi}}(\GU, \varrho_{\GU}) := \left\{ (\pi, \chi) | \pi \in \Pi_{\psi}(\U, \varrho), \omega_{\pi}|_{E_v^{\times} \bigcap \U(\Q_v)} = \chi|_{E_v^{\times} \bigcap \U(\Q_v)} \right\}.
\end{equation*}

We now use the internal structure of $\Pi_{\psi}(\U, \varrho_{\U})$ to describe that of $\Pi_{\tilde{\psi}}(\GU, \varrho_{\GU})$. Let us first describe the relations between the various centralizer groups for $\psi$ and $\tilde{\psi}$. 

\begin{lemma} \label{itm : centralizer of GU and U}
With $\psi$ and $\tilde{\psi}$ as above, we have
\[
 S_{\tilde{\psi}} = S_{\psi}^+ \times \C^{\times}, \qquad \overline{\mathcal{S}}_{\psi} = \overline{\mathcal{S}}_{\tilde{\psi}}, \qquad S^{\natural}_{\tilde{\psi}}= \pi_0(S^+_{\psi}) \times \C^{\times},
\]
where $ S_{\psi}^+ = \{ g \in S_{\psi} | \det g = 1 \} $.
\end{lemma}
\begin{proof}
For $(g , c)$ and $(x, t)$ in $\GL_n(\C) \times \C^{\times} $ and $\sigma \in W_{\Q_v}$ projecting to the nontrivial element of $\Gal(E_v, \Q_v)$, we have
\begin{align*}
    (g , c) \cdot (x, t) \rtimes \sigma \cdot (g^{-1} , c^{-1}) &= (gx , ct) \rtimes \sigma \cdot (g^{-1} , c^{-1}) \\
  &= \big( gx( J g^t J^{-1}) , t \det g^{-1} \big) \rtimes \sigma,
\end{align*}
where the second equality comes from the action of $\sigma$ on $(g^{-1} \times c^{-1})$. In particular, $(g, c) \in S_{\tilde{\psi}}$ if and only if $g \in S_{\psi}$ and $\det g = 1$. In other words, $S_{\psi}^+ \times \C^{\times} = S_{\tilde{\psi}}$. 

We now prove that $\overline{\mathcal{S}}_{\psi} = \overline{\mathcal{S}}_{\tilde{\psi}}$. By a direct calculation, we see that $ Z(\widehat{\U^*_{\Q_v}(n)})^{\Gamma_{\Q_v}} = \big\{ \pm \Id_n \big\} \simeq \Z/ 2\Z $ and $ Z(\widehat{\GU^*_{\Q_v}(n)})^{\Gamma_{\Q_v}} = \Id_n \times \C^{\times} $ (because $n$ odd). Hence $ \overline{\mathcal{S}}_{\tilde{\psi}} = \pi_0(S_{\psi}^+) $. We also remark that the equality $ gx( Jg^t J^{-1} ) = x $ implies $ (\det g)^2 = 1 $. Therefore, for every $ g \in S_{\psi} $ we have $ (\det g)^2 = 1 $. Moreover, since $ \det (- \Id_n) = -1 $, we have $ S^+_{\psi}  \cong S_{\psi} / Z(\widehat{\U^*_{\Q_v}(n)})^{\Gamma_{\Q_v}}$. Thus, $ \overline{\mathcal{S}}_{\tilde{\psi}} = \overline{\mathcal{S}}_{\psi} $ as desired.

Finally, we have $S^{\rad}_{\tilde{\psi}} = (S_{\tilde{\psi}} \cap (\SL_n(\C) \times 1))^{\circ} = (S^+_{\psi})^{\circ}$ which implies the description of $S^{\natural}_{\tilde{\psi}}$ in the statement of the lemma.
\end{proof}

We now construct a pairing:
\begin{equation*}
    \langle \cdot , \cdot \rangle_{\GU} : \Pi_{\tilde{\psi}}(\GU, \varrho_{\GU}) \times S^{\natural}_{\tilde{\psi}} \longrightarrow \C^{\times}. 
\end{equation*}
Let $(\pi, \chi) \in \Pi_{\tilde{\psi}}(\GU, \rho_{\GU})$. Then $\pi \in \Pi_{\psi}(\U, \varrho_{\U})$ and by Theorems \ref{itm: local} and \ref{localECIpsi+} there is an associated character $\langle \pi, \cdot \rangle_{\U} : S^{\natural}_{\psi} \longrightarrow \C^{\times}$. Note that since $S^{\rad}_{\psi} = (S^+_{\psi})^{\circ}$, we can restrict this character to $S^+_{\psi} / S^{\rad}_{\psi} $, and this restricted character factors to give a character of $\pi_0(S^+_{\psi})$. Via $z_{\GU}$ and the map 
\begin{equation*}
    \kappa: \mb{B}(\Q_v, \GU) \longrightarrow X^*(Z(\widehat{\GU})^{\Gamma_{\Q_v}})=X^*(1 \times \C^{\times}),
\end{equation*}
we get a character $\chi_{z_{\GU}}$ of $1 \times \C^{\times}$. Then in Lemma \ref{itm : centralizer of GU and U}, we showed that $S^{\natural}_{\tilde{\psi}} = \pi_0(S^+_{\psi}) \times \C^{\times}$. Hence we define
\begin{equation*}
    \langle (\pi, \chi) , (s, c) \rangle_{\GU} = \langle \pi, s \rangle_{\U} \chi_{z_{\GU}}(c). 
\end{equation*}

Suppose that $\tilde{\psi} \in \Psi(\GU^*_{\Q_v}(n))$ is generic. We show that $(\pi, \chi) \mapsto \langle (\pi, \chi), \cdot \rangle_{\GU}$ is bijective onto $\Irr( S^{\natural}_{\tilde{\psi}}, \chi_{z_{\GU}})$ by constructing an inverse. Pick a character $\mu_{\GU}$ of $S^{\natural}_{\tilde{\psi}}$ which restricts on $Z(\widehat{\GU})^{\Gamma_{\Q_v}}$ to the character $\chi_{z_{\GU}}$. Since  $\pi_0(S^+_{\psi})$ and $Z(\widehat{\U})^{\Gamma_{\Q_v}}$ generate $S^{\natural}_{\psi}$ and have trivial intersection, there is then a unique character $\mu_{\U}$ of $S^{\natural}_{\psi}$ that restricts to $\chi_{z_{\U}}$ on $Z(\widehat{\U})^{\Gamma_{\Q_v}}$ and $\mu_{\GU}$ on $\pi_0(S^+_{\psi})$. By $(2)$ of Theorem \ref{itm: local}, there then exists a $\pi \in \Pi_{\psi}(\U, \varrho_{\U})$ that gets mapped to $\mu_{\U}$, and by construction, $(\pi, \chi)$ maps to $\mu_{\GU}$. Hence $\mu_{\GU} \mapsto (\pi,  \chi)$ is our desired inverse.

We have now proven
\begin{theorem}{\label{localpairingGU}}
Parts $(1)$ and $(2)$ of Theorem \ref{itm: local} and Part $(1)$ of Theorem \ref{localECIpsi+} hold for $\GU$ for nonarchimedean $v$.
\end{theorem}

In the archimedean case, these results are known by work of Langlands and Shelstad.

In the next section, we will prove that this pairing also satisfies the endoscopic character identities.

We record the following proposition for later use.
\begin{proposition} \cite[section 8.4.4]{Moe}
Let $\phi : W_{\Q_v} \times \SU(2) \longrightarrow \prescript{L}{}{\U^*_{\Q_v}(n)}$ be a discrete $L$-parameter which is trivial over $\SU(2)$. Then the packet $\Pi_{\phi}(\U, \varrho_{\U})$ contains only supercuspidal representations. These $L$-parameters are called supercuspidal.
\end{proposition}
\begin{corollary}
From the above description of local $L$-packets of $\GU$, it follows that the $L$-packet of a supercuspidal $L$-parameter of $\GU$ will consist entirely of supercuspidal representations.
\end{corollary}
\begin{remark}
Suppose that $\phi$ is as above and $(H, s, \Leta)$ is an elliptic endoscopic datum and $\phi^{\mH} : W_{\Q_v} \times \SU(2) \longrightarrow \prescript{L}{}{\mH}$ an $L$-parameter such that $\Leta \circ \phi^{\mH} = \phi$. Then $\phi^{\mH}$ is also supercuspidal and hence the packet $\Pi_{\phi^{\mH}}(\mH)$ contains only supercuspidal representations.
\end{remark}

\subsection{The global classification for unitary similitude groups}\label{globalGUparams}

Fix a Hermitian form $V$ and global group $\U=\U(V)$ and $\GU=\GU(V)$. As in the local case, we give $\GU$ and $\U$ the structure of extended pure inner twists $(\varrho_{\GU}, z_{\GU})$ and $(\varrho_{\U}, z_{\U})$.

We begin by recalling the following result which relates automorphic representations of $\U(\mathbb{A})$ and of $\GU(\mathbb{A})$.

\begin{proposition}(\cite[Section. CHL.IV.C, Proposition 1.1.4]{CHLN}). \phantomsection \label{itm: from U to GU}
    Fix $n \in \N$ odd. Let $ \pi $ be an irreducible automorphic representation of $ \GU(\mathbb{A}) $ whose restriction to $ \U(\mathbb{A}) $ contains an irreducible automorphic representation $ \sigma $. If $ \sigma $ has multiplicity $1$ in the discrete spectrum of $ \U(\mathbb{A}) $ then $ \pi $ has multiplicity $1$ in the discrete spectrum of $ \GU(\mathbb{A}) $. Moreover, $\pi$ is the unique automorphic representation of $\GU(\A)$ with central character $\chi$ and containing $\sigma$ in its restriction.
\end{proposition}

Let $\chi$ be an automorphic central character of $\GU(\mathbb{A})$ and $\chi_{\U} := \chi|_{ Z(\U(\mathbb{A}))}$ its restriction to the center of $\U(\mathbb{A})$. Consider $\Dot{\psi}$ a generic $A$-parameter for a global unitary group whose automorphic representations have $\chi_{\U}$ as central character. The generic condition ensures the multiplicity one property of these automorphic representations by Theorem \ref{itm: global}. As in the local case, a pair $(\Dot{\psi}, \chi)$ satisfying the above conditions determines a generic $A$-parameter for $\GU$. In the following, we will denote such an $A$-parameter by $\widetilde{\Dot{\psi}}$ if $\chi$ is clear from the context. We define the associated $A$-packet $\Pi_{\widetilde{\Dot{\psi}}}(\GU, \varrho_{\GU}, \epsilon_{\widetilde{\Dot{\psi}}})$ to consist of the $\pi$ whose central character is $\chi$ and whose restriction to $\U(\mathbb{A})$ belongs to $\Pi_{\Dot{\psi}}(\U, \varrho_{\U}, \epsilon_{\Dot{\psi}})$. 

Now, by the proof of \cite[Section. CHL.IV.C, Proposition 1.1.4]{CHLN}, we have
\[
L^2( \Sigma \setminus \U(\mathbb{A}), \chi_{\U}) = \Res^{\GU(\mathbb{A})}_{\U(\mathbb{A})} L^2( \Gamma \setminus \GU(\mathbb{A}), \chi),
\]
where $\Gamma = \GU(\Q)Z(\GU)(\mathbb{A})$ and $ \Sigma = \Gamma \cap \U(\mathbb{A}) $. In particular, it follows from Theorem \ref{itm: global} that we can lift every representation $\sigma \in \Pi_{\Dot{\psi}}(\U, \varrho_{\U}, \epsilon_{\Dot{\psi}})$ to a representation of $\GU(\mathbb{A})$ whose central character is $\chi$. Combining with Proposition \ref{itm: from U to GU}, we see that there is a bijection between $\Pi_{\Dot{\psi}}(\U, \varrho_{\U}, \epsilon_{\Dot{\psi}})$ and $\Pi_{\widetilde{\Dot{\psi}}}(\GU, \varrho_{\GU}, \epsilon_{\widetilde{\Dot{\psi}}})$.

We now give a description of $\Pi_{\widetilde{\Dot{\psi}}}(\GU, \varrho_{\GU}, \epsilon_{\widetilde{\Dot{\psi}}})$ in the spirit of Definition \ref{itm : paquets}. We have defined global generic $A$-parameters of $\GU$ in terms of their counterpart for $\U$. We define the centralizer groups for such parameters of $\GU$ using the analogous groups for $\U$ and using Lemma \ref{itm : centralizer of GU and U} as our guide.
\begin{remark}
It would perhaps be possible to define these parameters and their centralizer groups in analogy with our definitions for $\U$ using cuspidal automorphic representations and the methods of \cite{ArthurBook}, \cite{C.P.Mok}, and \cite{KMSW}. For simplicity, we choose not to do this in the present paper. 
\end{remark}
\begin{definition}{\label{globalGUcentralizers}}
    Let $\widetilde{\Dot{\psi}}=(\Dot{\psi}, \chi) \in \Psi(\GU^*(n))$ be a generic parameter. We define:
    \[
 S_{\widetilde{\Dot{\psi}}} := S_{\Dot{\psi}}^+ \times \C^{\times}, \qquad \overline{\mathcal{S}}_{\widetilde{\Dot{\psi}}} := \overline{\mathcal{S}}_{\Dot{\psi}}, \qquad S^{\natural}_{\widetilde{\Dot{\psi}}} := \pi_0(S^+_{\Dot{\psi}}) \times \C^{\times}
\]
    
\end{definition}

We now discuss localization. First, by the localization map for algebraic cocycles (see \cite[\S 7]{Kot9}), the extended pure inner twists $(\varrho_{\GU}, z_{\GU})$ and $(\varrho_{\U}, z_{\U})$ give rise to local extended pure inner twists $(\varrho_{\GU_{\Q_v}}, z_{\GU_{\Q_v}})$ and $(\varrho_{U_{\Q_v}}, z_{\U_{\Q_v}})$ for each place $v$ of $\Q$.

Let $(\Dot{\psi}, \chi)$ be a generic $A$-parameter. At each place $v$ of $\Q$, we get a local parameter $\Dot{\psi}_v$ as well as a local character $\chi_v$. We define the localization of $(\Dot{\psi}, \chi)$ at $v$ to be $(\Dot{\psi}_v, \chi_v)$. The localization map 
$S^{\natural}_{\Dot{\psi}} \longrightarrow S^{\natural}_{\Dot{\psi}_v}$ restricts to give a map $S^+_{\Dot{\psi}} \to S^+_{\Dot{\psi}_v}$ and hence we get a localization map
\begin{equation*}
    S^{\natural}_{\widetilde{\Dot{\psi}}} \longrightarrow S^{\natural}_{\widetilde{\Dot{\psi}}_v}.
\end{equation*}
Similarly, we get a localization map $\ov{\mc{S}}_{\widetilde{\Dot{\psi}}} \longrightarrow \ov{\mc{S}}_{\widetilde{\Dot{\psi}}_v}$.

We now define:
\[
\Pi_{\widetilde{\Dot{\psi}}}(\GU, \varrho_{\GU}) := \left\lbrace  \bigotimes_v \pi_v : \pi_v \in \Pi_{\widetilde{\Dot{\psi}}_v}(\GU_{\Q_v}, \varrho_{\GU_{\Q_v}}) \mid \langle\pi_v, - \rangle_{\GU_{\Q_v}} = 1 \quad \text{for almost all $v$} \right\rbrace .
\]

We associate to each $ \pi = \bigotimes_v \pi_v \in \Pi_{\widetilde{\Dot{\psi}}}(\GU, \varrho_{\GU}) $ a character of $S^{\natural}_{\widetilde{\Dot{\psi}}}$. Each $\pi_v$ corresponds to a pair $(\pi'_v, \chi_v)$ where $\pi'_v \in \Pi(\U)$. We then define a global pairing by the formula
\[
\langle  \pi, (s,c) \rangle_{\GU} := \prod_v \langle (\pi'_v, \chi_v), (s_v, c_v) \rangle_{\GU_{\Q_v}}, \quad (s,c) \in S^{\natural}_{\widetilde{\Dot{\phi}}},
\]
where $(s_v, c_v)$ is the image of $(s,c)$ under the localization map defined above. We claim that $\langle \pi, \cdot \rangle$ descends to a character on $\ov{\mc{S}}_{\widetilde{\Dot{\psi}}}$. Indeed, by definition we have
\begin{equation*}
    \langle \pi, (s, c) \rangle_{\GU} = \prod_v \langle \pi'_v, s_v \rangle_{\U_{\Q_v}} \chi_{z_{\GU_{\Q_v}}}(c_v).
\end{equation*}
We showed previously that $\prod\limits_v \langle \pi'_v, s_v \rangle_{\U_{\Q_v}}$ descends to $\ov{\mc{S}}_{\widetilde{\Dot{\psi}}} = \ov{\mc{S}}_{\Dot{\psi}}$ and $\prod\limits_v \chi_{z_{\GU_{\Q_v}}}(c_v)$ is trivial by \cite[Proposition 15.6]{Kot9}.

\begin{proposition} \label{global multiplicity formula}
    For $\widetilde{\Dot{\psi}}$ a generic $A$-parameter of $\GU$, we have the following equality of sets:
\[
\Pi_{\widetilde{\Dot{\psi}}}(\GU, \varrho_{\GU}, \epsilon_{\widetilde{\Dot{\psi}}}) = \left\lbrace \pi \in \Pi_{\widetilde{\Dot{\psi}}}(\GU, \varrho_{\GU}) \mid \langle \pi, \cdot \rangle_{\GU} \equiv \epsilon_{\widetilde{\Dot{\psi}}} \right\rbrace.
\]
We note that since we are assuming $\widetilde{\Dot{\psi}}$ is generic, we in fact have $\epsilon_{\widetilde{\Dot{\psi}}}=1$.
\end{proposition}
\begin{proof}
The left-hand side consists of all pairs $(\pi, \chi)$ such that $\pi \in \Pi_{\Dot{\psi}}(\U, \varrho_{\U}, \epsilon_{\Dot{\psi}})$. By definition, we have
\begin{equation*}
    \Pi_{\Dot{\psi}}(\U, \varrho_{\U}, \epsilon_{\Dot{\psi}}) = \left\lbrace \pi \in \Pi_{\Dot{\psi}}(\U, \varrho_{\U}) \mid \langle \pi, \cdot \rangle_{\U} \equiv \epsilon_{\Dot{\psi}} \right\rbrace.
\end{equation*}
Hence we just need to show that $\langle \pi, \cdot \rangle_{\U}$ is trivial if and only if $\langle (\pi, \chi), \cdot \rangle_{\GU}$ is. But this is clear since these are the same character of $\ov{\mc{S}}_{\widetilde{\Dot{\psi}}} = \ov{\mc{S}}_{\Dot{\psi}}$.
\end{proof}

\begin{remark}{\label{GU to products}}
For our purposes, we also need to generalise the above description of automorphic representations to the groups $\G( \U_1 \times \cdots \times \U_k )(\A)$ with $n_1 + \cdots + n_k = n$ odd. In this case, Proposition \ref{itm: from U to GU} still holds true (\cite[Section. CHL.IV.C, Proposition 1.3.5]{CHLN}) and then the above process can be applied without any major change.
\end{remark}

\section{Endoscopic character identities}{\label{sectionECIGU}}
Let $E/\Q_p$ be a quadratic extension and $n$ an odd natural number. Our goal in this section is to prove the endoscopic character identities for elliptic endoscopic groups of $\G(\U(n_1) \times \cdots \times \U(n_r))$ with $ n_1 + \cdots + n_r = n$ and $\U(n_i)$ is an inner form of $\U^*_{\Q_p}(n_i)$. We prove this using the fact that these identities hold for $\U(n_1) \times \cdots \times \U(n_r)$ as in \cite{C.P.Mok}, \cite{KMSW}. 
\begin{itemize}
    \item We first show that if the endoscopic character identities hold for $\U(n_1) \times \cdots \times \U(n_r)$, then they also hold for $\U(n_1) \times \cdots \times \U(n_r) \times \Res_{E/\Q_p}\Gm$ where we note that $\Res_{E/\Q_p}\Gm$ is the center of $\G(\U(n_1) \times \cdots \times \U(n_r))$.\\
    \item We then show that if the endoscopic character identities hold for $\U(n_1) \times \cdots \times \U(n_r) \times \Res_{E/\Q_p}\Gm$, then they hold for $\G(\U(n_1) \times \cdots \times \U(n_r))$.\\
\end{itemize}

We recall the statement of the endoscopic character identity for an extended pure inner twist $\G, (\varrho, z)$ of a quasi-split reductive group $\G^*$ over $\Q_p$ with refined endoscopic datum $(\mH, s, \LL \eta)$. Fix a local Whittaker datum $\mf{w}$ of $\G^*$ giving a Whittaker normalized transfer factor $\Delta[\mf{w}, \varrho, z]$ (as in \cite[\S 4.3]{KalTai}) between $(\mH, s, \LL \eta)$ and $\G$. Suppose that $f \in \mathcal{H}(\G)$ and $f^{\mH} \in \mathcal{H}(\mH)$ are $\Delta[\mf{w}, \varrho, z]$-matching functions. 

Now, let $\psi \in \Psi^+(\G^*)$ and  $\psi^{\mH} \in \Psi^+(\mH)$ be such that $\psi =  \Leta \circ \psi^{\mH}$. Let $\Pi_{\psi^{\mH}}(\mH), \Pi_{\psi}(\G, \varrho)$ denote the respective $A$-packets for the parameters. Then the endoscopic character identity states that 
\begin{equation}
    \sum\limits_{\pi^{\mH} \in \Pi_{\psi^{\mH}}(\mH)} \langle \pi^{\mH}, s_{\psi^{\mH}}\rangle \tr( \pi^{\mH} \mid  f^{\mH})= e(\G)\sum\limits_{\pi \in \Pi_{\psi}} \langle \pi, s \cdot s_{\psi} \rangle \tr( \pi \mid f),
\end{equation}
where $\langle \pi, s \rangle$ is as defined in Theorem \ref{itm: local} and Theorem \ref{localECIpsi+}.  The elements $s_{\psi}$ and $s_{\psi^{\mH}}$ are defined to be the image of $(1, -I)$ under $\psi$ and $\psi^{\mH}$ respectively and $e(G)$ is the Kottwitz sign.

According to a theorem of Harish-Chandra, the trace distribution $f \mapsto \tr( \pi \mid f)$ is given by integrating against the Harish-Chandra character, which is a locally constant function $\mathcal{F}_{\pi}$ of $\G(\Qp)_{sr}$ (where $\G(\Qp)_{sr}$ denotes the strongly regular semisimple elements of $\G(\Qp)$). Then the above identity is equivalent to the equality
\begin{align*}
      &\int\limits_{\mH(\Qp)_{sr}}  \sum\limits_{\pi^{\mH} \in \Pi_{\psi^{\mH}}(\mH)} \langle \pi^{\mH}, s_{\psi^{\mH}} \rangle f^{\mH}(g)\mathcal{F}_{\pi^{\mH}}(g)dg\\
      =&e(\G)\int\limits_{\G(\Qp)_{sr}}  \sum\limits_{\pi \in \Pi_{\psi}(\G, \varrho)} \langle \pi, s \cdot s_{\psi} \rangle f(g)\mathcal{F}_{\pi}(g)dg.
\end{align*}
We remark that a Harish-Chandra character exists for representations $I^G_P(\pi)$ by \cite[Theorem 3]{vanDijk} and that this holds even in the case where the induction is not irreducible. Hence, $\pi \in \Pi_{\psi}(\G, \varrho)$ have Harish-Chandra characters even in the case where $\psi \in \Psi^+(\G^*)$.
\subsection{Endoscopic Identities for \texorpdfstring{$\U(n_1) \times \cdots \times \U(n_r) \times \Res_{E/ \Q_p}\Gm$}{U(n) x ResGm}}{\label{ECIproduct}}

In this section we use the notation $\U$ to denote the group $\U(n_1) \times \cdots \times \U(n_r)$. Our goal is to prove the endoscopic character identities for $\U \times \Res_{E/\Q_p}\Gm$ using the fact that these identities are known for $\U$ by \cite{KMSW} (Theorems \ref{itm: local} and \ref{localECIpsi+} in this paper). 

In fact, we will prove the following stronger result. Fix quasi-split reductive groups $\G^*_i$ for $i \in \{ 1, 2\}$. Let $\G_i, (\varrho_i, z_i)$ be extended pure inner twists of $\G^*_i$. Let $(\mH_i, s_i, \LL \eta_i)$ be refined endoscopic for $\G_i$. We denote by $(\mH_1 \times \mH_2, s_1 \times s_2 , \LL \eta_1 \times \LL \eta_2)$ the corresponding endoscopic datum of $\G_1 \times \G_2$. Fix a character $\varphi: \Q_p \to \C^{\times}$ and $\Q_p$-splittings of $\G^*_i$. This induces a Whittaker datum $\mf{w}_i$ of $\G^*_i$ as well as the Whittaker datum $\mf{w}_1 \times \mf{w}_2$ of $\G^*_1 \times \G^*_2$.

We will prove that if the endoscopic character identities are satisfied for $\G_i $ and $ (\mH_i, s_i, \LL \eta_i) $ then they are also satisfied for $\G_1 \times \G_2$ and $ (\mH_1 \times \mH_2, s_1 \times s_2, \LL \eta_1 \times \LL \eta_2) $.

Fix $\psi^{\G_1 \times \G_2} \in \Psi^+(\G^*_1 \times \G^*_2)$ and suppose $\psi^{\mH_1 \times \mH_2} \in \Psi^+(\mH_1 \times \mH_2)$ is such that $\psi^{\G_1 \times \G_2} = (\LL\eta_1 \times \LL \eta_2) \circ \psi^{\mH_1 \times \mH_2}$. Then $\psi^{\mH_1 \times \mH_2}$ factors as a product of parameters $\psi^{\mH_1}$ of $\mH_1$ and $\psi^{\mH_2}$ of $\mH_2$. As a result, $\psi^{\G_1 \times \G_2}$ factors as a product of parameters $\psi^{\G_1}$ of $\G_1$ and $\psi^{\G_2}$ of $\G_2$ such that $\psi^{G_i} = \LL\eta_i \circ \psi^{\mH_i}$.

%We now consider an endoscopic datum $(\mH, s, \LL \eta)$ of $\U$ (where $ \mH := \displaystyle \prod_{i = 1}^r  \U_F(N^+_i) \times \U_F(N^-_i)$ for $N^+_i + N^-_i=N_i$).  We let $(\mH \times \Res\Gm, (s,1), \LL \eta \times \LL \eta_{\Gm})$ be the corresponding endoscopic group of $\U_F(N) \times \Res\Gm$. Let $\psi^{\mH \times \Res\Gm}$ be a parameter of $\mH \times \Res\Gm$ and $\psi^{\U \times \Res\Gm}:= (\LL\eta \times \LL \eta_{\Gm}) \circ \psi^{\mH \times \Res\Gm}$. Then $\psi^{\mH \times \Res\Gm}$ factors as a product of parameters $\psi^{\mH}$ of $\mH$ and $\psi^{\Res\Gm}$ of $\Gm$. As a result, $\psi^{\U \times \Res\Gm}$ factors as a product of parameters $\psi^{\U}$ of $\U$ and $\psi^{\Res\Gm}$ of $\Gm$ such that $\psi^{\U} = \LL\eta \circ \psi^{\mH}$.

We need to show that for $\Delta[\mf{w}_1 \times \mf{w}_2, \varrho_1 \times \varrho_2, z_1 \times z_2]$-matching functions $f \in \mathcal{H}(\G_1 \times \G_2)$ and $f' \in \mathcal{H}(\mH_1\times\mH_2)$, the following identity holds:
\begin{equation} \phantomsection \label{itm : Gm x U}
      \int\limits_{(\mH_1\times\mH_2)(\Q_p)_{sr}} \sum\limits_{\pi' \in \Pi_{\psi^{\mH_1\times\mH_2}}} \langle \pi', s_{\psi^{\mH_1 \times \mH_2}} \rangle f'(g)\mathcal{F}_{\pi'}(g)dg
\end{equation}
\begin{equation*}
      = e(\G_1 \times \G_2) \int\limits_{ (\G_1 \times \G_2)(\Q_p)_{sr}}  \sum\limits_{\pi \in \Pi_{\psi^{\G_1 \times \G_2}}} \langle \pi, s \cdot s_{\psi^{\G_1 \times \G_2}} \rangle f(g)\mathcal{F}_{\pi}(g)dg.
\end{equation*}

The packets $\Pi_{\psi^{\mH_1\times\mH_2}}(\mH_1 \times \mH_2)$ resp. $\Pi_{\psi^{\G_1 \times \G_2}}(\G_1 \times \G_2, \varrho_1 \times \varrho_2)$ consist of representations of the form $ \pi^{\mH_1} \boxtimes \pi^{\mH_2} $ resp. $ \pi^{\G_1} \boxtimes \pi^{\G_2} $ where $\pi^{\mH_i}$ resp. $\pi^{\G_i}$ are representations in $\Pi_{\psi^{\mH_i}}(\mH_i)$ resp. $\Pi_{\psi^{\G_i}}(\G_i, \varrho_i)$. The pairings $\langle \pi^{\mH_1} \boxtimes \pi^{\mH_2} , \bullet \rangle$ resp. $\langle \pi^{\G_1} \boxtimes \pi^{\G_2} , \bullet \rangle$ are defined as $\langle \pi^{\mH_1} , \bullet \rangle \cdot \langle \pi^{\mH_2} , \bullet \rangle $ resp. $\langle \pi^{\G_1} , \bullet \rangle \cdot \langle \pi^{\G_2} , \bullet \rangle $. It is not difficult to see that $ \mathcal{F}_{\pi^{\G_1} \boxtimes \pi^{\G_2}} = \mathcal{F}_{\pi^{\G_1}} \boxtimes \mathcal{F}_{\pi^{\G_2}} $. It is also a basic property of the Kottwitz sign that $e(\G_1 \times \G_2)=e(\G_1)e(\G_2)$.

Moreover, a function $f \in \mathcal{H}(\G_1 \times \G_2)$ can be written as a sum of functions of the form $ f_1 \boxtimes f_2 $ where $f_1 \in \mathcal{H}(\G_1)$ and $f_2 \in \mathcal{H}(\G_2)$.
 Hence, for every such $ f_1 \boxtimes f_2 $ we have an equality between the following quantities
\[
e(\G_1 \times \G_2) \int\limits_{ (\G_1 \times \G_2)(\Q_p)_{sr} }  \sum\limits_{\pi \in \Pi_{\psi^{\G_1 \times \G_2}}} \langle \pi, s \cdot s_{\psi^{\G_1 \times \G_2}} \rangle (f_1 \boxtimes f_2)(x)\mathcal{F}_{\pi}(x)dx,
\]
and
\[
e(\G_1)\int\limits_{\G_1(\Q_p)_{sr}} \sum\limits_{\pi^{\G_1} \in \Pi_{\psi^{\G_1}}} \langle \pi^{\G_1}, s \cdot s_{\psi^{\G_1}} \rangle f_1 (x) \mathcal{F}_{\pi^{\G_1}}(x)dx
\]
\[
 \cdot e(\G_2)\int\limits_{\G_2(\Q_p)_{sr}} \sum\limits_{\pi^{\G_2} \in \Pi_{\psi^{\G_2}}} \langle \pi^{\G_2}, s \cdot s_{\psi^{\G_2}} \rangle f_2 (y) \mathcal{F}_{\pi^{\G_2}} (y)dy .
\]

Similarly, for every $ f^{\mH_1}_1 \boxtimes f^{\mH_2}_2 $ with $f^{\mH_1}_1 \in \mathcal{H}(\mH_1)$ a matching function of $f_1$ and $ f^{\mH_2}_2 \in \mathcal{H}(\mH_2) $ a matching function of $f_2$ we have an equality between 
\[
\int\limits_{(\mH_1 \times \mH_2)(\Q_p)_{sr}} \sum\limits_{\pi' \in \Pi_{\psi^{\mH_1 \times \mH_2}}} \langle \pi', s_{\psi^{\mH_1 \times \mH_2}} \rangle (f^{\mH_1}_1 \boxtimes f^{\mH_2}_2) (x)\mathcal{F}_{\pi'}(x)dx
\]
and
\[
\int\limits_{\mH_1(\Q_p)_{sr}}\sum\limits_{\pi^{\mH_1} \in \Pi_{\psi^{\mH_1}}} \langle \pi^{\mH_1}, s_{\psi^{\mH_1}} \rangle f^{\mH_1}_1 (x) \mathcal{F}_{\pi^{\mH_1}}(x)dx \int\limits_{\mH_2(\Q_p)_{sr}} \sum\limits_{\pi^{\mH_2} \in \Pi_{\psi^{\mH_2}}} \langle \pi^{\mH_2}, s_{\psi^{\mH_2}} \rangle f_2^{\mH_2} (y) \mathcal{F}_{\pi^{\mH_2}} (y)dy  .
\]

Therefore, in order to prove the equation (\ref{itm : Gm x U}), it suffices to prove that for each $f_1 \boxtimes f_2 \in \mc{H}(\G_1 \times \G_2)$, we may choose a $\Delta[\mf{w}_1 \times \mf{w}_2, \varrho_1 \times \varrho_2, z_1 \times z_1]$-matching function $f^{\mH_1}_1 \boxtimes f^{\mH_2}_2 \in \mathcal{H}(\mH_1 \times \mH_2) $ such that $ f^{\mH}_i \in \mathcal{H}(\mH_i) $ and $f_i \in \mathcal{H}(\G_i)$ are $\Delta[\mf{w}_i, \varrho_i, z_i]$-matching.  This follows from the following lemma.

\begin{lemma}{\label{prodmatch}}
If $ f^{\mH_i}_i \in \mathcal{H}(\mH_i) $ and $f_i \in \mathcal{H}(\G_i)$ are $\Delta[\mf{w}_i, \varrho_i, z_i]$-matching functions then $ f^{\mH_1}_1 \boxtimes f^{\mH_2}_2  \in \mathcal{H}(\mH_1 \times \mH_2) $ and $ f_1 \boxtimes f_2 \in \mathcal{H}( \G_1 \times \G_2) $ are $\Delta[\mf{w}_1 \times \mf{w}_2, \varrho_1 \times \varrho_2, z_1 \times z_2]$-matching functions.
\end{lemma}
\begin{proof}
Pick $\gamma_{\mH} = ( \gamma_{\mH_1}, \gamma_{\mH_2} ) \in (\mH_1 \times \mH_2)(\Q_p)$ such that $\gamma_{\mH}$ is strongly regular and transfers to a strongly regular $\gamma = (\gamma_1, \gamma_2) \in (\G_1 \times \G_2)(\Q_p)$. Then we need to show that
 \begin{equation} \phantomsection \label{itm : need to prove that}
     SO^{\mH_1 \times \mH_2}_{\gamma_{\mH}}(f^{\mH_1}_1 \boxtimes f^{\mH_2}_2) = \sum\limits_{\gamma' \sim_{st} \gamma} \Delta[\mf{w}_1 \times \mf{w}_2, \varrho_1 \times \varrho_2, z_1 \times z_1](\gamma_{\mH}, \gamma') O^{\G_1 \times \G_2}_{\gamma'}(f_1 \boxtimes f_2).
 \end{equation}
where the sum is taken over the set of $ \gamma' $ that are stably conjugate to $ \gamma $. 
 
By definition, for $ \gamma_i \in \G_i$ and $f_i \in C_c^{\infty}(\G_i)$ we have
\[
O^{\G_1 \times \G_2}_{\gamma_1 \times \gamma_2} (f_1 \boxtimes f_2) = O^{\G_1}_{\gamma_1} (f_1) O^{\G_2}_{\gamma_2} (f_2).
\]

Moreover, an element $(\gamma'_1, \gamma'_2) \in \mH_1 \times \mH_2 $ is stable conjugate to $(\gamma_1, \gamma_2) \in \mH_1 \times \mH_2 $ if and only if $ \gamma'_1 $ is stable conjugate to $ \gamma_1 $ in $ \mH_1 $ and $ \gamma'_2 $ is stable conjugate in $ \gamma_2 $ in $\mH_2$. Therefore we have 
\[
SO^{\mH_1 \times \mH_2}_{\gamma_{\mH}} (f^{\mH_1}_1 \boxtimes f^{\mH_2}_2) = SO^{\mH_1}_{\gamma_{\mH_1}} (f^{\mH_1}_1) SO^{\mH_2}_{\gamma_{\mH_2}} (f^{\mH_2}_2).
\]
and similarly
\begin{equation} \phantomsection \label{itm : matching functions}
 \sum\limits_{\gamma' \sim_{st} \gamma} \Delta[\mf{w}_1 \times \mf{w}_2, \varrho_1 \times \varrho_2, z_1 \times z_1](\gamma_{\mH}, \gamma') O^{\G_1 \times \G_2}_{\gamma'}(f_1 \boxtimes f_2)
 \end{equation}
 \begin{equation*}
      = \sum\limits_{\gamma_1' \sim_{st} \gamma_1 , \gamma_2' \sim_{st} \gamma_2} \Delta[\mf{w}_1 \times \mf{w}_2, \varrho_1 \times \varrho_2, z_1 \times z_1](\gamma_{\mH}, \gamma') O^{\G_1}_{\gamma_1} (f_1) O^{\G_2}_{\gamma_2} (f_2).  
 \end{equation*}

We will prove in Lemma \ref{itm : product of transfer factor} that 
\begin{equation*}
\Delta[\mf{w}_1 \times \mf{w}_2, \varrho_1 \times \varrho_2, z_1 \times z_1](\gamma_{\mH}, \gamma') = \Delta[\mf{w}_1, \varrho_1, z_1](\gamma_{\mH_1}, \gamma_1') \Delta[\mf{w}_2, \varrho_2, z_2](\gamma_{\mH_2}, \gamma_2').
\end{equation*}
We can then rewrite the right hand side of (\ref{itm : matching functions}) as
\[
 \big(\sum\limits_{\gamma_1' \sim_{st} \gamma_1} \Delta[\mf{w}_1, \varrho_1, z_1](\gamma_{\mH_1}, \gamma_1') O^{\G_1}_{\gamma_1} (f_1) \big) \big(\sum\limits_{\gamma_2' \sim_{st} \gamma_2} \Delta[\mf{w}_2, \varrho_2, z_2](\gamma_{\mH_2}, \gamma_2') O^{\G_2}_{\gamma_2} (f_2) \big),
\]
and because $ f_i^{\mH_i} $ and $ f_i $ are $\Delta[\mf{w}_i, \varrho_i, z_i]$ matching functions, this is exactly 
\[
SO^{\mH_1}_{\gamma_{\mH_1}} (f^{\mH}_1) SO^{\mH_2}_{\gamma_{\mH_2}} (f^{\mH}_2).
\]

In other words, Equation (\ref{itm : need to prove that}) is true.
\end{proof}

\subsection{Endoscopic Identities for \texorpdfstring{$\G(\U(n_1) \times ... \times \U(n_r))$}{\GU(N)}}{\label{ECIforGU}} \textbf{}

We now have the endoscopic character identities for $\U \times \Res_{E/ \Q_p}\Gm$ and need to show they also hold for $\GU$ where we use the letter $\GU$ to denote the group $\G(\U(n_1) \times \cdots \times \U(n_r))$  till the end of this section. Recall that we have a surjection of algebraic groups
\begin{equation}{\label{P}}
    P: \U \times \Res_{E/ \Q_p}\Gm \twoheadrightarrow \GU,
\end{equation}
with kernel $\U(1)$.

We fix quasi-split groups $\U^* \times \Res_{E/\Q_p} \Gm$ and $\GU^*$ as well as extended pure inner twists $(\varrho_{\U}, z_{\U})$ and $(\varrho_{\GU}, z_{\GU})$ for $\U \times \Res_{E/\Q_p} \Gm$ and $\GU$ respectively. The extended pure inner twist $(\varrho_{\U}, z_{\U})$ restricts to give extended pure inner twists $(\varrho'_{\U}, z'_{\U})$ and $(\varrho_{\Gm}, z_{\Gm})$ of $\U$ and $\Gm$ respectively.  We can choose these extended pure inner twists such that the projection $P$ takes $\varrho_{\U}$ to $\varrho_{\GU}$ and such that $z_{\U}$ and $z_{\GU}$ coincide under the map $\mb{B}(\Q_p, \U^* \times \Res_{E/ \Q_p} \Gm) \longrightarrow \mb{B}(\Q_p, \GU^*)$. We fix compatible $\Gamma_{\Q_p}$-splittings of these groups as well as a character $\varphi: \Q_p \to \C^{\times}$. Hence we get compatible Whittaker data which we denote by $\mf{w}_{\U}$ and $\mf{w}_{\GU}$ respectively.

A crucial input in the case we consider (where $n = n_1 + \cdots + n_r$ is odd) is that the projection $P$ is also a surjection on $\Qp$-points. This follows from Lemma \ref{itm : GU = ZxU}. Hence we get a map 
\begin{equation*}
    \Irr(\GU(\Qp)) \hookrightarrow \Irr((\U \times \Res_{E/ \Q_p}\Gm)(\Qp)),
\end{equation*}
given by pullback. The image of this map is the set of irreducible representations $\pi \boxtimes \chi$ such $\pi|_{\U(1)(\Q_p)}= \chi|_{\U(1)(\Q_p)}$. If this is satisfied by a single member of an $A$-packet of $\U \times \Res_{E/ \Q_p}\Gm$, then it will be satisfied by the entire packet since elements of an $A$-packet have the same central character (\cite[Theorem 1.6.1]{KMSW} and Theorem \ref{localECIpsi+}). In light of Theorem \ref{localpairingGU}, the $A$-packets of $\GU$ are in a natural way a subset of the $A$-packets of $\U \times \Res_{E/ \Q_p}\Gm$.

Since the kernel of $P$ is compact, any $f \in \mathcal{H}(\GU)$ lifts to an element $f' \in \mathcal{H}(\U \times \Res_{E/\Q_p}\Gm)$. Suppose $\pi$ is an admissible representation of $\GU(\Q_p)$ and $\pi'$ is the lift to $\Irr(\U \times \Res_{E/ \Q_p}\Gm)$. Then to prove the endoscopic character identities for $\GU$ it will be necessary to relate $\tr(\pi \mid f)$ and $\tr(\pi' \mid f')$. We have 
\begin{align*}
    \pi'(f')v & =\int\limits_{(\U \times \Res_{E/ \Q_p}\Gm)(\Q_p)}f'(g)\pi'(g)vdg=\int\limits_{\GU(\Q_p)} f(g) \pi(g)v dg \int\limits_{\U_F(1)(\Q_p)} dz\\
    &=\Vol(\U(1)(\Q_p))\pi(f)v,  
\end{align*}
where the middle equality holds by \cite[(3.21)]{PR}.

Analogously in the endoscopic case, we have a map 
\begin{equation}{\label{P^H}}
P^{\mH}: \mH \times \Res_{E/ \Q_p}\Gm \twoheadrightarrow \G(\mH),
\end{equation}
with kernel $\U(1)$ where $\mH = \displaystyle \prod_{i = 1}^r \U(n_i^+) \times \U(n_i^-) $ such that $ n_i = n_i^+ + n_i^- $ is an endoscopic group of $\U$ and $\G(\mH)$ is the associated similitude group. Suppose $n = n_1 + \cdots + n_r$ is odd. By Lemma \ref{itm : GU = ZxU}, the map is a surjection on $\Qp$-points.

We fix a refined endoscopic datum $(\G(\mH), s, \Leta)$ for $\GU$ as in Section $1$. The map $\Res_{E/ \Q_p} \Gm \cong Z(\GU) \hookrightarrow \GU$ induces a map of $L$-groups $\LL\GU \longrightarrow \LL(\Res_{E/ \Q_p} \Gm)$. We get an analogous map for $\G(\mH)$ and since $\widehat{\Res_{E/ \Q_p} \Gm}$ is the abelianization of $\widehat{\G(\mH)}$, we get a commutative diagram
\begin{equation*}
\begin{tikzcd}
\LL\G(\mH) \arrow[r, "\Leta"] \arrow[d] & \LL\GU \arrow[d]  \\
 \LL(\Res_{E/ \Q_p} \Gm) \arrow[r, "\LL \lambda "] & \LL(\Res_{E/ \Q_p} \Gm).
\end{tikzcd}    
\end{equation*}

We now fix an endoscopic datum of $\U \times \Res_{E/ \Q_p} \Gm$ which we denote by $(\mH \times \Res_{E/ \Q_p}, s', \Leta')$. We fix $\Leta'$ such that the restriction to $\widehat{\mH}$ induces an elliptic endoscopic datum for $\U$ as in Section $1$ and compatible with our datum for $\GU$ and such that $\Leta'$ restricted to $\Res_{E/ \Q_p} \Gm$ is just $\LL\lambda$. In particular, we have a commutative diagram:
\begin{equation}{\label{UResGUcommutativediagram}}
\begin{tikzcd}
\LL(\mH \times \Res_{E/ \Q_p} \Gm)  \arrow[r, "\Leta'"] & \LL(\U \times \Res_{E/ \Q_p} \Gm) \\
\LL\G(\mH) \arrow[u, "\LL P^{\mH}"] \arrow[r, swap, "\Leta"] & \LGU \arrow[u, swap, "\LL P"].   
\end{tikzcd}    
\end{equation}

We now prove the following lemma. 
\begin{lemma}{\label{matchlift}}
Using the above normalizations, if $f \in \mc{H}(\GU)$ and $f^{\mH} \in \mc{H}(\G(\mH))$  are $\Delta[\mf{w}_{\GU}, \varrho_{\GU}, z_{\GU}]$-matching, then the pullbacks $f' \in \mc{H}(\U \times \Res_{E/ \Q_p} \Gm)$ and $f'^{\mH} \in \mc{H}(\mH \times \Res_{E/ \Q_p} \Gm)$ are $\Delta[\mf{w}_{\U}, \varrho_{\U}, z_{\U}]$-matching.
\end{lemma}

We begin by proving an auxiliary lemma.

\begin{lemma} \phantomsection \label{itm : conjugacy classes}
For $(\gamma, z) \in \U \times \Res_{E/ \Q_p} \Gm(\Qp)$, the map $P$ gives a bijection between conjugacy classes in $\U \times \Res_{E/ \Q_p} \Gm(\Qp)$ that are stably conjugate to $(\gamma, z)$ and conjugacy classes in $\GU(\Qp) $ that are stably conjugate to $\gamma z$. The analogous result also holds for the map $P^{\mH}$.
\end{lemma}
\begin{proof}
Clearly, if $(\gamma', z')$ is conjugate or stable conjugate to $(\gamma, z)$ in $\U \times \Res \Gm (\Qp) $, then $\gamma' z$ and $\gamma z$ are conjugate or stably conjugate in $\GU (\Qp) $. Now, suppose that $g, \gamma z \in \GU(\Qp)$ are conjugate or stably conjugate. Then they must have the same similitude factor. In particular, this means that $gz^{-1}$ has trivial similitude factor and so $(gz^{-1}, z) \in \U \times \Res_{E/ \Q_p} \Gm(\Qp)$ and clearly $P(gz^{-1},z)=g$. 

We now aim to show that $(gz^{-1}, z)$ is conjugate or stably conjugate to $(\gamma ,z)$. To simplify the notation, we just show that $(gz^{-1}, z)$ and $(\gamma, z)$ are conjugate (although the argument to show stable conjugacy is similar). 

Let $x \in \GU(\Qp)$ be such that $xgx^{-1}=\gamma z$. We want to show that $x$ can be chosen to be an element of $\U(\Qp)$. Since the map $P$ is surjective on $\Q_p$ points, we can write $x=ur$ such that $u \in \U(\Qp)$ and $r \in \Res_{E/ \Q_p} \Gm (\Qp) $. Then $r$ lies in the center of $\GU(\Qp)$ and hence we have $ugu^{-1}=\gamma z$ as desired. Finally, we finish the argument by observing that that $(u,1)(gz^{-1}, z) (u,1)^{-1}=(\gamma ,z)$ since the restriction of $P$ to the first component is an injection.
\end{proof}

We now prove Lemma \ref{matchlift}.
\begin{proof}
Choose a strongly regular semisimple  $(\gamma_{\mH}, z) \in \mH \times \Res_{E/\Q_p} \Gm (\Qp)$ that transfers to a strongly regular $(\gamma, z) \in \U \times \Res_{E/ \Q_p} \Gm (\Qp)$. Then we need to show that
\begin{equation*}
    SO_{(\gamma_{\mH}, z)}({f'^{\mH}})= \sum\limits_{(\gamma', z) \sim_{st} (\gamma, z)} \Delta[\mf{w}_{\U}, \varrho_{\U}, z_{\U}]((\gamma_{\mH}, z), (\gamma', z))O_{(\gamma, z)}(f').
\end{equation*}
Expanding, this is equivalent to showing that
\begin{equation*}
    \sum\limits_{(\gamma'_{\mH}, z) \sim_{st} (\gamma_{\mH}, z)}\int\limits_{\mH \times \Res_{E/\Q_p} \Gm/ T_{(\gamma'_{\mH}, z)}} {f^{\mH}}'(h(\gamma'_{\mH}, z)h^{-1})dh
\end{equation*}
equals
\begin{equation*}
    \sum\limits_{(\gamma', z) \sim_{st} (\gamma, z)} \Delta[\mf{w}_{\U}, \varrho_{\U}, z_{\U}]((\gamma_{\mH}, z), (\gamma', z))\int\limits_{\U \times \Res_{E/\Q_p} \Gm/T_{(\gamma, z)}} f'(g(\gamma, z)g^{-1})dg.
\end{equation*}
Note that the kernels of $P^{\mH}, P$ are contained within $T_{(\gamma_{\mH}, z)}$ and $T_{(\gamma', z)}$ respectively. Hence we have $(\U \times \Res_{E/\Q_p} \Gm)(\Q_p)/T_{(\gamma, z)}(\Q_p)= \GU(\Q_p)/ T_{\gamma z}(\Q_p)$ and the analogous statement also holds for $P^{\mH}$.

By Lemma \ref{itm : conjugacy classes}, we can rewrite the equation above as 
\begin{equation*}
    \sum\limits_{\gamma'_{\mH}z \sim_{st} \gamma_{\mH} z}\int\limits_{\G(\mH)/ T_{\gamma'_{\mH}z}} f^{\mH}(h\gamma'_{\mH} zh^{-1})dh
\end{equation*}
equals
\begin{equation*}
    \sum\limits_{\gamma' z \sim_{st} \gamma z} \Delta[\mf{w}_{\U}, \varrho_{\U}, z_{\U}]((\gamma_{\mH}, z), (\gamma', z))\int\limits_{\GU /T_{\gamma' z}} f(g \gamma' zg^{-1})dg.
\end{equation*}
In Lemma \ref{itm : transfer factors for U and GU} we prove that there is an equality of transfer factors:
\begin{equation*}
    \Delta[\mf{w}_{\U}, \varrho_{\U}, z_{\U}]((\gamma_{\mH}, z) , (\gamma', z))=\Delta[\mf{w}_{\GU}, \varrho_{\GU}, z_{\GU}](\gamma_{\mH} z , \gamma'z).
\end{equation*}
Hence, the above equation reduces to 
\begin{equation*}
    SO_{\gamma_{\mH}z}(f^{\mH})=\sum\limits_{\gamma'z \sim_{st} \gamma z} \Delta[\mf{w}_{\GU}, \varrho_{\GU}, z_{\GU}](\gamma_{\mH}z, \gamma'z)O_{\gamma'z}(f),
\end{equation*}
which is true by assumption.
\end{proof}

With this lemma in hand, we now prove the endoscopic character identities. Pick a parameter $\psi \in \Psi^+(\GU^*_{\Q_p}(n))$ and let $\psi' \in \Psi^+(\U^*_{\Q_p}(n) \times \Res_{E/ \Q_p} \Gm)$ be the composition of $\psi$ with the map $\LL\GU \to \LL(\U^*_{\Q_p}(n) \times \Res_{E/ \Q_p} \Gm)$. We suppose $\psi$ factors through $\LL\G(\mH)$ and pick $\psi_{\G(\mH)}$ so that $\psi = \Leta \circ \psi_{\G(\mH)}$. We can write $\psi' = \psi_{\U} \times \psi_{\Gm}$ where $\psi_{\U}$ is the image of $\psi$ under the map $\LL\GU \to \LU$. Diagram \ref{UResGUcommutativediagram} implies there is a parameter $\psi'_{\mH}$ such that $\psi'=\Leta' \circ \psi'_{\mH}$. 

For a packet $\Pi_{\psi}(\GU, \varrho_{\GU}, z_{\GU})$, and matching functions $f \in \mc{H}(\GU)$, $f^{\mH} \in \mc{H}( \G(\mH))$ we have by definition 
\begin{equation*}
    e(\GU)\sum\limits_{\pi \in \Pi_{\psi}} \langle \pi, s \cdot s_{\psi} \rangle_{\GU} \tr( \pi \mid f)
    = e(\GU) \sum\limits_{(\pi, \chi) \in \Pi_{\psi}} \langle \pi, s \cdot s_{\psi_{\U}} \rangle_{\U} \chi_{z_{\GU}}(c)  \tr((\pi, \chi) \mid f).
\end{equation*}
We showed above that there is a natural bijection between $\Pi_{\psi}(\GU, \varrho_{\GU})$ and $\Pi_{\psi'}(\U \times \Res_{E/ \Q_p} \Gm, \varrho_{\U})$ and we related the traces of corresponding representations. The pairing $\langle \cdot, \cdot \rangle_{\U \times \Res_{E/ \Q_p} \Gm} : \Pi_{\psi'}(\U \times \Res_{E/\Q_p} \Gm) \times S^{\natural}_{\psi'} \longrightarrow \C^{\times}$ is given as a product of the pairing for $\U$ and $\Res_{E/ \Q_p} \Gm$. Hence we have the above equals
\begin{equation*}
    e(\GU)\frac{1}{\Vol(\U(1)(\Q_p))} \sum\limits_{\pi' \in \Pi_{\psi'}} \langle \pi', s \cdot s_{\psi'} \rangle_{\U \times \Res_{E/\Q_p} \Gm} \tr(\pi' \mid f').
\end{equation*}
Now, using that $e(\GU) = e(\U)=e(\U \times \Res_{E/ \Q_p} \Gm)$ (see \cite[pg. 292]{Kot4}) we can apply the previously established endoscopic character identity for $\U \times \Res_{E/ \Q_p} \Gm$ to get that the above equals
\begin{equation*}
    \frac{1}{\Vol(\U(1)(\Q_p))} \sum\limits_{\pi'_{\mH} \in \Pi_{\psi'_{\mH}}}\langle \pi'_{\mH}, s_{\psi'_{\mH}} \rangle_{\mH \times \Res_{E/ \Q_p} \Gm} \tr(\pi'_{\mH} \mid f'^{\mH}).
\end{equation*}
Finally we relate this to $\G(\mH)$ using that $\G(\mH)$ and $\mH \times \Res_{E/ \Q_p} \Gm$ are both assumed to be trivial extended pure inner forms so that the pairings are especially simple. We get:
\begin{equation*}
\sum\limits_{\pi_{\G(\mH)} \in \Pi_{\psi_{\G(\mH)}}}\langle \pi_{\G(\mH)}, s_{\psi_{\G(\mH)}} \rangle_{\G(\mH)} \tr(\pi_{\G(\mH)} \mid f^{\mH}),
\end{equation*}
which is the desired formula.

\subsection{Transfer factor identities} 

In this subsection, we prove a number of identities relating various transfer factors. These identities are used in the previous subsections. Remark that we use the letter $\Delta$ resp. $ \Delta' $ to denote the transfer factors that are compatible with the geometric normalization resp. arithmetic normalization of the local Artin reciprocity map.

\subsubsection{Transfer factors of a product}

We temporarily return to the notation of \S \ref{ECIproduct}.  We denote by $\G$ the group $\G_1 \times \G_2$ and by $\G^*$ the group $\G^*_1 \times \G^*_2$.

% \[
% \Delta : \mH(\mF)_{sr} \times \G(\mF)_{sr} \longrightarrow \C
% \]
% \[
% \Delta_1 : \mH_1(\mF)_{sr} \times \G_1(\mF)_{sr} \longrightarrow \C
% \]
% \[
% \Delta_2 : \mH_2(\mF)_{sr} \times \G_2(\mF)_{sr} \longrightarrow \C
% \]
% where the subscript ``sr'' means semi-simple and strongly regular (the centralizer is a maximal torus). 
%\begin{definition}(\cite{Kal1} p.6)
%Let $\gamma \in G^{\e}_{sr}(F)$ and $\delta \in G_{sr}(F)$. Let $S^{e}$ and $S$ be the centralizer of $\gamma$ resp $\delta$. The elements $\gamma$ and $\delta$ are called related if there exists an admissible isomorphism $S \longrightarrow S^{\e}$ mapping $\gamma$ to $\delta$. If such an isomorphism exists, it is unique and will be called $\varphi_{\gamma, \delta}$. 
%\end{definition}

% \begin{definition}(\cite{Kal1} p.6)
% Let $\gamma \in \mH_{sr}(\mF)$ and $\delta \in \G_{sr}(\mF)$. Let $S^{\mH}$ and $S$ be the centralizer of $\gamma$ resp $\delta$. The elements $\gamma$ and $\delta$ are called related if there exists an admissible isomorphism $S \longrightarrow S^{\mH}$ mapping $\gamma$ to $\delta$. If such an isomorphism exists, it is unique and will be called $\varphi_{\gamma, \delta}$. 
% \end{definition}
We prove the following lemma 
\begin{lemma} \phantomsection \label{itm : product of transfer factor}
Let $ (\gamma_1, \gamma_2) \in (\mH_1 \times \mH_2) (\Q_p)_{sr} $ and $(\delta_1, \delta_2)  \in (\G_1 \times \G_2)(\Q_p)_{sr}$ be related elements. We have
\begin{equation*}
    \Delta[\mf{w}_1 \times \mf{w}_1, \varrho_1 \times \varrho_2, z_1 \times z_2]((\gamma_1,\gamma_2), (\delta_1, \delta_2)) = \Delta[\mf{w}_1, \varrho_1, z_1](\gamma_1, \delta_1) \Delta[\mf{w}_2, \varrho_2, z_2](\gamma_2, \delta_2).
\end{equation*}
\end{lemma}
\begin{proof}
Each transfer factor is a product of terms
\[
\epsilon^{\G^*_i}_L(V^{\G^*_i}, \varphi){\Delta^{\G^*_i}_{I}} \Delta^{\G^*_i}_{II} \Delta^{\G^*_i}_{III_{2,D}} \Delta^{\G^*_i}_{IV}\langle \inv[z_i](\delta_i, \delta^*_i), s_i \rangle^{-1}.
\]

We state everything for $\G_i$ but the definitions are analogous for $\G$. We now explain the terms in the above formula. Notably, all the terms except the last only depend on $\G^*_i$ and $\mH_i$ (as opposed to $\G_i$). Fix a $\delta^*_i \in \G^*_i(\Q_p)$ such that $\delta^*_i$ is stably conjugate to $\varrho^{-1}_i(\delta)$. Recall that we have fixed $\Q_p$-splittings $(T_i, B_i, \{X_{i,\alpha}\})$ for $\G^*_i$ as well as the $\Q_p$-splitting $(T=T_1 \times T_2, B=B_1 \times B_2,  \{X_{\alpha}\} = \{X_{1,\alpha}\} \coprod  \{X_{2,\alpha}\})$ of $\G^*$. 

Now, $V^i$ is the degree $0$ virtual Galois representation $X^*(T_i) \otimes \C - X^*(T_i^{\mH_i}) \otimes \C$ and $\varphi$ is the additive character we fixed in order to define our Whittaker datum. The term $ \epsilon^{\G^*_i}_L(V^i, \varphi) $ is the local $\epsilon$-factor of this representation normalized as in \cite[\S 3.6]{Tate}. We also know that $\epsilon_L(V, \varphi)$ is additive for degree $0$ virtual representations $V$ (see \cite[Theorem. 3.4.1]{Tate}), therefore $ \epsilon^{\G^*}_L(V^{\G^*}, \varphi) = \epsilon^{\G^*_1}_L(V^{\G^*_1}, \varphi) \epsilon^{\G^*_2}_L(V^{\G^*_2}, \varphi)$. 

We denote by $S_i$ the centralizer of $\delta^*_i$ and $S_i^{\mH_i}$ the centralizer of $\gamma_i$ so that $ S = S_1 \times S_2$ and $ S^{\mH} = S_1^{\mH_1} \times S_2^{\mH_2}$ are the centralizers of $(\delta_1, \delta_2)$ resp. $(\gamma_1, \gamma_2)$.

We put $ D_{\G}((\delta_1, \delta_2)) = | \prod\limits_{\alpha} (\alpha(\delta_1, \delta_2)-1) |^{1/2} $ where the product is over all roots of $S$ in $\G$. Similarly $ D_{\G_i}(\delta_i) = | \prod\limits_{\alpha} (\alpha(\delta_i)-1) |^{1/2} $ where the product is over all roots of $S_i$ in $\G_i$. In particular we have
\[
D_{\G}((\delta_1, \delta_2)) = D_{\G_1}(\delta_1) D_{\G_2}(\delta_2).
\]

We define $D_{\mH}(\gamma_1, \gamma_2)$  and $D_{\mH_i}(\gamma_i)$ analogously and we also have the equality
\[
D_{\mH}(\gamma_1, \gamma_2) = D_{\mH_1}(\gamma_1) D_{\mH_2}(\gamma_2).
\]

By definition $\Delta_{IV} = D_{\G} D_{\mH}^{-1}$ so that we have
\[
\Delta_{IV}^{\G^*}((\gamma_1, \gamma_2), (\delta_1, \delta_2)) = \Delta_{IV}^{\G^*_1}(\gamma_1, \delta_1) \Delta_{IV}^{\G^*_2}(\gamma_2, \delta_2).
\]

For the other terms in the definition of the transfer factors, we need to explain the notions of $a$-data and $\chi$-data. A set of $a$-data for the set $R(\T,\G)$ of absolute roots of $S$ in $\G$ is a function 
\[
R(\T,\G) \longrightarrow \overline{\Q}_p^{\times}, \quad \alpha \longmapsto a_{\alpha}
\]
which satisfies $a_{-\lambda} = -a_{\lambda}$ and $a_{\sigma \lambda} = \sigma (a_{\lambda})$ for $\sigma \in \Gamma_{\Q_p}$. We recall the notion of $\chi$-data. For $\alpha \in R(\T, \G)$, we set $\Gamma_{\alpha} = \Stab(\alpha, \Gamma)$ and $\Gamma_{\pm \alpha} = \Stab(\{ \alpha, - \alpha \}, \Gamma)$ and denote $\mF_{\alpha}$, $\mF_{\pm \alpha}$ the fixed fields of $\Gamma_{\alpha}$ resp. $\Gamma_{\pm \alpha}$. A set of $\chi$-data is then a set of characters
\[
\chi_{\alpha} : \mF_{\alpha}^{\times} \longrightarrow \C^{\times} 
\]
satisfying the conditions $\chi_{\sigma \alpha} = \chi_{\alpha} \circ \sigma^{-1}, \chi_{-\alpha} = \chi_{\alpha}^{-1}$ and if $[\mF_{\alpha} : \mF_{\pm \alpha}] = 2 $ then $\chi_{\alpha}|_{\mF_{\pm \alpha}^{\times}} $ is non-trivial but trivial on the subgroup of norms from $\mF^{\times}_{\alpha}$.

Since $\Gamma_{\Q_p}$ acts on $\G^*_{\overline{\Q}_p}$ and preserves $(\G^*_i)_{\overline{\Q}_p}$, we see that if $(a_{\alpha})_{\alpha \in R(S_i, \G^*_i)}$ and $(\chi_{\alpha})_{\alpha \in R(S_i, \G^*_i)}$ are $a$-data resp $\chi$-data of $(S_i, \G^*_i)$ then $(a_{\alpha})_{\alpha \in R(S, \G^*)}$ and $(\chi_{\alpha})_{\alpha \in R(S, \G)}$ are $a$-data resp $\chi$-data of $(S, \G^*)$.

Now, we define 
\[
\Delta_{II}^{\G^*_i} = \prod_{\alpha} \chi_{\alpha} \left( \dfrac{\alpha(\delta_i) - 1}{a_{\alpha}} \right)
\]
where the product is taken over the set $R(S_i, \G_i) \setminus \varphi_{\gamma, \delta}^{*, -1} R(S^{\mH_i}_i, \mH_i) $. We have a similar formula for $\Delta_{II}^{\G^*}$ in which the product runs over the set $R(S, \G^*) \setminus \varphi_{\gamma, \delta}^{*, -1} R(S^{\mH}, \mH)  = R(S_1, \G^*_1) \setminus \varphi_{\gamma, \delta}^{*, -1} R(S^{\mH_1}_1, \mH_1)  \bigsqcup R(S_2, \G^*_2) \setminus \varphi_{\gamma, \delta}^{*, -1} R(S^{\mH_2}_2, \mH_2) $. In particular we have
\[
\Delta_{II}^{\G^*} = \Delta_{II}^{\G^*_1} \Delta_{II}^{\G^*_2}.
\]

Next, we want to show that 
\[
\Delta_I^{\G^*} = \Delta_I^{\G^*_1} \Delta_I^{\G^*_2}.
\]

First, for $i \in \{1,2\} $ one constructs an element $\lambda_i \in H^1(\Gamma_{\Q_p}, (S_i)_{sc})$ and then uses the Tate-Nakayama duality for tori in order to get a pairing $ \langle \cdot , \cdot \rangle $ between $H^1(\Gamma_{\Q_p}, (S_i)_{sc})$ and $\pi_0([\widehat{S_i}/Z(\widehat{\G^*_i})]^{\Gamma_{\Q_p}})$. One can view $s_i$ as an element of $[Z(\widehat{\mH}_i)/Z(\widehat{\G^*_i})]^{\Gamma_{\Q_p}}$, embed the latter into $\widehat{S^{\mH}_i} / Z(\widehat{\G^*_i})$, and transport it to $\widehat{S_i} / Z(\widehat{\G^*_i})$ by the admissible isomorphism $\varphi_{\gamma, \delta}$. We then define
\[
\Delta_I^{\G^*_i} = \langle \lambda_i,s_i \rangle.
\]

Because $S = S_1 \times S_2$ and $S_{sc} = (S_1)_{sc} \times (S_2)_{sc}$, to show the necessary product relation for this term, it is enough to show that $\lambda = \lambda_1 \times \lambda_2$.

We recall the construction of $\lambda$ for $\G^*$ and $S$. Write $\Omega(T,\G^*)$ for the absolute Weyl group and let $g \in \G^*$ be such that $gTg^{-1} = S$. For each $\sigma \in \Gamma_{\Q_p}$ there exists $\omega(\sigma) \in \Omega(T,\G^*)$ such that for all $t \in T$
\[
\omega(\sigma)\sigma(t) = g^{-1}\sigma(gtg^{-1})g.
\]

Let $\omega(\sigma) = s_{\alpha_1} \cdots s_{\alpha_k}$ be a reduced expression and let $n_i$ be the image of $\left( \begin{array}{cc}
     0 & 1  \\ 
    -1 & 0
\end{array} \right)$ under the homomorphism $\SL_2 \longrightarrow \G^*$ attached to the simple root vector $X_{\alpha_i}$. Then $n(\alpha) = n_1 \cdots n_k$ is independent of the choice of the reduced expression. So $\lambda \in H^1(\Gamma_{\Q_p}, S_{sc})$ is defined by the following $1$-cocycle
\[
\sigma \longmapsto g(\prod_{\alpha} \alpha^{\vee}(a_{\alpha})n(\sigma)[g^{-1}\sigma(g)]^{-1})g^{-1}
\]
where the product runs over the subset $\{ \alpha > 0, \sigma^{-1} \alpha < 0 \}$ of $R(S, \G^*)$ where positivity is determined by the Borel subgroup $gBg^{-1}$. The construction is analogous for $\G^*_i$.

Now, we have
\begin{enumerate}
    \item $B = B_1 \times B_2$,
    \item $T = T_1 \times T_2$, $S = S_1 \times S_2$,
    \item $R(S, \G^*) = R(S_1, \G^*_1) \bigsqcup R(S_2, \G^*_2)$ so that $(X_{\alpha})_{\G^*} = (X_{\alpha})_{\G^*_1} \bigsqcup (X_{\alpha})_{\G^*_2}$.
\end{enumerate}

We see that $\Omega(T, \G^*) = \Omega(T_1, \G^*_1) \times \Omega(T_2, \G^*_2)$ and we can take $g = g_1 \times g_2$ so that $\omega(\sigma)_{\G^*} = \omega(\sigma)_{\G^*_1} \times \omega(\sigma)_{\G^*_2} $. Therefore, $n(\sigma)_{\G^*} = n(\sigma)_{\G^*_1} \times n(\sigma)_{\G^*_2} $. We conclude that $\lambda = \lambda_1 \times \lambda_2$.

We are now going to show that 
\[
\Delta_{III_{2,D}}^{\G^*} = \Delta_{III_{2,D}}^{\G^*_1} \Delta_{III_{2,D}}^{\G^*_2}.
\]

The construction is as follows. First, we associate to the fixed $\chi$-datum a $\widehat{\G^*}$-embedding $\xi_{\G^*} : \prescript{L}{}{S} \longrightarrow \prescript{L}{}{\G^*}$ (\cite{LS86}, 2.6). Next via the admissible isomorphism $\varphi_{\gamma, \delta}$ the $\chi$-datum can be transferred to $S^{\mH}$ and gives an $L$-embedding $\xi^{\mH} : \prescript{L}{}{S}^{\mH} \longrightarrow \LL \mH$. The admissible isomorphism $\varphi_{\gamma, \delta}$ also provides dually an $L$-isomorphism $\prescript{L}{}{\varphi}_{\gamma, \delta} : \prescript{L}{}{S} \longrightarrow \prescript{L}{}{S}^{\mH}$. The composition $\xi' = \prescript{L}{}{\eta} \circ \xi^{\mH} \circ \prescript{L}{}{\varphi}_{\gamma, \delta}$ gives another $L$-embedding $\prescript{L}{}{S} \longrightarrow \prescript{L}{}{\G^*}$. Via conjugation by an element of $\widehat{\G^*}$, we can arrange that $\xi_{\G^*}$ and $\xi'$ coincide on $\widehat{S}$ so that $\xi' = a \cdot \xi_{\G}$ for some $a \in Z^{1}(W_{\Q_p}, \widehat{S})$.

The term $\Delta^{\G^*}_{III_{2,D}}$ is given by $\langle a, \delta \rangle $ where the paring $\langle \cdot , \cdot \rangle$ is the Langlands duality for tori under the geometric normalization. More precisely, the element $a$ of $Z^{1}(W_{\mF}, \widehat{S})$ is an $L$-parameter of $S$. By the local Langland correspondence for tori, $a$ gives rise to a character $ \langle a, \cdot  \rangle$ of $S$.  

In our case, we have $ S = S_1 \times S_2 $ and $\delta = (\delta_1 , \delta_2)$ so it suffices to show that $a = a_1 \times a_2$. In order to verify that, we need to review carefully the formation of the $L$-embedding $ \xi : \prescript{L}{}{S} \longrightarrow \prescript{L}{}{\G} $ associated to a $\chi$-datum \cite[\S 2.6]{LS86}.

Fix a Borel pair $(\widehat{B}, \widehat{T})$ of $\widehat{\G^*}$ as well as a Borel subgroup $B_S$ (possibly not defined over $\Q_p$) of $\G^*$ containing $S$. The pair $(B_S, S)$ yields a set of positive coroots of $S$ and equivalently a set of elements of $X^*(\widehat{S})$. Then $\xi$ is defined so that the restriction to $\widehat{S}$ maps $\widehat{S}$ to $\widehat{T}$ by the unique isomorphism mapping our chosen subset of $X^*(\widehat{S})$ to the set of positive roots of $\widehat{T}$ determined by $\widehat{T}$.

Then, to specify $\xi$ we have only to give a homomorphism $w \mapsto \xi(w) = \xi_0 (w) \times w$ where $\xi_0(w) \in \mathrm{Norm}(\widehat{T}, \widehat{\G^*})$. We require that if $w \mapsto \sigma$ under $W_{\Q_p} \longrightarrow \Gamma_{\Q_p}$ then $\Int(\xi(w))$ acts on $\widehat{\T}$ as the transport by $\xi$ of the action of $\sigma \in \Gamma_{\Q_p}$ on $\widehat{S}$.

We then define
\[
\xi(w) = r_p(w)n(\sigma) \times w
\]
for $w \in W_{\Q_p}$ and $w \mapsto \sigma$ under $W_{\Q_p} \longrightarrow \Gamma_{\Q_p} $. The term $n(\sigma)$ is defined above, in the definition of $\Delta_I$ and we have already seen that $n(\sigma) = n(\sigma)_{\G_1} \times n(\sigma)_{\G_2} $.

We recall briefly the construction of $r_p(w)$. We denote by $\mR$ the set $R^{\vee}(\G^*,S)$ and define $\Sigma$ to be the group of automorphisms of $\mc{R}$ generated by $\Gamma_{\Q_p}$ and $\epsilon$ where $\epsilon$ acts on $X_*(S)$ by $\epsilon (t) = - t$ (as in\cite[Lemma 2.1A]{LS86}). The group $\Sigma$ acts on $\mR$ and divides it into $\Sigma$-orbits $\mR = \mR_1 \bigsqcup \cdots \bigsqcup \mR_k$. For each $\Sigma$-orbit $\mR_i$, we define an element $r_p^i(w)$ and then take the product over the orbits to obtain $r_p(w)$. Since $\mR_{\G^*} = \mR_{\G^*_1} \bigsqcup \mR_{\G^*_2}$ and the group $\Sigma$ preserves $\mR_{\G^*_1}$, $\mR_{\G^*_1}$, we have that $r_p(w)_{\G^*} = r_p(w)_{\G^*_1} \times r_p(w)_{\G^*_2}$. This implies the desired product identity for $\Delta^{\G^*}_{III_{2, D}}$.

Finally, we show that 
\begin{equation*}
    \langle \inv[z_1 \times z_2]((\delta_1, \delta_2), (\delta^*_1, \delta^*_2), s_1 \times s_2 \rangle=\langle \inv[z_1](\delta_1, \delta^*_1), s_1 \rangle\langle \inv[z_2](\delta_2, \delta^*_2), s_2 \rangle.
\end{equation*}
We have a natural isomorphism $\mb{B}(\Q_p, S) = \mb{B}(\Q_p, S_1 \times S_2)$ that maps the class of $g^{-1}(z_1 \times z_2)\sigma(g)$ to the product of the classes of $g^{-1}_1z_1\sigma(g_1)$ and $g^{-1}_2z_2\sigma(g_2)$. Moreover, this product decomposition respects the Kottwitz maps $\kappa_i: \mb{B}(\Q_p, S_i) \to X^*(\widehat{S_i})^{\Gamma_{\Q_p}}$ defining the above pairings. This implies the desired product formula.
\end{proof}
\subsubsection{Transfer factor and changing the normalization}

\begin{lemma}{\label{arithmetic - geometric normalization}}
Let $f \in \mc{H}(\U)$ and $f^{\mH} \in \mc{H}(\mH)$ be $\Delta[\mf{w}^{-1}, \varrho, z]$-matching functions for an endoscopic datum $(\mH, s, \Leta)$ of $\U$. If $i_{\U}: \U(\Q_p) \to \U(\Q_p)$ and $i_{\mH}: \mH(\Q_p) \to \mH(\Q_p)$ are the inverse functions, then $f^{\mH} \circ i_{\mH}$ and $f \circ i_{\U}$ are matching for the transfer factors $\Delta'[\mf{w}, \varrho, z]$ with respect to the endoscopic datum $(\mH, s^{-1}, \Leta)$.
\end{lemma}
\begin{proof}
We consider first the ordinary endoscopic case. Suppose $ \gamma_{\mH} \in \mH(\Q_p) $ is strongly regular and transfers to a strongly regular element $ \gamma \in \U (\Q_p) $. By hypothesis, we have
\[
SO^{\mH}_{\gamma_{\mH}}(f^{\mH} ) = \sum\limits_{\gamma' \sim_{st} \gamma} \Delta[\mf{w}^{-1}, \varrho, z](\gamma_{\mH}, \gamma') O^{\U}_{\gamma'}(f).
\]

Then we need to show that 
\[
SO^{\mH}_{\gamma_{\mH}}(f^{\mH} \circ i_{\mH} ) = \sum\limits_{\gamma' \sim_{st} \gamma} \Delta'[\mf{w}, \varrho, z](\gamma_{\mH}, \gamma') O^{\U}_{\gamma'}(f \circ i_{\U}).
\]

Since $SO^{\mH}_{\gamma_{\mH}}(f^{\mH} \circ i_{\mH} ) = SO^{\mH}_{\gamma^{-1}_{\mH}}(f^{\mH} )$ and $O^{\U}_{\gamma'}(f \circ i_{\U}) = O^{\U}_{(\gamma')^{-1}}(f) $, it suffices to show that the transfer factor $\Delta[\mf{w}^{-1}, \varrho, z](\gamma^{-1}_{\mH}, (\gamma')^{-1})$ with respect to the endoscopic datum $(\mH, s, \Leta)$ is the same as the transfer factor $\Delta'[\mf{w}, \varrho, z](\gamma_{\mH}, \gamma')$ with respect to the endoscopic datum $(\mH, s^{-1}, \Leta)$. 

Recall that the transfer factor $\Delta'[\mf{w}, \varrho, z]$ is a product of terms
\[
\epsilon_L(V, \varphi){\Delta^{-1}_I} \Delta_{II} \Delta_{III_2} \Delta_{IV}\langle \inv[z](\gamma, \gamma^*), s \rangle
\]
which we need to use $\chi$-data and $a$-data in order to define and moreover the transfer factors \emph{do not depend} on the choices of $\chi$-data and $a$-data. 

By \cite[Section 5.1]{KS2}, the transfer factor $\Delta[\mf{w}, \varrho, z]$ is defined by the same formula, except that one replaces the term $ \Delta_{III_2} $ by $ \Delta_{III_2, D} $, inverts $\Delta_I$ and inverts $\langle \inv[z](\delta, \delta^*), s \rangle$. If one keeps track of the dependence on $\chi$-data and $a$-data, then $ \Delta_{III_2, D, \chi^{-1}} (\gamma^{-1}_{\mH}, (\gamma')^{-1}) = \Delta_{III_2, \chi} (\gamma_{\mH}, \gamma') $. 

By using the definitions of the terms appearing in the transfer factors which we recalled in Lemma \ref{itm : product of transfer factor}, we have 
\begin{equation*}
   \epsilon_L(V, \varphi) \Delta^{-1}_{I, a}[s^{-1}] \Delta_{IV} (\gamma_{\mH}, \gamma') = \epsilon_L(V, \varphi) \Delta_{I, a}[s] \Delta_{IV} (\gamma^{-1}_{\mH}, (\gamma')^{-1}), 
\end{equation*}
 since these terms do not depend on $\chi$-data and where the $\Delta_{I,a}[s]$ notation keeps track of whether we plug in  $s$ or $s^{-1}$ into the pairing defining $\Delta_I$. Moreover $ \Delta_{II, \chi^{-1}, a^{-1}} (\gamma^{-1}_{\mH}, (\gamma')^{-1}) = \Delta_{II, \chi, a} (\gamma_{\mH}, \gamma') $. Thus we have
\begin{align*}
    &\Delta'[\mf{w}, \varrho, z](\gamma_{\mH}, \gamma')\\
     = &\epsilon_L(V, \varphi)\Delta^{-1}_{I, a} [s^{-1}]\Delta_{IV} (\gamma_{\mH}, \gamma')\Delta_{II, \chi, a} (\gamma_{\mH}, \gamma')\Delta_{III_2, \chi} (\gamma_{\mH}, \gamma')\langle \inv[z](\gamma, \gamma^*), s^{-1} \rangle \\
    = &\epsilon_L(V, \varphi)\Delta_{I, a}[s]  \Delta_{IV} (\gamma^{-1}_{\mH}, (\gamma')^{-1})\Delta_{II, \chi^{-1}, a^{-1}} (\gamma^{-1}_{\mH}, (\gamma')^{-1})\\
    &\cdot \Delta_{III_2, D, \chi^{-1}} (\gamma^{-1}_{\mH}, (\gamma')^{-1})\langle \inv[z](\gamma^{-1}, (\gamma^{-1})^*), s \rangle^{-1}.
\end{align*}

Therefore $\Delta[\mf{w}, \varrho, z](\gamma^{-1}_{\mH}, (\gamma')^{-1})$ with respect to the endoscopic datum $(\mH, s, \Leta)$ is nearly the same as $\Delta'[\mf{w}, \varrho, z](\gamma_{\mH}, \gamma')$ with respect to the endoscopic datum $(\mH, s^{-1}, \Leta)$. The only difference is that in the above second product, the term $\Delta_I$ is defined with respect to $a$-data and the term $\Delta_{II}$ is defined with respect to the $a^{-1}$-data. However, the  $\Delta_I$ and $\epsilon_L(V, \varphi)$ terms also depend on the Whittaker datum. According to \cite[page 16]{KalContra}, we have $ \epsilon_L(V, \varphi)  \cdot \Delta_{I, a} (\gamma^{-1}_{\mH}, (\gamma')^{-1}) = \epsilon_L(V, \varphi^{-1})  \cdot \Delta_{I, a^{-1}} (\gamma^{-1}_{\mH}, (\gamma')^{-1}) $.

Since inverting the character $ \varphi $ leads to the inverse Whittaker datum $ \mf{w}^{-1} $, the second product is actually the transfer factor $\Delta[\mf{w}^{-1}, \varrho, z](\gamma^{-1}_{\mH}, (\gamma')^{-1})$ with respect to the endoscopic datum $(\mH, s, \Leta)$. 

For the twisted endoscopic case, the same arguments still work. Indeed, in this case $ H = G \rtimes \theta $ and we need to show that
\[
SO^{\mH}_{\gamma_{\mH}}(f^{\mH} \circ i_{\mH} ) = \sum\limits_{\gamma' \sim_{st} \gamma} \Delta'[\mf{w}, \varrho, z](\gamma_{\mH}, \gamma') O^{\U}_{\gamma'}(f \circ i_{\U}).
\]

Since $SO^{\mH}_{\gamma_{\mH}}(f^{\mH} \circ i_{\mH} ) = SO^{\mH}_{\gamma^{-1}_{\mH}}(f^{\mH} )$ and $O^{\U}_{\gamma'}(f \circ i_{\U}) = O^{\U}_{(\gamma')^{-1}}(f) $, it suffices to show that the transfer factor $\Delta[\mf{w}^{-1}, \varrho, z](\gamma^{-1}_{\mH}, (\gamma')^{-1})$ with respect to the endoscopic datum $(\mH, s, \Leta)$ is the same as the transfer factor $\Delta'[\mf{w}, \varrho, z](\gamma_{\mH}, \gamma')$ with respect to the endoscopic datum $(\mH, s^{-1}, \Leta)$. By the results in \cite[Sections 5.3, 5.4]{KS2}, we know that the twisted transfer factor $\Delta'[\mf{w}, \varrho, z]$ is a product of terms
\[
\epsilon_L(V, \varphi){(\Delta^{\text{new}}_I)^{-1}} \Delta_{II} \Delta^{-1}_{III_2} \Delta_{IV}\langle \inv[z](\delta, \delta^*), s \rangle
\]
and the twisted transfer factor $\Delta_D[\mf{w}, \varrho, z]$ is a product of terms
\[
\epsilon_L(V, \varphi){\Delta^{\text{new}}_I} \Delta_{II} \Delta^{\text{new}}_{III_2} \Delta_{IV}\langle \inv[z](\delta, \delta^*), s \rangle^{-1}.
\]

Since $\Delta^{\text{new}}_{III_2}$ is the term $\Delta_{III_2}$ computed for the inverse set of $\chi$-data, we see that $ \Delta^{\text{new}}_{III_2, \chi^{-1}} (\gamma^{-1}_{\mH}, (\gamma')^{-1}) = \Delta_{III_2, \chi} (\gamma_{\mH}, \gamma') $. Moreover $(\Delta^{\text{new}}_I)^{-1} (\gamma_{\mH}, \gamma')[s^{-1}] = \Delta^{\text{new}}_I (\gamma_{\mH}, \gamma')[s] $. Thus we have
\begin{align*}
    &\Delta'[\mf{w}, \varrho, z](\gamma_{\mH}, \gamma') \\
    =& \epsilon_L(V, \varphi) (\Delta^{\text{new}}_{I, a})^{-1}[s^{-1}] \Delta_{IV} (\gamma_{\mH}, \gamma')\Delta_{II, \chi, a} (\gamma_{\mH}, \gamma')\Delta^{-1}_{III_2, \chi} (\gamma_{\mH}, \gamma')\langle \inv[z](\gamma, \gamma^*), s^{-1} \rangle \\
    = &\epsilon_L(V, \varphi) \Delta^{\text{new}}_{I, a}[s] \Delta_{IV} (\gamma^{-1}_{\mH}, (\gamma')^{-1})
    \Delta_{II, \chi^{-1}, a^{-1}} (\gamma^{-1}_{\mH}, (\gamma')^{-1})\\ 
    & \cdot \Delta^{\text{new}}_{III_2, \chi^{-1}} (\gamma^{-1}_{\mH}, (\gamma')^{-1}) \langle \inv[z](\gamma^{-1}, (\gamma^{-1})^*), s \rangle^{-1}.
\end{align*}

As in the standard endoscopy case, the second product is actually the twisted transfer factor $\Delta[\mf{w}^{-1}, \varrho, z](\gamma^{-1}_{\mH}, (\gamma')^{-1})$ with respect to the endoscopic datum $(\mH, s, \Leta)$. 
\end{proof}

\subsubsection{Endoscopy for \texorpdfstring{$\Res_{E/\Q_p} \Gm$}{Res Gm}}

We now study the endoscopy of $\Res_{E/ \Q_p} \Gm$.

We must have  $\mH = \Res_{E/\Q_p} \Gm$ and pick $s \in \widehat{\mH}^{\Gamma_{\Q_p}}$. We will be most interested in the case where  $\Leta|_{\widehat{\mH}}$ is the identity map and so we assume this is the case. Then $\Leta$ is determined up to conjugacy by an element of $H^1(W_{\Q_p}, \widehat{\Res_{E/ \Q_p} \Gm})$. By the Langlands correspondence for tori, this cocycle corresponds to a character $\lambda$ of $\Res_{E/ \Q_p} \Gm(\Q_p) = E^{\times}$.

We now study transfer factors for the endoscopic datum $(\mH, s, \LL \lambda)$ of $\Res_{E/ \Q_p} \Gm$. Recall we have fixed an extended pure inner twist $(\varrho_{\Gm}, z_{\Gm})$ of $\Res_{E/ \Q_p} \Gm$ such that $\varrho_{\Gm}: (\Res_{E/\Q_p} \Gm)^* \to \Res_{E/\Q_p} \Gm$. Consider $z_{\mH} \in \mH(\Q_p)$ which transfers to $z \in \Res_{E/ \Q_p} \Gm$ and $z^* \in (\Res_{E/\Q_p} \Gm)^*$. Our goal is to compute the transfer factor $ \Delta[\mf{w}_{\Gm}, \varrho_{\Gm}, z_{\Gm}](z_{\mH}, z) $.
 
 \begin{lemma} \phantomsection \label{itm : L-embedding}
 We have 
 \[
 \Delta[\mf{w}_{\Gm}, \varrho_{\Gm}, z_{\Gm}](z_{\mH}, z)  = \lambda (z^*)\langle \inv[z_{\Gm}](z, z^*), s \rangle^{-1} .
 \]
 \end{lemma}
 \begin{proof}
 We will calculate each term in the definition of transfer factor. The virtual representation $V$ in this case is $0$ so that the factor $\epsilon (V, \varphi) = 1$. The terms $\Delta_{IV}$, $\Delta_{II}$ are trivial since $\Res_{E/ \Q_p} \Gm$ has no absolute roots. The term $\Delta_I$ is trivial since the group $\widehat{S}/Z(\widehat{\Res_{E/ \Q_p} \Gm})$ is trivial.
 
 We now compute $\Delta_{III_2}$. The $L$-maps $\xi_{(\Res_{E/\Q_p} \Gm)^*}$, $\xi_{\mH}$ and $\prescript{L}{}{\varphi_{z^*,z^*}}$ are all the identity. Hence, by comparing $\xi' = \eta \circ \xi_e \circ \prescript{L}{}{\varphi_{z^*,z^*}}$ with $\xi_{\Res_{E/\Q_p} \Gm}$, we see that $\Delta_{III_2} = \lambda(z^*)$.
 
The final term then contributes the factor $\langle \inv[z_{\Gm}](z, z^*), s \rangle^{-1}$, completing the argument.
 \end{proof}

\subsubsection{Transfer factors for \texorpdfstring{ $\GU$ and $\U \times \Res_{E/\Q_p} \Gm$}{GU and U times Res}}  

We use the notation of \S \ref{ECIforGU}. We denote the Whittaker datum and extended pure inner twists of $\U$ induced by restriction from $\GU$ by $\mf{w}'_{\U}$ and $(\varrho'_{\U}, z'_{\U})$. We record the following lemma:
\begin{lemma} \phantomsection \label{itm : Compare U and GU}
Suppose that $\gamma_H \in \mH(\Q_p)$ and $\gamma \in \U(\Q_p)$ are strongly regular and related. Then we have the following equality
\begin{equation*}
    \Delta[\mf{w}'_{\U}, \varrho'_{\U}, z'_{\U}](\gamma_H, \gamma)= \Delta[\mf{w}_{\GU}, \varrho_{\GU}, z_{\GU}](\gamma_{\mH}, \gamma)\langle \inv[z_{\GU}](\gamma, \gamma^*), s \rangle\langle \inv[z'_{\U}](\gamma, \gamma^*) , s \rangle^{-1}.
\end{equation*}
\end{lemma}
\begin{proof}
This is \cite[Lemma 3.6]{Xu} adapted to the non-quasisplit setting.
\end{proof}

Finally, we prove the following lemma:
\begin{lemma} \phantomsection \label{itm : transfer factors for U and GU}
Suppose $(\gamma, z) \in (\U \times \Res_{E/ \Q_p} \Gm)(\Q_p)_{sr}$ and $(\gamma_{\mH}, z_{\mH}) \in (\mH \times \Res_{E/ \Q_p} \Gm)(\Q_p)_{sr}$ are related. Then we have an equality of transfer factors
\[
\Delta[\mf{w}_{\U}, \varrho_{\U}, z_{\U}]((\gamma_{\mH}, z_{\mH}), (\gamma, z)) = \Delta[\mf{w}_{\GU}, \varrho_{\GU}, z_{\GU}](\gamma_Hz_{\mH}, \gamma z).
\]
\end{lemma}

\begin{proof}
First of all, by Lemma \ref{itm : product of transfer factor} we have
\[
\Delta[\mf{w}_{\U}, \varrho_{\U}, z_{\U}]((\gamma_{\mH}, z_{\mH}), (\gamma, z)) = \Delta[\mf{w}'_{\U}, \varrho'_{\U}, z'_{\U}](\gamma_H, \gamma) \cdot \Delta[\mf{w}_{\Gm}, \varrho_{\Gm}, z_{\Gm}] (z_{\mH}, z). 
\]
By Lemma \ref{itm : L-embedding}, this equals
\begin{equation*}
    \Delta[\mf{w}'_{\U}, \varrho'_{\U}, z'_{\U}](\gamma_H, \gamma) \cdot \lambda (z^*)\langle \inv[z_{\Gm}](z, z^*), s \rangle^{-1},
\end{equation*}
and by Lemma \ref{itm : Compare U and GU} we have
\[
\Delta[\mf{w}'_{\U}, \varrho'_{\U}, z'_{\U}](\gamma_H, \gamma)= \Delta[\mf{w}_{\GU}, \varrho_{\GU}, z_{\GU}](\gamma_H, \gamma)\langle \inv[z_{\GU}](\gamma, \gamma^*), s \rangle\langle \inv[z'_{\U}](\gamma, \gamma^*) , s \rangle^{-1}.
\]
Since the Kottwitz set and the Kottwitz map $\kappa$ respect products, we get that
\begin{equation*}
    \langle \inv[z'_{\U}](\gamma, \gamma^*), s \rangle\langle \inv[z_{\Gm}](z,z^*), s \rangle = \langle \inv[z_{\U}]((\gamma, z), (\gamma^*, z^*)), s \rangle.
\end{equation*}
By the functoriality of the Kottwitz map,
\begin{equation*}
    \langle \inv[z_{\U}]((\gamma, z), (\gamma^*, z^*)), s \rangle = \langle \inv[z_{\GU}](\gamma z, \gamma^* z^*), s \rangle.
\end{equation*}
Hence we get
\begin{equation*}
    \Delta[\mf{w}_{\U}, \varrho_{\U}, z_{\U}]((\gamma_{\mH}, z_{\mH}), (\gamma, z))
\end{equation*}
\begin{equation*}
    =  \Delta[\mf{w}_{\GU}, \varrho_{\GU}, z_{\GU}](\gamma_H, \gamma)\langle \inv[z_{\GU}](\gamma, \gamma^*), s \rangle\langle \inv[z_{\GU}](\gamma z, \gamma^* z^*), s \rangle^{-1}.
\end{equation*}

On the other hand, by \cite[Lemma 4.4A]{LS86}, there is a character $\lambda'$ on $(\Res_{E/ \Q_p} \Gm)(\Q_p)$ such that
\[
\Delta[\mf{w}_{\GU}, \varrho_{\GU}, z_{\GU}](\gamma_{\mH}z_{\mH}, \gamma z)
\]
\[
= \Delta[\mf{w}_{\GU}, \varrho_{\GU}, z_{\GU}](\gamma_{\mH}, \gamma)\lambda'(z^*)\langle \inv[z_{\GU}](\gamma, \gamma^*), s \rangle\langle \inv[z_{\GU}](\gamma z, \gamma^* z^*), s \rangle^{-1}.
\]

Hence, it remains to show that $\lambda'(z^*) = \lambda(z^*)$. We recall that $\lambda$ is the character arising from the construction of the $\Delta_{III_2}$-term of the transfer factor for $\Res_{E/ \Q_p} \Gm$. From the description in \cite[Lemma 4.4A]{LS86}, $\lambda'$ is the restriction to $Z(\GU)=\Res_{E/ \Q_p} \Gm$ of the character arising from the $\Delta_{III_2}$-term of the transfer factor for $\GU$.

The characters $\lambda$ and $\lambda'$ are determined by the failure of the following diagram to commute:
\begin{equation*}
\begin{tikzcd}
\LL(\Res_{E/\Q_p} \Gm) \arrow[ddd] & & & \LL(\Res_{E/\Q_p} \Gm) \arrow[ddd] \arrow[lll] \\
 & \LL S^{\G(\mH)} \arrow[d, swap, "\xi^{\G(\mH)}"] \arrow[lu] & \LL S \arrow[d, "\xi^{\GU}"] \arrow[ur] \arrow[l, swap, "\LL\varphi_{z^*,z^*}"] & \\
 & \LL \G(\mH) \arrow[r, swap, "\Leta"] \arrow[ld] & \LL \GU \arrow[rd] & \\
\LL(\Res_{E/\Q_p} \Gm) \arrow[rrr, "\LL \lambda"] & & & \LL(\Res_{E/\Q_p} \Gm)
\end{tikzcd}    
\end{equation*}
We explain this diagram. The objects $S^{\G(\mH)}$ and $S$ are maximal tori in their respective groups that are isomorphic by an admissible embedding $\LL \varphi_{z^*,z^*}$. The maps $\xi^{\G(\mH)}$ and $\xi^{\GU}$ are the $L$-embeddings constructed in \cite[\S (2.6)]{LS86} from a choice of $\chi$-data. The lower two diagonal maps are induced by the embeddings $\Res_{E/\Q_p} \Gm \cong Z(\GU) \hookrightarrow \GU$ and $\Res_{E/\Q_p} \Gm \cong Z(\G(\mH)) \hookrightarrow \G(\mH)$. Since the images of these embeddings lie in the image of the embeddings $S \hookrightarrow G$ and $S^{\G(\mH)} \hookrightarrow \G(\mH)$ respectively, we get induced maps $\Res_{E/\Q_p} \Gm \hookrightarrow S^{\G(\mH)}$ and $\Res_{E/\Q_p} \Gm \hookrightarrow S$. These induce the upper diagonal maps in the above diagram. The outer vertical arrows are then defined so that the left and right trapezoids commute. Note that by definition of $n(w)$ and $r_p(w)$ the vertical maps $\LL(\Res_{E/ \Q_p} \Gm) \to \LL(\Res_{E/ \Q_p} \Gm)$ are both the identity. The bottom trapezoid commutes by construction. Finally the top map is defined so that the top trapezoid commutes and will agree with $\Leta$ on $\widehat{\Res_{E/\Q_p} \Gm}$ and map $(1,w)$ to $(1,w)$.

Then the outer square fails to commute by the cocycle $\lambda \in Z^1(W_{\Q_p}, \widehat{\Res_{E/ \Q_p} \Gm})$ and the inner square fails to commute by $\lambda' \in  Z^1(W_{\Q_p}, \widehat{T})$. Since the trapezoids all commute, these cocycles agree under the natural map $ Z^1(W_{\Q_p}, \widehat{T}) \longrightarrow Z^1(W_{\Q_p}, \widehat{\Res_{E/ \Q_p} \Gm})$. This is the desired result.
\end{proof}

\section{Properties of the local and global correspondences}
In this section we prove a number of properties and compatibilities of the local and global Langlands correspondences. These properties are needed to derive our main theorem.

\subsection{Unramified representations}{\label{unramifiedGU}}
In this subsection we suppose that $v$ is a finite place of $\Q$ and that $E_v/\Q_v$ is unramified. We let $\GU, (\id, 1)$ and $\U, (\id, 1)$ be the trivial extended pure inner twists of $\GU^*_{\Q_v}(n)$ and $\U^*_{\Q_v}(n)$ respectively. Let $\GU(\Z_p)$ be the standard hyperspecial subgroup. Then we say that $\pi$ is $\GU(\Z_p)$-spherical if it has nontrivial $\GU(\Z_p)$-invariants.
\begin{proposition}
Let $\psi : L_{\Q_v} \longrightarrow \prescript{L}{}{\GU^*_{\Q_v}(n)}  \in \Psi^+(\GU^*_{\Q_v}(n))$ be a generic parameter. Then $ \Pi_{\psi}(\GU, \id)$ contains a $\GU(\Z_v)$-spherical representation if and only if $\psi$ is unramified. In that case, $ \Pi_{\psi}(\GU, \id)$ contains a unique $\GU(\Z_p)$-spherical representation $\pi$, which satisfies $ \langle \pi, \cdot \rangle = 1 $. The same results hold true for $\U$.
\end{proposition}

\begin{proof}
We first consider the case where $\psi \in \Psi(\GU^*_{\Q_v}(n))$ (or $\Psi(\U^*_{\Q_v}(n))$). By Corollary \ref{itm : GU = U E}, we see that a spherical representation $\pi$ of $\U(\Q_v)$ lifts to a spherical representation $\widetilde{\pi}$ of $\GU(\Q_v)$ and vice versa. Moreover, by the construction local packets for $\GU(\Q_v)$, we have that $ \langle \pi, \cdot \rangle = 1 $ if and only if $ \langle \widetilde{\pi}, \cdot \rangle = 1 $. Therefore it suffices to prove the proposition for unitary groups.

We mimic the proof of Lemma $ 4.1.1 $ in \cite{Tai2}. Denote $f$ the characteristic function of the standard special maximum compact subgroup of $ \U(\Q_v) $. If $ \psi $ is unramified then by proposition \cite[Proposition 7.4.3]{C.P.Mok} we have
\[
 1 = \sum_{\pi \in \Pi_{\psi}} \tr (\pi \mid f).
\]

In other words, the packet $\Pi_{\psi}$ contains an unramified representation. The uniqueness comes from Theorem $ 2.5.1 a $ in  \cite{C.P.Mok}.

Suppose now that $ \psi $ is ramified. Then the base change $L$-parameter $\eta_B \circ \psi $ is also ramified. By the local Langlands correspondence for $\GL_n(E_v)$, one gets a representation $\pi$ of $\GL_n(E_v)$ corresponding to $\eta_B \circ \psi$. Then, as in \cite[\S 3.2]{C.P.Mok}, one lifts $\pi$ to a representation $\tilde{\pi}$ of $\GL_n(E_v) \rtimes \theta \subset \GL_n(E_v) \rtimes \langle \theta \rangle$, where $\theta$ is the automorphism $g \mapsto J_n \sigma(g)^{-t} J^{-1}_n$ of $\Res_{E_v/ \Q_v} \GL_{n, E_v}$. Hence the corresponding representation of $\GL_n(\mc{O}_{E_v}) \rtimes \theta $ is ramified. We want to show that $ \sum\limits_{\pi \in \Pi_{\psi}} \langle \pi , x \rangle \tr (\pi \mid \ f) = 0 $ for every $ x \in \ov{\mathcal{S}}_{\psi}$. If we denote $ f_N $ the characteristic function of $\GL_n(\mc{O}_{E_v}) \rtimes \theta $ then $ f_N (\eta_B \circ \psi) = 0$. The twisted fundamental lemma implies that $ f_N $ is the twisted transfer of $f$ and hence by \cite[Theorem 3.2.1 a]{C.P.Mok} we have 
\[
\sum_{\pi \in \Pi_{\psi}} \tr (\pi \mid f) = \sum_{\pi \in \Pi_{\psi}} \langle \pi , 1 \rangle \tr (\pi \mid f) = f_N( \eta_B \circ \psi ) = 0.
\]

By the same argument we have $ \displaystyle \sum_{\pi^{\mH} \in \Pi_{\phi^{\mH}}} \langle \pi^{\mH} , 1 \rangle \tr (\pi^{\mH} \mid f_{\mH}) = 0 $ for every refined endoscopic datum $ (\mH, s, \Leta) $ of $\U$ where $f_{\mH}$ is the characteristic function of a hyperspecial subgroup $\mc{H}(\Z_v) $ of $\mH$. By the fundamental lemma, $ f_{\mH} $ is the transfer of $f$. Then, again by \cite[Theorem 3.2.1 ]{C.P.Mok} we have 
\[
\sum_{\pi \in \Pi_{\psi}} \langle \pi , x \rangle \tr (\pi \mid f) = \sum_{\pi^{\mH} \in \Pi_{\psi^{\mH}}} \langle \pi^{\mH} , 1 \rangle \tr (\pi^{\mH} \mid f_{\mH}) = 0.
\]
where $ (\psi, x) $ corresponds to $ (\mH, s, \Leta, \psi^{\mH})$ under \cite[Proposition 3.10]{BM2}. Hence we conclude that $ \tr (\pi \mid f) = 0 $ for every $ \pi \in \Pi_{\psi}(\U, \id) $. Therefore the packet $ \Pi_{\psi}(\U, \id) $ does not contain any unramified representations.

We now consider the case of general $\psi \in \Psi(\GU^*_{\Q_v}(n), \id)$. This follows from the fact that $I^{\GU}_P(\pi)$ is $\GU(\Q_v)$-spherical if and only if $\pi$ is $M(\Q_v)$-spherical for $M$ a standard Levi subgroup with parabolic subgroup $P$.
\end{proof}

\subsection{On the hypothesis \texorpdfstring{$ST^{\mH}_{\el} (f^{\mH}) = ST^{\mH}_{\disc} (f^{\mH}) $}{STell=STdisc} }{\label{STell=STdisc}}

In this section, we prove that for $(\mH,s,\eta)$ a refined elliptic endoscopic datum of $\GU=\GU(V)$ and $f^{\mH} \in C^{\infty}_c(\GU(\A))$ that is stable cuspidal at infinity and cuspidal at a finite place $v$, we have an equality of traces:
\begin{equation*}
    ST^{\mH}_{\el}(f^{\mH}) = ST^{\mH}_{\disc}(f^{\mH}).
\end{equation*}

We begin with some preparatory notation and lemmas. Let $\G$ be a connected reductive group defined over $\Q$ and let $\nu$ be a sufficiently regular (in the sense of Lemma \ref{generic}) quasi-character of $A_{\G}(\R)^0 $ and $ C_c^{\infty} (G(\R), \nu^{-1}) $ be the set of functions $ f_{\infty} : G(\R) \longrightarrow \C $ smooth, with compact support modulo $ A_G(\R)^0 $ and such that for every $ (z, g) \in A_G(\R)^0 \times G(\R), f_{\infty} (zg) = \nu^{-1}(z)f_{\infty} (g) $. Fix $K_{\G}$ a maximal compact subgroup of $\G(\R)$.

\begin{definition} (Stable cuspidal function at infinity)
We say that $ f_{\infty} \in C_c^{\infty} (\G(\R), \nu^{-1}) $ is stable cuspidal if $ f_{\infty} $ is left and right $K_{\G}$-finite and if the function 
\[
\Pi_{\text{temp}} (\G(\R), \nu) \longrightarrow \C, \quad \pi \longmapsto \tr( \pi \mid f_{\infty} )
\]
vanishes outside $\Pi_{\text{disc}} (\G(\R))$ and is constant in the $L$-packets of $\Pi_{\text{disc}} (\G(\R), \nu)$.
\end{definition}
\begin{definition}{(cuspidal function)}
We say that $ f_v \in C_c^{\infty} (\G_v(\Q_v), \nu^{-1}) $ is cuspidal if $ f_v$ if for each proper Levi subgroup we have that the constant term, $f_{v, M}$, vanishes (as defined in \cite[(7.13.2)]{GKM}).
\end{definition}

We record the following well-known lemma
\begin{lemma}{\label{cusptransferlem}}
If $f_{\infty} \in C_c^{\infty} (\G(\R), \nu^{-1}) $ is a stable cuspidal function and $(\mH, s, \eta)$ is an endoscopic triple of $\G$ then there exists a stable cuspidal transfer function $ f^{\mH}_{\infty} \in C_c^{\infty} (\mH(\R), \nu^{-1}) $ of $ f_{\infty} $.
\end{lemma}
\begin{proof}[Proof sketch]
By \cite{Shel} we can find a function $f^{\mH}_{\infty} \in C^{\infty}_c(\mH(\R), \nu^{-1})$ that transfers to $f_{\infty}$. Define the function $ F $ on the set of unitary tempered representations of $ \mH(\R) $ by setting  
\begin{equation*}
    \displaystyle F(\pi) = \dfrac{1}{ | \Pi_{\phi} (\mH(\R)) | } \sum_{\pi' \in \Pi_{\phi} (\mH(\R)) } \tr ( \pi' | f^{\mH}_{\infty}),
\end{equation*}
for $ \pi \in \Pi_{\phi} (\mH(\R)) $. Then $F$ must be supported on finitely many discrete series packets since $f_{\infty}$ is stable cuspidal and $(\mH, s, \eta)$ is elliptic. Hence, by \cite[Theorem 1]{CD90} there exists a function $f'^{\mH}_{\infty}\in C^{\infty}_c(\mH(\R), \nu^{-1})$ that is stable cuspidal and $F(\pi) = \tr ( \pi | f'^{\mH}_{\infty}) $. Thus, $f'^{\mH}_{\infty}$ has the same stable orbital integrals as $f^{\mH}_{\infty}$. This implies that $f'^{\mH}_{\infty}$ is a stable cuspidal transfer of $f_{\infty}$. 
\end{proof}
We recall that $ST^{\mH}_{\el}(f^{\mH})$  is defined by the formula
\begin{equation}
    ST^{\mH}_{\el}(h) := \sum\limits_{\gamma_{\mH}} \tau(\mH)SO_{\gamma_{\mH}}(h),
\end{equation}
where the sum is over a set of representatives of the $(\GU, \mH)$-regular, semisimple, $\Q$-elliptic, stable conjugacy classes in $\mH(\Q)$.
\begin{definition}
 We define the term $ST^{\mH}_{\disc}(f^{\mH})$  to equal
\begin{equation*}
    \sum\limits_{\psi \in \Psi_2(\mH)} \frac{1}{|\ov{\mc{S}}_{\psi}|} \sum\limits_{ \pi \in \Pi_{\psi}(\mH, \nu)} \langle 1, \pi \rangle \tr( \pi \mid f^{\mH}).
\end{equation*} 
\end{definition}
Note that we have suppressed the term $\epsilon_{\psi}(s_{\psi})$ from this expression because our assumption on $\nu$ implies that all $\psi$ are generic by Lemma \ref{generic}.

Separately, we have for every Levi subgroup $\M$ of $\mH$ the term $ST^{\mH}_{\M}$ defined in \cite[pg 86]{Mo} as well as the term $ST^{\mH}$ defined by
\begin{equation*}
    ST^{\mH} := \sum\limits_{\M} (n^{\mH}_{\M})^{-1}ST^{\mH}_{\M},
\end{equation*}
for certain constants $(n^{\mH}_{\M})^{-1}$.

We prove the following standard result.
\begin{lemma}{\label{cuspstablem}}
Suppose $h \in \mH(\A)$ is stable cuspidal at infinity and cuspidal at a finite place. Then
\begin{itemize}
\item For any $\M \neq \mH$ we have 
\begin{equation*}
    ST^{\mH}_{\M}(h) = 0.
\end{equation*}
\item If $\M = \mH$ then
\begin{equation*}
    ST^{\mH}_{\mH}(h)= ST^{\mH}_{\el}(h).
\end{equation*}
\end{itemize}
\end{lemma}
\begin{proof}
To prove the first part, we note that by definition, for $\M$ a proper Levi subgroup, the ``constant term'' $h^{\infty}_{\M}$ is $0$ (for instance see the definition before \cite[Theorem 7.1]{Art2}). This implies that $ST^{\mH}_{\M}(h)=0$.

We now prove the second part. We first show that $S\Phi_{\mH}(\gamma_{\mH}, h_{\infty})=SO_{\gamma_{\mH}}(h_{\infty})$. By \cite[Theorem 5.1]{Art3}, we have that
\begin{equation*}
    O_{\gamma_{\mH}}(h_{\infty}) = \Phi_{\mH}(\gamma_{\mH}, h_{\infty}) =  v(I_{\gamma_{\mH}})^{-1} \sum\limits_{\Pi} \Phi_{\mH}(\gamma^{-1}_{\mH}, \Pi)\tr( \pi \mid h_{\infty}),
\end{equation*}
where the sum is over discrete series $L$-packets of $\mH(\R)$ with central character $\nu_{\mH}$ (the unique character of $A_{\mH}(\R)^{0}$ such that if a parameter $\Psi_{\mH}$ has central character restricting to $\nu_{\mH}$ then $\Leta \circ \Psi_{\mH}$ has central character $\nu$). The representation $\pi$ is some representative of $\Pi$, and the value of $\tr(\pi \mid h_{\infty})$ does not depend on the choice of representative since $h_{\infty}$ is stable cuspidal. The $\gamma^{-1}$ in this formula that is seemingly at odds with the formula of Arthur is explained by \cite[7.19]{GKM}.

Therefore we have 
\begin{equation*}
    SO_{\gamma_{\mH}}(h_{\infty}) = \sum\limits_{ \gamma'_{\mH} \sim_{st} \gamma_{\mH} } e(I_{\gamma'_{\mH}}) \Phi_{\mH}(\gamma_{\mH}, h_{\infty}).
\end{equation*}
Now by definition,
\begin{equation*}
    S\Phi_{\mH}(\gamma_{\mH}, h_{\infty}) =  \ov{v}(I_{\gamma_{\mH}})^{-1}  \sum\limits_{\Pi} \Phi_{\mH}(\gamma^{-1}_{\mH}, \Pi)\tr( \Pi \mid h_{\infty}).
\end{equation*}
Since $h_{\infty}$ is stable cuspidal, we have $\tr( \Pi \mid h_{\infty}) = |\Pi | \tr( \pi \mid h_{\infty} )$. Furthermore, it follows from the definitions and basic properties of the Kottwitz sign, that
\begin{equation*}
    e(I_{\gamma_{\mH}})\ov{v}(I_{\gamma_{\mH}}) = (-1)^{q(I_{\gamma_{\mH}})} \Vol(\ov{I_{\gamma_{\mH}}}(\R) /A_{\mH}(\R)^0) = v(I_{\gamma_{\mH}})d(I_{\gamma_{\mH}}),
\end{equation*}
where $d(I_{\gamma_{\mH}}) = |\ker( H^1(\R, T) \to H^1(\R, I_{\gamma_{\mH}})|$ for $T$ an elliptic maximal torus of $I_{\gamma_{\mH}}$. 

% We now claim that 
% \begin{equation*}
%     | \Pi | = | [\gamma_{\mH}] | \cdot d(I_{\gamma_{\mH}}),
% \end{equation*}
% where $[\gamma_{\mH}]$ denotes the set of conjugacy classes in the stable conjugacy class of $\gamma_{\mH}$.

% Indeed, it is a standard fact that $| \Pi | = W_{\C}(\mH) /W_{\R}(\mH)$ where $W_F(\mH) := N_{\mH}(T)(F) / T(F)$ and also that this latter set is in bijection with $\ker(H^1(\R, T) \to H^1(\R, \mH))$. Hence the desired equality becomes
% \begin{equation*}
%     |\ker(H^1(\R, T) \to H^1(\R, \mH))| = |\ker(H^1(\R, T) \to H^1(\R, I_{\gamma_{\mH}}))| \cdot |\ker(H^1(\R, I_{\gamma_{\mH}}) \to H^1(\R, \mH))|.
% \end{equation*}
% This follows from the description of these cohomology sets given in \cite[Theorem 9]{Borov1}. Indeed, the left hand side is precisely the cardinality of the identity orbit of the action of $W_{\R}(\mH)$ on $H^1(\R, T)$ and the right hand side corresponds to the decomposition of this identity orbit first by $W_{\R}(I_{\gamma_{\mH}}) \subset W_{\R}(\mH)$

Finally, we put everything together to get
\begin{align*}
     SO_{\gamma_{\mH}}(h_{\infty}) &= \sum\limits_{ \gamma'_{\mH} \sim_{st} \gamma_{\mH} } e(I_{\gamma'_{\mH}}) \Phi_{\mH}(\gamma'_{\mH}, h_{\infty})\\
     &= \sum\limits_{ \gamma'_{\mH} \sim_{st} \gamma_{\mH} } \frac{d(I_{\gamma'_{\mH}})}{\ov{v}(I_{\gamma'_{\mH}})} \sum\limits_{\Pi} \Phi_{\mH}({\gamma'}^{-1}_{\mH}, \Pi)\tr(\pi \mid h_{\infty})\\
     &= \sum\limits_{ \gamma'_{\mH} \sim_{st} \gamma_{\mH} } \frac{d(I_{\gamma'_{\mH}})}{| \Pi|}  S\Phi_{\mH}(\gamma'_{\mH}, h_{\infty})\\
     &= S\Phi_{\mH}(\gamma_{\mH}, h_{\infty}).
\end{align*}
The last equality follows from the fact that $S\Phi_{\mH}(\gamma_{\mH}, h_{\infty})$ only depends on the stable class of $\gamma_{\mH}$ and
\begin{equation*}
    \sum\limits_{\gamma'_{\mH} \sim_{st} \gamma_{\mH}}  \frac{d(I_{\gamma'_{\mH}})}{| \Pi|}=1.
\end{equation*}
Indeed, $|\Pi|$ is well known to equal $|\ker(H^1(\R, T) \to H^1(\R, \mH))|$ for $T$ an elliptic maximal torus of $\mH$. Hence, it suffices to show that
\begin{equation*}
    \sum\limits_{\gamma'_{\mH} \sim_{st} \gamma_{\mH}}  d(I_{\gamma'_{\mH}})=|\ker(H^1(\R, T) \to H^1(\R, \mH))|.
\end{equation*}
To see this, first note that the set of conjugacy classes that are stably conjugate to $\gamma_{\mH}$ is in natural bijection with $\ker(H^1(\R, I_{\gamma_{\mH}}) \to H^1(\R, \mH))$. For each such conjugacy class, we can choose a representative $\gamma'_{\mH} \in T$. This follows from the fact that since $\mH$ contains an elliptic maximal torus, any elliptic element of $\mH(\R)$ is contained in an elliptic maximal torus and all elliptic maximal tori are conjugate in $H(\R)$. Then the set of classes in $H^1(\R, T)$ mapping to the class of $\gamma'_{\mH}$ in $H^1(\R, I_{\gamma_{\mH}})$ is in bijection with $\ker(H^1(\R, T) \to H^1(\R, I_{\gamma'_{\mH}}))$.

It then follows that
\begin{equation*}
    ST^{\mH}_{\mH}(h)= \tau(\mH)\sum\limits_{\gamma_{\mH}} SO_{\gamma_{\mH}}(h),
\end{equation*}
where the sum is over stable conjugacy classes in $\mH(\Q)$ that are semisimple and elliptic in $\mH(\R)$.

Since $h_{\infty}$ is stable cuspidal, its orbital integrals vanish on $\gamma_{\mH}$ that are not elliptic at $\R$, so we may as well impose this condition. By \cite[Proposition 3.3.4, Remark 3.3.5]{Mo} we may also restrict the sum to $\gamma_{\mH}$ that are $(\GU, \mH)$-regular. We then see that this is equal to $ST^{\mH}_{\el}(h)$.
\end{proof}

Suppose now that $f \in \mc{H}(\GU(\A))$ is stable cuspidal at infinity and cuspidal at a finite place. Then by the above Lemma \ref{cusptransferlem} and \cite[Lemma 3.4]{Art1}, for each elliptic endoscopic datum $(\mH, s, \eta)$, we can find a function $f^{\mH}$ that is stable cuspidal at infinity, cuspidal at a finite place, and a transfer of $f$. 

Our proof of the main result of this section will be by induction. We now state the key formulas we will need.

First, we have the following theorem of Morel:
\begin{theorem}{See \cite[Theorem 5.4.1]{Mo} \label{morelinput}}
Let $\G$ be a connected reductive group. Let $f = f^{\infty} f_{\infty}$ where $ f_{\infty} \in C_c^{\infty} (\G(\R), \C) $ and $ f^{\infty} \in C_c^{\infty} (\G(\mathbb{A}_f), \C) $. Assume that $f_{\infty}$ is stable cuspidal and that for every $(\mH, s, \eta) \in \mathcal{E}(\G)$, there exists a transfer $f^{\mH}$ of $f$. Then: 
\[
T^{\G}(f) = \sum_{(\mH, s, \eta) \in \mathcal{E}(\G)} \iota(\G, \mH) ST^{\mH} (f^{\mH})
\]
where $ \mathcal{E}(\G) $ is the set of isomorphism classes of elliptic endoscopic triples in the sense of Kottwitz and we recall that $T^{\G}(f)$ is defined to be the trace of $f$ on $L^2_{\disc}(\G(\Q) \setminus \G(\A))$.
\end{theorem}

Now we fix an odd positive integer $n$. By Proposition \ref{itm: from U to GU} and Remark \ref{GU to products} we have the following formula for each group $\G'$ of the form $\G(\U^*(n_1) \times ... \times \U^*(n_k))$ such that $\sum\limits^k_{i=1} n_i=n$. We note that all such groups are quasisplit.

For a function $f^{\G'} \in \mc{H}(\G'(\A))$:
\begin{equation*}
    T^{\G'}(f^{\G'}) = \sum\limits_{\psi \in \Psi_2(\G')} \sum\limits_{ \pi \in \Pi_{\psi}(\G', \xi, 1)} \tr(\pi \mid f^{\G'}),
\end{equation*}
where $\Pi_{\psi}(\G', \xi, 1)$ is the subset of $\Pi_{\psi}(\G', \xi)$ containing those $\pi$ with trivial character $\langle \cdot , \pi \rangle$.

We now prove by induction that for each group $\G'$ that we consider and for each $f^{\G'} \in \mc{H}(\G'(\A))$ stable cuspidal at infinity, we have 
\begin{equation}
    ST^{\G'}(f^{\G'}) = ST^{\G'}_{\disc}(f^{\G'}).
\end{equation}
We induct on $\sum\limits^k_{i=1} n^2_i$. Hence, the base case is when each $n_i=1$. Such a group $\G'$ is a torus and hence has no non-trivial elliptic endoscopy. In particular, by Theorem \ref{morelinput} we have that
\begin{equation}
T^{\G'}(f^{\G'}) =ST^{\G'}(f^{\G'})
\end{equation}
and hence it suffices to show that $T^{\G'}(f^{\G'}) = ST^{\G'}_{\disc}(f^{\G'})$. By \ref{quadbij} since there is no non-trivial endoscopy, each $\ov{\mc{S}}_{\psi}=1$ and hence $\langle \cdot , \pi \rangle$ is the trivial character for all $\pi$. The result follows.

We now settle the inductive step. Suppose we have shown $ST^{\G'}(f^{\G'}) = ST^{\G'}_{\disc}(f^{\G'})$ for each $\G'$ satisfying $\sum\limits^k_{i=1} n^2_i \leq N$. Suppose that $\G'$ satisfies $\sum\limits^k_{i=1} n^2_i = N+1$. Pick a function $f^{\G'} \in \mc{H}(\G'(\A))$ that is stable cuspidal at infinity and for each elliptic endoscopic datum $(\mH, s, \eta)$ of $\G'$ we pick by Lemma \ref{cusptransferlem} a transfer $f^{\mH} \in \mc{H}(\mH(\A))$ that is stable cuspidal at infinity.

Then we can write Theorem \ref{morelinput} in the form
\begin{equation*}
    T^{\G'}(f^{\G'}) = ST^{\G'}(f^{\G'}) + \sum_{(\mH, s, \eta) \in \mathcal{E}(\G')} \iota(\G', \mH) ST^{\mH} (f^{\mH}),
\end{equation*}
where for each nontrivial elliptic endoscopic group $\mH$ appearing in the sum on right-hand side, we have verified $ST^{\mH}(f^{\mH}) = ST^{\mH}_{\disc}(f^{\mH})$ by inductive assumption. 

To conclude, it suffices to show that we have an equality
\begin{equation*}
    T^{\G'}(f^{\G'}) = ST^{\G'}_{\disc}(f^{\G'}) + \sum_{(\mH, s, \eta) \in \mathcal{E}(\G')} \iota(\G', \mH) ST^{\mH}_{\disc} (f^{\mH}).
\end{equation*}
We prove this by arguing as in \cite[pg30]{Tai1} (cf \cite[\S12]{Kot5}). 
Indeed, we have
\begin{equation*}
    \sum_{(\mH, s, \eta) \in \mathcal{E}(\G')} \iota(\G', \mH) ST^{\mH}_{\disc} (f^{\mH}) =  \sum_{(\mH, s, \eta) \in \mathcal{E}(\G')} \iota(\G', \mH)\sum\limits_{\psi \in \Psi_2(\mH^*)} \frac{1}{|\ov{\mc{S}}_{\psi}|} \sum\limits_{ \pi \in \Pi_{\psi}(\mH, \nu)} \langle 1, \pi \rangle \tr( \pi \mid f^{\mH}).
\end{equation*}
Now, we apply at each place the endoscopic character identity we proved in Section \ref{sectionECIGU} and argue as for the equation \cite[(11), pg30]{Tai1} to get that the above equals
\begin{equation*}
 \sum\limits_{\psi \in \Psi_2(\G'^*)}\sum\limits_{s \in \ov{\mc{S}}_{\psi}} \frac{1}{|\ov{\mc{S}}_{\psi}|} \sum\limits_{ \pi \in \Pi_{\psi}(\G', \nu)} \langle s, \pi \rangle \tr( \pi \mid f^{\G'}).
\end{equation*}
Now we use that
\begin{equation*}
    \sum\limits_{s \in \ov{\mc{S}}_{\psi}} \frac{1}{|\ov{\mc{S}}_{\psi}|}  \langle s, \pi \rangle 
\end{equation*}
is $1$ if $\pi \in  \Pi_{\psi}(\G, \nu,1)$ and $0$ otherwise to get that the above equals
\begin{equation*}
 \sum\limits_{\psi \in \Psi_2(\G'^*)} \sum\limits_{ \pi \in \Pi_{\psi}(\G', \nu,1)} \tr( \pi \mid f^{\G'}),
\end{equation*}

which equals $T^{\G'}(f^{\G'})$ as desired.

\subsection{Some special global liftings} \textbf{} \label{global liftings}

We remind that parameters with a dot above are global parameters. Now consider $ \widetilde{\psi} : W_{\Q_v} \longrightarrow \prescript{L}{}{\GU^*_{\Q_v}(n)} = \big( \GL_n(\C) \times \C^{\times} \big) \rtimes W_{\Q_v} $ a supercuspidal $L$-parameter. We denote $\psi$ the $L$-parameter of $\U^*_{{\Q_v}}(n)$ obtained from $\widetilde{\psi}$ by the projection $ \prescript{L}{}{\GU^*_{\Q_v}(n)} \longrightarrow \prescript{L}{}{\U^*_{\Q_v}(n)}$.  There is a (standard) base change morphism :
\begin{equation} \phantomsection \label{itm : base change map}
  \eta_B : \Phi(\GU^*_{\Q_v}(n)) \longrightarrow \Phi(\GL_{E_v}(n) \times \Gm ).  
\end{equation}

Denote by $ \widetilde{\psi}^n $ the image of $ \widetilde{\psi} $ by this morphism. Then $\widetilde{\psi}^n$ is just the restriction of $ \widetilde{\psi} $ to $ W_{E_v} $. Since $ W_{E_v} $ acts trivially on $\GL_n(\C) \times \C^{\times}$, if we denote $\psi^n$ the projection of $\widetilde{\psi}^n$ to $\GL_n(\C)$, then it is an $n$ dimensional representation of $W_{E_v}$ and moreover $\psi^n$ is the the image of $\psi$ by the (standard) base change morphism. 

Since $\psi$ is a supercuspidal $L$-parameter (in particular, a discrete $L$-parameter), the group $ S_{\psi} $ is finite (see \cite[Lemma 10.3.1]{Kot5}) and we can write $ \psi^n = \psi^{n_1}_1 \oplus \cdots \oplus \psi^{n_r}_r $ where the $\psi^{n_i}_i $ are irreducible and pairwise distinct. By the computation in \cite[page 62, 63]{KMSW}, all the $ \psi^{n_i}_i $ are conjugate-orthogonal and we have
\[
S_{\psi} \simeq S^{\natural}_{\psi} \simeq \prod_{i = 1}^r O(1, \C) \simeq \prod_{i = 1}^r \Z / 2\Z.
\]
Moreover, the group $ Z(\widehat{\U}_{\Q_v}(n))^{\Gamma} = \big\{ \pm \Id \big\} $ embeds diagonally into $S_{\psi}$. Furthermore $ \det (- \Id) = -1 $ and $ S_{\psi} = \big\{ \pm \Id \big\} \times S^+_{\psi} $. By Lemma \ref{itm : centralizer of GU and U} we have
\[
\overline{\mathcal{S}}_{\widetilde{\psi}} \simeq \overline{\mathcal{S}}_{\psi} = S_{\psi} /  \big\{ \pm \Id \big\} \simeq \prod_{i = 1}^{r-1} \Z / 2\Z,
\]
and 
\[
S_{\widetilde{\psi}} \simeq S^{\natural}_{\widetilde{\psi}} \simeq S^+_{\psi} \times \C^{\times} \simeq \Big( \prod_{i = 1}^{r-1} \Z / 2\Z \times \Big) \times \C^{\times}.
\]

Let $ \widetilde{\Dot{\psi}} = (\Dot{\psi}, \chi)$ be a discrete global $L$-parameter of $\GU(\mathbb{A})$. The corresponding $L$-packet consists of automorphic representations of $\GU(\mathbb{A})$ whose central character is $\chi$ and whose restriction to $\U(\mathbb{A})$ is an automorphic representation in the $L$-packet of $\Dot{\psi}$. Again, we denote by $\Dot{\psi}^n= \Pi_1 \boxplus \cdots \boxplus \Pi_m $ the isobaric sum of automorphic representations of $\GL_n(\mathbb{A}_E)$ corresponding to $\Dot{\psi}$. As in the local case, we see that $S_{\Dot{\psi}}$ is finite and by \cite[page 69]{KMSW} we have then 
\[
S_{\Dot{\psi}} \simeq S^{\natural}_{\Dot{\psi}}  = \prod_{i = 1}^m O(1, \C) \simeq \prod_{i = 1}^m \Z / 2\Z
\]
with the group $ Z(\widehat{\U}_{\Q}(n))^{\Gamma} = \big\{ \pm \Id \big\} $ embedded diagonally into $S_{\Dot{\psi}}$ and an  isomorphism $ S_{\Dot{\psi}} \cong \big\{ \pm \Id \big\} \times S^+_{\Dot{\psi}} $. Thus
\[
\overline{\mathcal{S}}_{ \widetilde{\Dot{\psi}}} = \overline{\mathcal{S}}_{\Dot{\psi}} = S_{\Dot{\psi}} /  \big\{ \pm \Id \big\} \simeq \prod_{i = 1}^{m-1} \Z / 2\Z,
\]
\[
S_{\widetilde{\Dot{\psi}}} \simeq S^{\natural}_{\widetilde{\Dot{\psi}}} \simeq S^+_{\Dot{\psi}} \times \C^{\times} \simeq \Big(  \prod_{i = 1}^{m-1} \Z / 2\Z \Big) \times \C^{\times}.
\]

We say that a global parameter $ \widetilde{\Dot{\psi}} = (\Dot{\psi}, \chi)$ is a global lifting of $\widetilde{\psi}$ if we have $ (\Dot{\psi}_v, \chi_v) $ = $\widetilde{\psi}$ where $ (\Dot{\psi}_v, \chi_v) $ is the localization at $v$. In this case, there exist morphisms $ \lambda : S_{\Dot{\psi}} \longrightarrow S_{\psi} $, $ \widetilde{\lambda} : S_{\widetilde{\Dot{\psi}}} \longrightarrow S_{\widetilde{\psi}} $ and $ \overline{\lambda} : \overline{\mathcal{S}}_{\Dot{\psi}} \longrightarrow \overline{\mathcal{S}}_{\psi} $. Since the the local and global parameters $\psi$ and $\Dot{\psi}$ are discrete, these maps are injective (see \cite[page 28-31]{C.P.Mok} for more details). In this section, we construct some global liftings $ \widetilde{\Dot{\psi}} = (\Dot{\psi}, \chi) $ such that the above maps $\lambda$, $\widetilde{\lambda}$ and $\overline{\lambda}$ have some special properties. 

\subsubsection{First construction}{\label{construction1}} (\text{c.f. Lemma 4.2.1 in \cite{KMSW}})

We choose an auxiliary place $u$ of $\Q$ which splits over $E$ as $u= w\overline{w}$. Therefore $\U(\Q_u)$ is isomorphic to $\GL_n(E_w)$. By \cite[Theorem 5.7]{Shin}, there exists a cuspidal automorphic representation $\Pi$ of $\U(\mathbb{A})$ satisfying the following properties
\begin{enumerate}
    \item[$\bullet$] $\Pi_{\infty}$ is discrete series corresponding to a regular highest weight and with sufficiently regular infinitesimal character in the sense of \cite[Def. 2.2.10]{KH1},
    \item[$\bullet$] $\Pi_v$ belongs to the packet $\Pi_{\psi}(\U(\Q_v), \varrho_{\U})$,
    \item[$\bullet$] $ \Pi_u $ is a supercuspidal representation of $\GL_n(E_w)$.
\end{enumerate}
Note that such a $\Pi$ will be cohomological by the first condition and the remark at the end of \S$2$ of \cite{Kot2}.

By Lemma 4.1.2 of \cite{LJ}, we can extend $\Pi$ to an algebraic cuspidal automorphic representation $\overline{\Pi}$ of $\GU(\A)$. Furthermore, we can assume that $\overline{\Pi}$ is cohomological since $\Pi$ is.

Consider the exact sequence
\[
1 \longrightarrow \U \longrightarrow \GU \xrightarrow{c} \Gm \longrightarrow 1.
\]

Since $ \Pi_v $ belongs to the packet $\Pi_{\psi}(\U(\Q_v), \varrho_{\U})$, the central character $\omega_{\ov{\Pi}_v}$ and the central character $\omega_{\widetilde{\psi}}$ of any representation in $\Pi_{\widetilde{\psi}}(\GU(\Q_v), \varrho_{\GU})$ must agree on $Z(\GU) \cap \U$. The map $c$ restricted to $Z(\GU)$ has kernel equal to $Z(\GU) \cap \U$ so that $\omega_{\ov{\Pi}_v}\omega^{-1}_{\widetilde{\psi}}$ factors to give a character of $\im(c)$ which (since $n$ is odd) is the norm subgroup $N_{E^{\times}_v/\Q^{\times}_v} \subset \Q^{\times}_v$. We can choose a lift of this character to $\Q^{\times}_v$ and hence we conclude that there is some character $\omega: \Q^{\times}_v \to \C^{\times}$ such that $\overline{\Pi}_v \otimes (\omega \circ c) $ belongs to the packet $\Pi_{\widetilde{\psi}}(\GU(\Q_v), \varrho_{\GU})$.

There is an isomorphism of topological groups 
\begin{align*}
    \Q^{\times} \times \R_{ > 0 } \times \prod \Z_p^{\times} &\longrightarrow \Gm(\A) \\
    (r, t, (u_p)) &\longmapsto (rt, ru_2, ru_3, \cdots).
\end{align*}
 Then there is a character $ \overline{\Omega} $ of $ \Q^{\times} \times \R_{ > 0 } \times \prod \Z_p^{\times} $ such that $ \overline{\Omega} $ is trivial on $\Q^{\times} \times \R_{ > 0 }$, $ \overline{\Omega}_{| \Z_v^{\times}} \equiv \omega_{| \Z_v^{\times}} $ and $ \overline{\Omega} (-1, 1, (-1)) = 1 $. This character descends to a Hecke character $\Omega$ of $ \Gm(\Q) \backslash \Gm (\A) $ such that $ \Omega_v = \omega \otimes \kappa $ where $\kappa : \Q^{\times}_v \longrightarrow \C^{\times}$ is an unramified character and $ \Omega_{\infty} $ is trivial. In particular if we denote $ \widetilde{\Pi} := \overline{\Pi} \otimes ( \Omega \circ c ) $ then it is still cohomological (since $\overline{\Pi}$ is) and the local representation $\widetilde{\Pi}_v $ belongs to the packet $\Pi_{\widetilde{\psi}}(\GU(\Q_v), \varrho_{\GU})$ up to an unramified character twist.

Therefore the global parameter $\widetilde{\Dot{\psi}} = (\Dot{\psi}, \chi) $ is a globalisation of $ \widetilde{\psi} $, up to an unramified twist (where $\Dot{\psi}$ is the global parameter of $\Pi$ and $\chi$ corresponds to the central character of $\widetilde{\Pi}$). Since $\Pi_{\infty}$ has sufficiently regular infinitesimal character, $ \Dot{\psi} $ is generic (Lemma \ref{generic}). The last condition implies that $\Dot{\psi}^n$ is a cuspidal automorphic representation of $\GL_n(\mathbb{A}_E)$ which is self-dual and conjugate orthogonal. Therefore we have $ S_{\Dot{\psi}} = \{ \pm \Id \} $ (\cite[page 69]{KMSW}) so $\overline{\mathcal{S}}_{\Dot{\psi}} = \big\{ \Id \big\}$. The above second condition implies that $\Dot{\psi}$ is a global lift of $\psi$. Since the map $\lambda$ is injective, we see that $\lambda$ is the diagonal embedding of $\big\{ \pm \Id \big\} $ into $S_{\psi}$. 

Moreover since $\overline{\mathcal{S}}_{\widetilde{\psi}} = \overline{\mathcal{S}}_{\psi} = \big\{ \Id \big\} $ and $\overline{\mathcal{S}}_{\widetilde{\Dot{\psi}}} = \overline{\mathcal{S}}_{\Dot{\psi}}$, the map $\overline{\lambda}$ is the trivial map. The group $S^+_{\Dot{\psi}}$ is also trivial and the map $ \tilde{\lambda} $ is given by

\begin{align*}
 S_{\widetilde{\Dot{\psi}}} \simeq \C^{\times} & \longrightarrow  S_{\widetilde{\psi}} \simeq S^+_{\psi} \times \C^{\times}  & \\
t \quad & \longmapsto \quad \quad ( \Id, t ). &   
\end{align*}

\subsubsection{Second construction}{\label{construction2}}. (\text{We adapt the proof of Lemma 4.4.1 in \cite{KMSW}})

Consider an element $ \displaystyle s = \big( x_i \big)_{i=1}^r \in S_{\psi} = \prod_{i = 1}^r \Z / 2\Z$ whose image in $\overline{\mc{S}_{\psi}}$ is denoted by $\overline{s}$. We can suppose that $ x_i = 1$ for $ i \in X \subset \{ 1, \cdots, r \} $ and $x_i = -1$ for $i \in Y \subset \{ 1, \cdots, r \} $. Denote $ \displaystyle \psi^{n_X}_X = \bigoplus_{i \in X} \psi^{n_i}_i $ and $ \displaystyle \psi^{n_Y}_Y = \bigoplus_{i \in Y} \psi^{n_i}_i $ (where $ \displaystyle n_X = \sum_{i \in X} n_i $ and $ \displaystyle n_Y = \sum_{i \in Y} n_i $). Since all the $\psi^{n_i}_i$ are conjugate orthogonal, by \cite[Lemma 2.2.1]{C.P.Mok}, the $L$-parameters $\psi^{n_X}_X$ resp. $\psi^{n_Y}_Y$ come from $L$-parameters $\psi_X$ resp. $\psi_Y$ of unitary groups $ \U_{\Q_v}(n_X) $ resp. $ \U_{\Q_v}(n_Y) $ by the base change map $\eta_B$ (see (\ref{itm : base change map})). Now as in the first construction, for these $L$-parameters we can construct cuspidal automorphic representations $ \Pi^{n_X}_X $ resp. $\Pi^{n_Y}_Y$, of $\GL_{n_X}(\mathbb{A}_E)$ resp. $\GL_{n_Y}(\mathbb{A}_E)$. Since these automorphic representations are self-dual and conjugate-orthogonal, the isobaric sum $ \Pi^{n_X}_X \boxplus \Pi^{n_Y}_Y $ factors through the base change map $\eta_B$ (\cite[Proposition 1.3.1]{KMSW}, \cite[page 27]{C.P.Mok}). Denote this global $L$-parameter of $\U(\mathbb{A})$ by $\Dot{\psi}$. Again by \cite[page 69]{KMSW} we know that $ \displaystyle S_{\Dot{\psi}} \simeq \prod_{ i \in \{ X, Y \} } \Z / 2\Z  $. As in the first construction, the $L$-parameter $\Dot{\psi}$ is generic (Lemma \ref{generic}) and is a global lift of $\psi$. Moreover the localization map $\lambda$ is defined as follows
\begin{align*}
 S_{\Dot{\psi}} \quad \quad & \longrightarrow \quad \quad \quad S_{\psi} & \\
(x_1, x_2) \quad & \longmapsto  (\underbrace{x_1, \cdots x_1}_{i \in X}, \underbrace{x_2, \cdots, x_2}_{i \in Y}). &   
\end{align*}

Now, taking the quotient by $ \big\{ \pm \Id \big\} $ we see that $ \overline{\mathcal{S}}_{\Dot{\psi}} = S_{\Dot{\psi}} /  \big\{ \pm \Id \big\} \simeq \Z / 2\Z $ and the map $\overline{\lambda}$ is given by : 
\begin{align*}
 \overline{\mathcal{S}}_{\Dot{\psi}} \quad \quad & \longrightarrow \quad  \overline{\mathcal{S}}_{\psi} & \\
(-1) \quad & \longmapsto \quad \overline{s} &   
\end{align*}

Now take an automorphic representation $\Pi$ of $\U(\A)$ in the packet $\Pi_{\Dot{\psi}}(\U, \varrho_{\U})$. By the same argument as in the first construction, we can extend it to an automorphic representation $\widetilde{\Pi}$ of $\GU(\A)$ such that $\widetilde{\Pi}_v $ belongs to the packet $\Pi_{\widetilde{\psi}}(\GU(\Q_v), \varrho_{\GU})$ up to an un-ramified twist. Thus the global parameter $\widetilde{\Dot{\psi}}$ of $\widetilde{\Pi}$ is a globalisation of $\widetilde{\psi}$. We have then $ S^+_{\Dot{\psi}} \simeq \Z \slash 2 \Z $ and $ S_{\widetilde{\Dot{\psi}}} = S^+_{\Dot{\psi}} \times \C^{\times} $.

Furthermore, if the element $s$ belongs to $ S^+_{\psi} $ then $(x_1, x_2)$ belongs to $S^+_{\Dot{\psi}}$ since the map $\lambda$ is injective and restricts to a map from $S^+_{\Dot{\psi}}$ to $ S^+_{\psi} $. Therefore, we have the following description of the map $\widetilde{\lambda}$ 

\begin{align*}
 S_{ \widetilde{\Dot{\psi}}} \simeq \Z \slash 2 \Z \times \C^{\times} \quad \quad & \longrightarrow \quad \quad \quad S_{\widetilde{\psi}} \simeq S^+_{\psi} \times \C^{\times} & \\
 1 \times t \quad \quad \quad  & \longmapsto \quad \quad \quad \quad \quad \quad 1 \times t \\
 -1 \times t \quad \quad \quad  & \longmapsto  \quad \quad (\underbrace{x_1, \cdots x_1}_{i \in X}, \underbrace{x_2, \cdots, x_2}_{i \in Y}) \times t . &   
\end{align*}

\subsection{Galois representations associated to global cohomological generic parameters} {\label{globalgaloisrep}}

We have fixed a quadratic imaginary extension $E$ of $\Q$. In this subsection, we associate representations of $\Gamma_{E}$ to certain global parameters. 

Let $( \Dot{\psi}, \chi )$ be a global parameter of a global unitary similitude group $\GU$. In particular $\Dot{\psi}$ is a global parameter for the corresponding unitary group $\U$. We suppose further that the localization at infinity  $( \Dot{\psi}_{\infty}, \chi_{\infty} )$ is regular and sufficiently regular so that $\Dot{\psi}$ will be generic.

We first associate a $\Gamma_E$ representation to $\Dot{\psi}$. Associated to $\Dot{\psi}$, we have the quadratic base change, $\Dot{\phi}^n$, which is an automorphic representation of $\GL_n(\A_E)$. Since the global parameter is generic, the representation $\Dot{\psi}^n$ is of the form $ \Pi_1 \boxplus \cdots \boxplus \Pi_k $ where $\Pi_i$ are self dual cuspidal generic and cohomological automorphic. Now, fix a place $\ell$ of $\Q$ and an isomorphism $\iota_{\ell}: \ov{\Q_{\ell}} \to \C$. Then by \cite[Theorem 1.2]{Shi2}, for each representation $\Pi_i$ there is a unique $\ell$-adic $\Gamma_E$-representation $\rho^i$ such that for each place $ \mathcal{P} $ of $E$ not dividing $\ell$, we have the following isomorphism of Weil-Deligne representations
\[
\WD(\rho^i|_{\Gamma_{E_{\mathcal{P}}}})^{F-ss} \cong \iota^{-1}_{\ell}\mathcal{L}((\Pi_i)_{\mathcal{P}}),
\]
where $\mathcal{L}((\Pi_i)_{\mathcal{P}})$ is the local parameter associated to $(\Pi_i)_{\mathcal{P}}$ under the local Langlands correspondence.

Similarly, if we denote $\rho = \rho_1 \oplus \cdots \oplus \rho_k $, then for each place $ \mathcal{P} $ dividing $q$ and not dividing $\ell$, we have that
\begin{equation*}
    \WD(\rho|_{\Gamma_{E_{\mathcal{P}}}})^{F-ss} = \iota^{-1}_{\ell}\mathcal{L}( (\Pi_1 \boxplus \cdots \boxplus \Pi_k)_{\mathcal{P}}).
\end{equation*}

Denote by $ \Dot{\psi}_{\mathcal{P}} $ the localization of $ \Dot{\psi} $ at $ \mathcal{P} $. By the definition of localization map of global parameters (\cite[p. 18, 19]{C.P.Mok}), we see that the local $L$-parameter (not necessarily bounded) corresponding to $\WD(\rho|_{ \Gamma_{E_{\mc{P}}}})^{F-ss}$ is $ \Dot{\psi}_{\mathcal{P}} $ if $q$ is split in $E$. If $q$ is inert in $E$ then $ q = \mathcal{P} $ and $E_{\mc{P}}$ is a quadratic extension of $ \Q_q $. In this case $\WD(\rho|_{\Gamma_{E_{\mc{P}}}})^{F-ss}$ corresponds to the image of $ \Dot{\psi}_{\mathcal{P}} $ via the base change map $\eta_B$ and equals $ \Dot{\psi}_{\mathcal{P}|_{ \mc{L}_{E_{\mathcal{P}}}}} $.

The central character $\chi$ gives rise to a character of $\GL_1(\A_E)$ and hence an $\ell$-adic character $\chi'$. The pair $(\rho, \chi')$ then gives us a morphism
\[
\widetilde{\rho} : \Gamma_E \longrightarrow \GL_n(\overline{\Q}_{\ell}) \times \overline{\Q}^{\times}_{\ell}.
\]

From the local-global compatibility properties of $\rho$, we conclude that for every place $\mathcal{P}$ dividing a prime $ q \neq \ell$, the restriction $\widetilde{\rho}_{| W_{E_{\mathcal{P}}}}$ equals $ (\psi_q, \chi_q)_{| W_{E_{\mathcal{P}}}} $ where $(\psi_q, \chi_q)$ is the localization of the global parameter $(\Dot{\psi}, \chi)$ at the prime $q$.

%Consider an element $ \displaystyle s = \big( x_i \big)_{i=1}^r \in S_{\overline{\phi}} = \prod_{i = 1}^r \Z / 2\Z$. We can suppose that $ x_i = 1$ for $ i \in I $ and $x_i = -1$ for $i \in J $. Denote $ \displaystyle \overline{\phi}^n_1 = \bigoplus_{i \in I} \overline{\phi}^n_i $ and $ \displaystyle \overline{\phi}^n_2 = \bigoplus_{i \in J} \overline{\phi}^n_i $. As in the first construction, there exists automorphic representations $\Pi^1$ and $\Pi^2$ of $GL_n(\mathbb{A})$ such that the local $L$-parameter of $\Pi^1_p$ and $\Pi^2_p$ are resp. $\overline{\phi}_1^n$ and $\overline{\phi}_2^n$. Next, we lift the isobaric sum $ \Pi^1 \boxplus \Pi^2 $ to a $L$-parameter $(\Dot{\phi}, \chi)$ of $GU(\mathbb{A})$. In this case $ \displaystyle S_{\Dot{\phi}} \simeq \prod_{ i \in \{ 1,2 \} } \Z / 2\Z  $ and the map $\lambda$ is defined as follow
%\begin{align*}
% S_{\Dot{\phi}} \quad \quad & \longrightarrow \quad \quad \quad S_{\phi} & \\
%(x_1, x_2) \quad & \longmapsto  (\underbrace{x_1, \cdots x_1}_{i \in I}, \underbrace{x_2, \cdots, x_2}_{i \in J}) &   
%\end{align*}

%Now, taking the quotient by $ \big\{ \pm \Id \big\} $ we get the map $\overline{\lambda}$ 
%\begin{align*}
% \overline{\mathcal{S}}_{\Dot{\phi}} \quad \quad & \longrightarrow \quad  \overline{\mathcal{S}}_{\phi} & \\
%(-1) \quad & \longmapsto \quad s &   
%\end{align*}
\section{Rapoport--Zink spaces and an averaging formula}
\subsection{Rapoport--Zink spaces}
We continue with our fixed prime number $p$ as  before. Let $ \breve{\Q}_p := \widehat{\Q_p^{\text{nr}}} = \text{Frac}W(\overline{\F}_p) $ the completion of the maximal unramified extension of $\Q_p$ and $\sigma$ the geometric Frobenius automorphism of $\breve{\Q}_p / \Q_p$. 

We will be interested in the subset $\mb{B}(\Q_p. \G, \mu) $ of $\mb{B}(\Q_p, \G)$  associated with a minuscule cocharacter $ \mu: \Gm_{ / \overline{\Q}_p} \longrightarrow \G_{\overline{\Q}_p} $as defined in \cite[\S 6.2]{kot1}. The Bruhat ordering on the image of the Newton map induces a partial order on  $\mb{B}(\Q_p, \G, \mu)$.

\begin{definition} \label{itm: defPEL} A Rapoport--Zink data of simple unramified unitary PEL type $ (E_p, *, V, \langle \cdot | \cdot \rangle, \GU, \mu, b)$ consists of the following
	\begin{enumerate}
		\item[$\bullet$] an unramifed extension $E_p$ of degree $2$ of $\Q_p$ with a non trivial involution $*$, 
		\item[$\bullet$] a $E_p$-vector space $V$ of dimension $n$,
		\item[$\bullet$] a symplectic Hermitian form $ \langle \cdot | \cdot \rangle: V \times V \longrightarrow \Q_p $ for which there is a self-dual lattice $ \Lambda $,  
		\item[$\bullet$] a conjugacy class of minuscule cocharacters $ \mu: \Gm_{ \overline{\Q}_p} \longrightarrow \GU _ {\overline{\Q}_p} $ where $ \GU $ is the similitude unitary group defined over $ \Q $ by 
		\[
		\GU(R) = \big\{ g \in \GL(V \otimes R) | \langle gv, gw \rangle = c(g) \langle v,w \rangle, \ v,w \in V \otimes R \big\}
		\]
		for all $ \Q $-algebra $ R $ and $ c (g) \in R^{\times} $.
		\item[$\bullet$] a $\sigma$-conjugacy class  $ b \in \mb{B}(\Q_p, \GU, -\mu) $. We also suppose that $ c \circ \mu (z) = z $ where $ c $ is the similitude factor of $\GU$.
	\end{enumerate}
\end{definition}

Denote $ \mathrm{I}_{E_p} := \Hom_{\Qp} (E_p, \ov{\Q}_p) $ then the cocharacter $\mu$ is determined by the integral couples $ (p_{\tau}, q_{\tau})_{\tau \in \mathrm{I}_{E_p}} $ such that $ p_{\tau} + q_{\tau} = n $ and $ (p_{\tau}, q_{\tau}) = (q_{\tau^*}, p_{\tau^*}) $.

To such a data, we associate the isocrystal $ N = \Big (V \otimes_{\Q_p} \breve{\Q}_p, b \circ (\Id \otimes \sigma) \Big) $ with an action $ \iota: \mathcal{O}_{E_p} \longrightarrow \End (N) $ and an alternating non degenerate form $ \langle \cdot | \cdot \rangle: N \times N \longrightarrow \breve{\Q}_p (n) $ where $ n = val_p (c (b)) $. By Dieudonne's theory, the isocrystal $ N $ corresponds to a $p$-divisible group $ (\mathbb{X}, \iota, \lambda) $ defined over $ \overline {\F}_p $ provided with an action of $ \mathcal{O}_{E_p} $ and a polarization $ \lambda $.

\begin{theorem} \cite[Theorem 3.25]{RZ1}
	Let $\mathcal{M}$ be the functor associating to each $ \mathcal{O}_{\breve{\Q}_p}$ scheme $ S $ on which $ p $ is locally nilpotent the set of pairs $ ( X, \rho) $ where:
	\begin{enumerate}
		\item [-] $X$ is a $p$-divisible group over $S$ with a $p$-principle polarization $ \lambda_X $ and an action $ \iota_X $ such as the Rosati involution inducing by $ \lambda_X $ induces $ * $ on $ \mathcal{O}_{E_p} $. 
		\item [-] A $\mathcal{O}_{E_p} $-linear quasi-isogeny $ \rho: X \times_S \overline{S} \longrightarrow \mathbb{X} \times_{Spec (\overline {\F}_p)} \overline{S} $ such that $ \rho^{V} \circ \lambda_X \circ \rho $ is a $ \Q_p $-multiple of $ \lambda_X $ in $ \mathop {\mathrm{Hom}} \nolimits_{\mathcal{O}_{E_p}} (X, X^{V}) \otimes_{\Z} \Q $. (here, $ \overline{S} $ is the modulo $ p $ reduction of $ S $).	
	\end{enumerate}
	
	We also require that $ (X, \iota_X) $ satisfies the Kottwitz determinant condition. More precisely, under the action of $ E_p $, we have a decomposition: $ \Lie(X) = \bigoplus_{\tau} \Lie(X)_{\tau} $ then $ \Lie(X)_{\tau} $ is locally free of rank $ p_{\tau} $. This functor is then represented by a formal scheme %$ \mathcal{M} (\mu, b) $ 
	defined over $ \Spf(\mathcal{O}_{\breve {\Q}_p}) $.
\end{theorem}

In order to introduce the usual level structures, we work with the rigid generic fiber $\mathcal{M}^{\text{an}}$ of $\mathcal{M}$ over $\breve{\Q}_p$. We set $C_0 = \{g \in \GU(\Q_p) \ \vert \ g \Lambda = \Lambda \} $, a maximal compact subgroup of $ \GU (\Q_p) $.

\begin{definition}
	Let $\mathcal{T} / \mathcal{M}^{\text{an}} $ be the local system defined by the $p$-adic Tate module of the universal $p$-divisible group  on $ \mathcal{M} $. For $ K \subset C_0 $ we define $ \mathcal{M}_K $ as the etale covering of $ \mathcal{M}^{\text{an}} $ which classifies the $ \mathcal{O}_{E_p} $- trivializations  modulo $ K $ of $ \mathcal{T} $ by $ \Lambda $. We also require that the trivialization preserves the alternating form up to $ \Q_p^{\times} $. 
\end{definition}

We have, in particular, that $\mathcal{M}^{an} = \mathcal{M}_{C_0} $. We then get a tower $(\mathcal{M}_{K_p})_{K_p}$ of analytic spaces on $ \breve{\Q}_p$ provided with finite étale transition maps $\Phi_{K_p^{'}, K_p}: \ \mathcal{M}_{K^{'}_p} \ \longrightarrow \ \mathcal{M}_{K_p}$ (for $K_p^{'} \subset K_p $)
%\[
%\Phi_{K_p^{'}, K_p}: \ \mathcal{M}_{K^{'}_p} \ \longrightarrow \ \mathcal{M}_{K_p}
%\]
which forget the level structure. The map $ \Phi_{K_p^{'}, K_p} $ is Galois of Galois group $K_p / K_p^{'} $ if $ K_p^{'} $ is normal in $ K_p $. \\

Let $ \J_b(\Q_p) $ be the group of $ \mathcal{O}_{E_p} $-linear quasi-isogenies $g$ of $ \X $ such that $ \lambda \circ g $ is a $ \Q^ {\times} $- multiple of $ g^{\vee} \circ \lambda $. The group $ \J_b (\Q_p) $ acts on the left on $ \mathcal{M} $  by the formula
\[
\forall g \in \J_b(\Q_p) \ \forall (X, \rho) \in \ \mathcal{M}, \quad (X, \rho) \cdot g = (X, \rho \circ g^{- 1}).
\]
\textbf{}

We say that a simple unramified unitary Rapoport--Zink datum $ (E_p, *, V, \langle \cdot | \cdot \rangle, \GU, \mu, b) $ is basic if the associated group $\J_b (\Q_p) $ is an inner form of $\GU$. The above datum is basic if and only if $b$ is the unique minimal element in $ \mb{B} (\Q_p, \GU, \mu) $. In this case, we also say that $b$ is basic. \\

Let $ \ell \neq p$ be a prime number. Let $K_p \subset C_0 $ be a level. As in \cite[remark 2.6.3]{Far1} we denote:
	\[
	H_c^{\bullet}(\mathcal{M}_{K_p}, \overline{\Q}_{\ell}) := \mathop{\mathrm{lim}}_{\overrightarrow{V}} \mathop{\mathrm{lim}}_{\overleftarrow{n}} H_c^{\bullet}(V \otimes_{\breve{\Q}_p} \C_p, \Z / \ell^n \Z ) \otimes \overline{\Q}_{\ell}
	\]
	where $V$ runs through the relatively compact open subsets of $\mathcal{M}_{K_p}$. \\

The group $ \J_b(\Q_p) $ acts on $ \mathcal{M}_{C_0} $ and this action extends to $ \mathcal{M}_{K_p} $ so that $ \J_b(\Q_p) $ acts on $ H_c^{\bullet}(\mathcal{M}_ {K_p}, \overline{\Q}_{\ell}) $. Since $n$ is odd, the reflex field of the conjugacy class of $\mu$ is $E_p$. We can also define an action of the Weil group $ W_{E_p} $ on these cohomology groups thanks to the Rapoport--Zink descent data defined as below. 

Let $ \sigma_{E_p}: \breve{\Q}_p \xrightarrow {\sim} \breve {\Q}_p $ the relative Frobenius automorphism with respect to $E_p$. We denote by $ \overline {\sigma}_{E_p} $ the Frobenius morphism induced on $ \overline {\F}_p $. For $ \mathbb{X}$ a $p$-divisible group defined over $ \overline{\F}_p $, we note $ F_{E_p}: \mathbb{X} \longrightarrow \overline {\sigma}_{E_p}^ * \mathbb{X} $ the relative Frobenius morphism. We construct a functor isomorphism $ \alpha: \mathcal{M} \longrightarrow \sigma_{E_p}^ * \mathcal{M} $ as follows.

%\begin{enumerate}
%\item[-] $ X^{\alpha}: = X $ with the action of $ \iota_{X^{\alpha}}: = \iota_X $ and with the polarization $ \lambda_{X^{\alpha}}: = \lambda_X $,
%\item[-] $ \rho^{\alpha}: = \rho \circ F_{E_p}^{- 1}. $
%\end{enumerate}

For $S$ a $ \mathcal{O}_{\breve{\Q}_p} $ scheme on which $p$ is nilpotent as well as a point $ (X, \rho) \in \mathcal{M}( S) $, the point $ (X^{\alpha}, \rho^{\alpha}) $ associated in $ \sigma_{E_p}^* \mathcal{M}(S) $ is defined as follows: $ X^{\alpha}: = X $ with the action of $ \iota_{X^{\alpha}}: = \iota_X $, with the polarization $ \lambda_{X^{\alpha}}: = \lambda_X $ and $ \rho^{\alpha}: = \rho \circ F_{E_p}^{- 1}$. The isomorphism of functors $ \alpha: \mathcal{M} \longrightarrow \sigma_{E_p}^* \mathcal{M} $ is the Rapoport--Zink descent data associated with $ \mathcal{M} $. Since the descent data commute with the action of $ \J_b( \Q_p) $, the groups $ H_c^{\bullet}(\mathcal{M}_{K_p}, \overline{\Q}_{\ell}) $ has an action of $ \J_b(\Q_p) \times W_{E_p} $. In addition, when $ K_p $ varies, the system $ (H_c^{\bullet} (\mathcal{M}_{K_p}, \overline{\Q}_{\ell}))_{K_p} $  has an action of $ \GU(\Q_p ) $. Thus, this system has an action of $\GU(\Q_p ) \times \J_b (\Q_p) \times W_{E_p}$. Let $ \rho $ be an admissible $\ell$-adic representation of $ \J_b(\Q_p)$, we define 
\[
H^{i,j} (\GU, b, \mu)[\rho] := \mathop{\mathrm{lim}}_{\overrightarrow{K_p}} \Ext^j_{J_b(\Q_p)}(H_c^{i}(\mathcal{M}_{K_p}, \overline{\Q}_{\ell}), \rho).  
\]

By \cite[Theorem 8]{Man1}, the $H^{i,j} (\GU, b, \mu)[\rho]$ are admissible and are zero for almost all $ i, j \geq 0 $. Finally, we define the homomorphism of Grothendieck groups $\Mant_{\GU, b, \mu}: \Groth(\J_b(\Q_p)) \to \Groth(G(\Q_p) \times W_{E_{\mu}}$ by
\[
\Mant_{\GU, b, \mu} (\rho) := \sum_{i,j} (-1)^{i+j} H^{i,j} (\GU, b, \mu)[\rho](- \dim \mc{M}^{\an}).
\]
	
%\begin{proposition} \cite[Theorem 8]{Man1} \phantomsection \label{itm: RZ}
%Let $ \rho $ be an admissible $\ell$-adic representation of $ \J_b(\Q_p) $. Then
%\begin{enumerate}
%\item [-] The  $ \displaystyle H^{i, j}(\mathcal{M}^{\infty})_{\rho}: = \mathop{\mathrm{lim}}_{\overrightarrow{K}} \Ext^j_{\J_b(\Q_p)} (H^i(\mathcal{M}_K, %\Q_{\ell}(d_K)), \rho) $ are zero for almost all $ i, j \geq 0 $.
%\item [-] The representations $ H^{i, j}(\mathcal{M}^{\infty})_{\rho}$ are admissible.
%\end{enumerate}
 %\end{proposition}
 
\subsection{An averaging formula for the cohomology of Rapoport--Zink spaces}

In this section we deduce an averaging formula for the cohomology of Rapoport--Zink spaces using the results of \cite{BM2}.

We begin with some endoscopic preliminaries. To state the formula, we need the following notion of endoscopic data for Levi subgroups.
\begin{definition}{(cf. \cite[Definition 2.18]{BM2})}
    Let $\M \subset \G$ be a Levi subgroup. We say that $(\mH,{\mH}_{\M}, s, \eta)$ is an \emph{embedded endoscopic datum} of $\G$ relative to $\M$ if $(\mH,s,\eta)$ is a refined endoscopic datum of $\G$ and the restriction $({\mH}_{\M}, s, \eta|_{\widehat{{\mH}_{\M}}})$ gives a refined endoscopic datum of $\M$. 
    
    We say that two embedded endoscopic data $(\mH, {\mH}_{\M}, s, \eta)$ and $(\mH', {\mH'}_{\M}, s', \eta')$ are isomorphic if there exists an isomorphism $\alpha: \mH \to \mH'$ of refined endoscopic data $(\mH,s,\eta)$ and $(\mH', s', \eta')$ whose restriction $\alpha_{\M}$ to ${\mH}_{\M}$ gives an isomorphism of $({\mH}_{\M}, s, \eta)$ and $(\mH'_{\M}, s', \eta')$. We denote the set of isomorphism classes of embedded endoscopic data of $\G$ relative to $\M$ by $\mc{E}^e(\M,\G)$.

\end{definition}

We now fix a refined elliptic endoscopic datum $(\mH,s,\eta)$ of $\GU$. Note that for each standard Levi subgroup $\M \subset \G$, there is a natural forgetful map 
\begin{equation*}
    Y^e: \mc{E}^e(\M,\GU) \to \mc{E}^r(\GU).
\end{equation*}
We define $\mc{E}^i(\M,\GU;\mH)$ to be the set of embedded endoscopic data $(\mH', \mH'_{\M}, s', \eta')$ such that $\mH' = \mH$ and whose class lies in the fiber $(Y^e)^{-1}((\mH,s,\eta))$ modulo the relation that two data  $(\mH, \mH_{\M}, s, \eta)$ and $(\mH, \mH'_{\M}, s', \eta')$ are equivalent if there exists an inner automorphism $\alpha$ of $\mH$ inducing an isomorphism of the embedded endoscopic data.

Fix a maximal torus $\widehat{\T_{\mH}} \subset \widehat{\mH}$ and define $\widehat{\T} := \eta(\widehat{\T_{\mH}}) \subset \widehat{\GU}$. By the comment before \cite[Proposition 2.27]{BM2}, we have that the set $\mc{E}^i(\M,\GU;\mH)$ is parametrized by the set of double cosets $W(\widehat{\T}, \widehat{\M}) \backslash W(\M,\mH) / W(\widehat{{\T}_{\mH}}, \widehat{\mH})$ where $W(\widehat{\T}, \widehat{\M})$ and $W(\widehat{\T_{\mH}}, \widehat{\mH})$ are the Weyl groups of $\widehat{\M}$ and $\widehat{\mH}$ respectively and $W(\M,\mH)$ is defined in \cite[Definition 2.23]{BM2}.

Finally, for an inner form $\J$ of $\M$, we define the subset $\mc{E}^i_{\eff}(\J,\GU; \mH) \subset \mc{E}^i(\M, \GU; \mH)$ to consist of those equivalence class of endoscopic data $(\mH,\mH_{\M},s, \eta)$ such that there exists a maximal torus $\T_{\mH} \subset \mH$ that transfers to $\J$. 

We now fix $b \in \mb{B}(\Q_p, \GU,\mu)$ and let $\tilde{b} \in \GU(\breve{\Q}_p)$ be a decent lift. We get a standard Levi subgroup $\M_b$ of $\GU$ and an extended pure inner twist $\J_b$ of $\M_b$. Let $\nu_b: \D \to A_{\M_b}$ (where $A_{\M_b}$ is the maximal split torus in the center of $\M_b$) denote the image of the Newton map applied to $b$. Fix $(\mH,s,\LL\eta)$ an elliptic endoscopic group of $\GU$ and a set, $X^{\mf{e}}_{\J_b}$, of representatives of $\mc{E}^i_{\eff}(\J_b, \GU; \mH)$. Furthermore, for each $(\mH, \mH_{\M_b} , s, \eta) \in X^{\mf{e}}_{\J_b}$ we may choose an extension $\LL\eta: \LL \mH \to \LL \GU$ of $\eta$. We also get a natural map $A_{\M_b} \hookrightarrow A_{\mH_{\M_b}}$. Then we define $\nu$ to be the composition of $\nu_b$ with this map.

We then make the following definition.
\begin{definition}
  We define
    \begin{equation}
        \Red^{\mc{H}^{\mf{e}}}_b : \Gr^{st}(\mH(\Q_p)) \to \Gr(\J_b(\Q_p))
    \end{equation}
    by
    \begin{equation}
        \pi  \mapsto \sum\limits_{X^{\mf{e}}_{\J_b}}    \Trans^{\mH_{\M_b}}_{\J_b} (\Jac^{\mH}_{P(\nu)^{op}}(\pi)) \otimes \ov{\delta}^{1/2}_{P(\nu_b)},
    \end{equation}
    where $\Trans^{\mH_{\M_b}}_{\J_b}$ denotes the transfer of distributions from $\mH_{\M_b}(\Q_p)$ to $\J_b(\Q_p)$ and $\Groth(\J_b(\Q_p))$ denotes the Grothendieck group of admissible representations of $\J_b(\Q_p)$ and $\Groth^{st}(\mH(\Q_p))$ is the subgroup of $\Groth(\mH(\Q_p))$ consisting of those elements with stable distribution character.
\end{definition}

Our aim in this subsection is to establish the theorem below using the results of \cite{BM2}.

\begin{theorem}{\label{AvgFormula}}
Let $(\mH, s, \LL\eta)$ be a refined elliptic endoscopic datum of $\GU$. Let $\psi: W_{\Q_p} \to \LL \GU$ be a supercuspidal Langlands parameter such that there exists a Langlands parameter $\psi^{\mH}$ of $\mH$ with $\psi= \LL\eta \circ \phi^{\mH}$. Then we have the following equality in $\mathrm{Groth}(\GU(\Q_p) \times W_{E_{\mu}})$:

\begin{equation*}
\sum\limits_{b \in \mb{B}(\Q_p, \GU, - \mu)} \Mant_{\GU,b, \mu}(\Red^{\mc{H}^{\mf{e}}_b}(S\Theta_{\psi^{\mH}}))=
\end{equation*}
\begin{equation*}
 \sum\limits_{\rho} \sum\limits_{\pi_p \in
 \Pi_{\psi}(\GU, \varrho)} \langle \pi_p, \eta(s) \rangle \frac{\tr( \eta(s) \mid V_\rho)}{\dim \rho} \pi_p \boxtimes [\rho \otimes | \cdot |^{-\langle \rho_{\GU}, \mu \rangle}], 
\end{equation*}
where the first sum on the right-hand side is over irreducible factors of the representation $r_{- \mu} \circ \psi$ and $V_{\rho}$ is the $\rho$-isotypic part of $r_{- \mu} \circ \psi$.
\end{theorem}

This theorem is \cite[Theorem 6.4]{BM2}. To verify this theorem we essentially just need to check a number of hypotheses from \cite{BM2}. 

First, we need a global group $\bm{\GU}$ such that $\bm{\GU}_{\Q_p} \cong \GU$ and such that there exists a Shimura variety $(
\bm{\GU}, X)$ of PEL type such that the global conjugacy class of cocharacters $\{\bm{\mu}\}$ of $\widehat{\bm{\GU}}$ associated to $X$ localizes to the conjugacy class of $\mu$. Since $\mu$ is assumed minuscule, its weights are equal to $1$ and $0$. In particular, $\mu$ is determined by a pair $(p,q)$ such that $p+q=n$ and $p$ denotes the number of $1$ weights and $q$ denotes the number of $0$ weights.  

We fix $n$ an odd positive integer and define $\bm{\GU}$ to be the group $\mb{GU}(p, q)$ coming from the hermitian form $I_{p,q}$ as in Section \ref{section1}. Following \cite[\S 2.1]{Mo}, we have a PEL Shimura $(\mb{GU}, X)$ for this group (in Morel's notation, this is the datum $(\bm{\GU}, \mc{X}, h)$). As we observed in Section \ref{section1}, the group $\bm{\GU}$ can be equipped with the structure of an extended pure inner twist $(\varrho, z)$. As in \cite{BM3},  this twist gives us for each refined endoscopic datum $(\mb{H}, s, \eta)$ of $\mb{GU}$ a normalized transfer factor at each place $v$.

We observe that, in accordance with \cite[\S 4.1, \S 5.1 ]{BM2}, we have $\mb{GU}_{\der}$ is simply connected and $\mb{GU}_{\Q_p}$ is unramified. The center $Z(\mb{GU})$ is isomorphic to $\Res_{E/ \Q} \Gm$ which has split rank equal to $1$. Since $E/\Q$ is an imaginary quadratic extension, the split rank of $\Z(\mb{GU})_{\R}$ also equals $1$.

We verify that $\mb{GU}$ satisfies the Hasse principle. By \cite[Lemma 4.3.1]{Kot5} it suffices to show that $\ker^1(\Q, \mb{GU}/\mb{GU}_{\der})=\ker^1(\Q, \Gm)$ vanishes but this latter group is trivial.

We now note an important difference between the exposition in \cite[\S 4]{BM2} and our current situation. This is that the group $\mb{GU}$ will not in general be anisotropic modulo center. For this reason, the stabilization of the trace formula carried out in that paper does not carry over exactly to our case. 
%\begin{remark}
%It is possible to choose $\mb{GU}$ so that it is anisotropic modulo center. We have not done so in this paper because we are not able to establish the representation-theoretic results of \cite{KMSW} for such groups. In particular, this is because the results of loc. cit. are not fully proven in the case of unitary groups arising from general central simple algebras (see \cite[pg 5]{KMSW}).
%\end{remark}
Instead, we use Morel's work on the cohomology of these Shimura varieties to establish the desired stabilization. However, Morel's work is on the intersection cohomology of Shimura varieties whereas we need to study compactly supported cohomology. We introduce some necessary notation.

Let $K \subset \mb{GU}(\A_f)$ be a compact open subgroup that factors as $K^pK_p$ where $K_p$ is a hyperspecial subgroup of $\mb{GU}(\Q_p)$. Following the notation of \cite{Mo}, we let $M^K(\mb{GU}, \mc{X})^*$ be the Baily-Borel-Satake compactification of the Shimura variety $M^K(\mb{GU}, \mc{X})$. Fix primes $p$ and $ \ell $ and an algebraic representation $V$ of $\mb{GU}$. Choose the highest weight of $V$ to be `sufficiently regular' in the sense of \cite[Def. 2.2.10]{KH1}.  Let $L \subset \C$ be a number field containing the field of definition of $V$ and let $\lambda$ be a place of $L$ over $\ell$. Then let $IC^KV$ denote the intersection complex on $M^K(\mb{GU}, \mc{X})^*$ with coefficients in $V$. Then we define an element $W^I_{\lambda}$ in the Grothendieck group of $\mc{H}_K \times \Gal(\ov{\Q} / E_{\bm{\mu}})$ representations by
\begin{equation*}
    W^{I}_{\lambda} := \sum\limits_{i \geq 0} (-1)^i [H^i (M^K(\mb{GU}, \mc{X})^*_{\ov{\Q}} , IC^KV_{\ov{\Q}})].
\end{equation*}

Similarly, we let $\mc{F}$ be the local system on $M^K(\mb{GU}, \mc{X})$ associated to $V$ and define the element $W^C_{\lambda}$ in the Grothendieck group of $\mc{H}_K \times \Gal(\ov{\Q} / E_{\bm{\mu}})$ representations by
\begin{equation*}
    W^{C}_{\lambda} := \sum\limits_{i \geq 0} (-1)^i [H^i_c (M^K(\mb{GU}, \mc{X})_{\ov{\Q}} , \mc{F})].
\end{equation*} 

Let $f^{\infty} \in \mc{H}_K$ and assume that $f^{\infty}$ factors as $f^{\infty}=f^{p,\infty} 1_{K_p}$. Fix a place $\mf{p}$ of $E_{\bm{\mu}}$ above $p$ and let $\Phi_{\mf{p}}$ be a lift of of the geometric Frobenius at $\mf{p}$. We will often consider functions $f \in \mb{GU}(\A)$ such that  $f=f^{\infty}f_{\infty}$ where  $f_{\infty}$ is stable cuspidal and  at some finite place $v$, we have $f=f^{v, \infty}f_v$ and $f_v$ is cuspidal.   For instance, $f_v$ could be a coefficient for a supercuspidal representation. Recall that these terms were defined in Section 4.2.
\begin{lemma}{\label{cuspcohlem}}
Suppose that $f$ is cuspidal at a finite place. Then we have $\tr( W^C_{\lambda} \mid f^{\infty}\times \phi^j_{\mf{p}}) = \tr(W^I_{\lambda} \mid f^{\infty} \times \phi^j_{\mf{p}})$.
\end{lemma}
\begin{proof}
Indeed, this follows from the fact we have a natural $\Gamma$-equivariant morphism for each $i$
\begin{equation}
    H^i_c (M^K(\mb{GU}, \mc{X})_{\ov{\Q}} , \mc{F}) \to H^i (M^K(\mb{GU}, \mc{X})^*_{\ov{\Q}} , IC^KV_{\ov{\Q}})
\end{equation}
and the cuspidal part of $H^i (M^K(\mb{GU}, \mc{X})^*_{\ov{\Q}} , IC^KV_{\ov{\Q}})$ lies in the image of this map (see, for instance, \cite[Proposition 3.2]{KH1}).
\end{proof}

We now remark on the definitions of the functions $f^{\mb{H}} , f^{(j)}_{\mb{H}} \in \mc{H}(H(\A))$ defined in \cite[\S 4]{BM2} and \cite[\S 6.2]{Mo} respectively. Morel's normalization of transfer factors away from $p$ and $\infty$ is arbitrary up to the global constaint given by \cite[6.10b]{kot6}. At $v \neq p, \infty$ the definitions of $f^{\mb{H}}$ and $f^{(j)}_{\mb{H}}$ coincide up to differences in transfer factor normalization. At $p$, Morel normalizes her transfer factors as in \cite[pg180]{kot7}. If one chooses a different normalization at $p$, then Kottwitz explains (\cite[pg180-181]{kot7}) how to modify the function $f^{(j)}_{\mb{H}, p}$ by a constant such that it satisfies and analogous fundamental lemma formula. At $v= \infty$, Morel uses the normalization given on \cite[pg184]{kot7}. We can again modify the function $f^{(j)}_{\mb{H}, \infty}$ by a constant so that it satisfies the same formulas. Hence, so long as one modifies the normalizations of the transfer factors  at each place in such as way that the global constraint is still satisfied, one gets an analogous modification of the function $f^{(j)}_{\mb{H}}$ such that it satisfies the same transfer formulas. By examining the constructions at each place, it is clear that if $f^{(j)}_{\mb{H}}$ is modified to be compatible with our chosen normalization of transfer factors, then the functions $f^{(j)}_{\mb{H}}$ and $f^{\mb{H}}$ can be chosen to be equal.

Since the transfer of a cuspidal function is cuspidal \cite[Lemma 3.4]{Art1} and $f^{\mb{H}}_{\infty}$ is stable cuspidal by definition, we have that $f^{\mb{H}}$ satisfies the hypotheses of Lemma \ref{cuspcohlem} and Lemma \ref{cuspstablem}. In particular, we have the following proposition.

\begin{proposition}
    Suppose $f^{\infty}$ is a cuspidal at a finite place and factors as $f^{p,\infty} 1_{K_p}$. Then
    \begin{equation*}
        \tr(W^C_{\lambda} \mid f^{\infty}\times \phi^j_{\mf{p}}) = \sum\limits_{(\mb{H}, s, \eta) \in \mc{E}(\mb{GU})} \iota(\mb{GU} ,\mb{H}) ST^{\mb{H}}_{ell}(b_Hf^{\mb{H}}).
    \end{equation*}
\end{proposition}
\begin{proof}
By Lemma \ref{cuspcohlem}  and \cite[Theorem 7.1.7]{Mo} (keeping in mind her remark that the result holds for general $p$) we have
\begin{equation*}
    \tr(W^C_{\lambda} \mid f^{\infty} \times \phi^j_{\mf{p}}) = \sum\limits_{ (\mb{H}, s, \eta) \in \mc{E}(\mb{GU})} \iota(\mb{GU} , \mb{H}) ST^{\mb{H}}(f^{\mb{H}}).
\end{equation*}
Now, we apply Lemma \ref{cuspstablem} to the righthand side to get the desired equality.
\end{proof}

At this point, we have finished using the work of Morel and have arrived at the formula \cite[Equation (4.17)]{BM2}. We now need to show that we can perform the destabilization procedure as in \cite[\S 4.7]{BM2}. To do so we need to prove that we have a sufficiently good theory of the Langlands correspondence for $\mb{GU}$ and its localizations. Globally, we will work with ``automorphic parameters'' in the style of \cite{KMSW} and \cite{ArthurBook} and as we defined in \ref{globalGUparams}. Since our ultimate goal is to prove a local formula, these parameters are sufficient for our purpose. We list the following properties we need and where these facts have been proven.

\begin{enumerate}
    \item We need a construction of local Arthur packets of generic parameters at all localizations of $\mb{GU}$ and descriptions of the elements in each local $A$-packet in terms of representations of the various centralizer groups (Theorem \ref{localpairingGU}).\\
    \item The local packets must satisfy the endoscopic character identities (Section \ref{sectionECIGU}).\\
    \item A local generic $A$-packet contains a $K$-unramified representation if and only if the parameter is unramified. In the case that an $A$-parameter is unramified, this $K$-unramified representation is unique (Subsection \ref{unramifiedGU}).\\
    \item We need a construction for global Arthur packets for generic  ``$v$-cuspidal'' parameters. These consist of parameters that are supercuspidal at some fixed local place $v$. We need a description of the global $A$-packet in terms of the local packets (Section \ref{globalGUparams}).\\
    
    \item{\label{quadbij}} We need $v$-cuspidal parameters to satisfy a version of \cite[Proposition 3.10]{BM2} (this is discussed in \cite[pg 36]{ArthurBook}).\\
    
    \item We need a decomposition of the generic $v$-cuspidal part of  $L^2_{\disc}( \mb{GU}(\Q) \setminus \mb{GU}(\A))$ in terms of global Arthur packets and this decomposition should satisfy the global multiplicity formula (Section \ref{globalGUparams}).\\
    
    \item We need to attach to a global generic parameter a global Galois representation whose localizations at each place are compatible with the corresponding localization of the global parameter (Subsection \ref{globalgaloisrep}).
\end{enumerate}

With these properties in hand, we can now apply the results of  Section 4.2 (which is analogous to \cite[Assumption 4.8]{BM2}) to get
\begin{equation*}
     \tr(W^C_{\lambda} \mid f^{\infty} \times \phi^j_{\mf{p}}) = \sum\limits_{ (\mb{H}, s, \eta) \in \mc{E}(\mb{GU})} \iota(\mb{GU} , \mb{H}) ST^{\mb{H}}_{\disc}(f^{\mb{H}}).
\end{equation*}
Following the argument of \cite[\S 4.7]{BM2}, we derive the formula
\begin{equation}
\tr(W^C_{\lambda} \mid f^{\infty}\phi^j_{\mf{p}})  = \sum\limits_{[\psi]} \, \sum\limits_{\nu}\sum\limits_{\pi^{\infty} \in \Pi_{\psi^{\infty}}(\mathbf{GU}, z^{iso, \infty})} m(\pi^{\infty}, \nu)\nu(s_{\psi})(-1)^{q(\mathbf{GU})} \tr(\pi^{\infty} \mid f^{\infty}) \boxtimes V(\psi, \nu)_{\lambda},
\end{equation}
in the Grothendieck group of $\mb{G}(\A_f) \times \Gamma_{\mb{E}}$-modules where the first sum is over equivalence classes of $v$-cuspidal parameters.

Suppose $\pi_f$ is a representation of $\mb{GU}(\A_f)$ appearing in the cohomology of Shimura varieties whose associated automorphic $A$-parameter is $v$-cuspidal. We need to compute the $\pi_f$-isotypic part, $W^C_{\lambda}(\pi_f)$, of $W^C_{\lambda}$. To do so, we apply the argument at the end of \cite[\S 4.7]{BM2} along with the following lemma.
\begin{lemma}{\label{shimuraseplem}}
Suppose $\pi_f$ is an admissible representation of $\mb{GU}(\A_f)$ such that the $A$-parameter at $v$ is supercuspidal. There exists a compact open $K \subset \mb{GU}(\A_f)$ such that $\pi^K_f \neq \emptyset$ and $K$ factors as $K^vK_v$ and there exists a $v$-cuspidal function $f^{\infty} \in \mc{H}(\mb{GU}(\A_f), K)$ such that $\tr( \pi_f \mid f^{\infty}) \neq 0$ and for any $\pi'_f$ with nontrivial $K$-invariants and appearing in either $W^C_{\lambda}$ or
\begin{equation*}
\sum\limits_{[\psi]} \, \sum\limits_{\nu}\sum\limits_{\pi^{\infty} \in \Pi_{\psi^{\infty}}(\mathbf{G}, z^{iso, \infty})} m(\pi^{\infty}, \nu)\nu(s_{\psi})(-1)^{q(\mathbf{G})}(\pi^{\infty}) \boxtimes V(\psi, \nu)_{\lambda},
\end{equation*}
we have 
\begin{equation*}
    \tr(\pi'_f \mid f^{\infty})=0.
\end{equation*}
\end{lemma}
\begin{proof}
The set $R'$ of isomorphism classes of $\pi'_f$ satisfying the above conditions is finite. Hence we can find a function $f^{v, \infty}$ such that $\tr( (\pi'_f)^v \mid f^{v, \infty}) = 0$ for all $\pi'_f \in R'$ unless $(\pi'_f)^v \cong \pi^v_f$ in which case the trace is nonzero. Now, at $v$ we have that $(\pi_f)_v$ is supercuspidal and so we choose $f_v \in \mc{H}( \mb{GU}(\Q_v) , K_v)$ to be a coefficient for $(\pi_f)_v$. Then $f^{v, \infty}f_v$ has the desired properties. Indeed any $\pi'_f$ not isomorphic to $\pi_f$ will differ from $\pi_f$ either at $v$ or away from it, and hence $\tr( \pi'_f \mid f^{v, \infty}f_v)=0$.
\end{proof}

Following the argument at the end of \cite[\S 4.7]{BM2} we conclude that
\begin{equation}
    W^C_{\lambda}(\pi_f) = (\sum\limits_{[\psi]} \, \sum\limits_{\nu}\sum\limits_{\pi^{\infty} \in \Pi_{\psi^{\infty}}(\mathbf{G}, z^{iso, \infty})} m(\pi^{\infty}, \nu)\nu(s_{\psi})(-1)^{q(\mathbf{G})}(\pi^{\infty}) \boxtimes V(\psi, \nu)_{\lambda})(\pi_f)
\end{equation}
in $\Gr(\mb{GU}(\A_f) \times W_{E_{\mb{\mu}}})$.

We now need to show that a similar result holds for the compactly supported cohomology of Igusa varieties. In this case the stabilization in \cite[\S 5]{BM2} does not require that $\mb{GU}$ is anisotropic modulo center and so that argument goes through essentially unchanged. The only difference is that we only prove the equality of $ST^{\mb{H}}_{ell}(f^{\mb{H}})$ and $ST^{\mb{H}}_{\disc}(f^{\mb{H}})$ in the case that $f^{\mb{H}}_{\infty}$ is cuspidal at a finite place. In particular, this means that when we compute the $\pi_f$-isotypic part of the cohomology of Igusa varieties, we need the following lemma.
\begin{lemma}{\label{shimseplem}}
Suppose $\pi_f$ is an irreducible admissible representation of $\mb{GU}(\A^p_f) \times J_b(\Q_p)$ such that the corresponding local $A$-parameter at $v$ is supercuspidal. Let  $K \subset \mb{GU}(\A^p_f) \times J_b(\Q_p)$ be a compact open subgroup such that $\pi^K_f \neq \emptyset$ and $K$ factors as $K^{v,p}K_vK_p$. Let $R$ be a finite set of isomorphism classes of irreducible admissible $\mb{GU}(\A^p_f) \times J_b(\Q_p)$ representations such that $\pi_f \in R$. Then there exists a $v$-cuspidal function $f^{\infty} \in \mc{H}(\mb{GU}(\A^p_f) \times J_b(\Q_p), K)$ that is acceptable in the sense of \cite[Definition 6.2]{Shi4} such that $f^{\infty}$ factors as $f^{p,v,\infty}f_pf_v$ and
$\tr( \pi'_f \mid f^{\infty}) \neq 0$ for $\pi'_f \in R$ if and only if $\pi'_f \cong \pi_f$.
\end{lemma}
\begin{proof}
Consider the linear map from $v$-cuspidal functions to $\C^{|R|}$ given by $f^{\infty} \mapsto (\tr(\pi_1 \mid  f^{\infty}), ..., \tr(\pi_n \mid f^{\infty}))$ where $R= \{ \pi_1, ..., \pi_n \}$. It suffices to show this map is surjective. If the map is not surjective, then its image is a proper subspace and hence lies in a hyperplane of $\C^{|R|}$. Hence we can find some element $c_1,...,c_n \in \C[R]$ such that for all $v$-cuspidal $f^{\infty}$, we have $c_1\tr(\pi_1 \mid f^{\infty})+...+c_n\tr( \pi_n \mid f^{\infty})=0$.

Now, by the argument of \cite[Lemma 6.4]{Shi4} and also \cite[Lemma 6.3]{Shi4} it follows that every $f^{\infty} =f^{p, v, \infty} f_pf_v$ that is cuspidal at $v$ satisfies
\begin{equation*}
    \tr( c_1\pi_1+...+c_n\pi_n \mid f^{\infty}) = 0.
\end{equation*}
By the argument of Lemma \ref{shimuraseplem}, we can find an $f^{\infty}$ that does not vanish at $c_1\pi_1+...+c_n\pi_n$. This is a contradiction and implies our desired result.
\end{proof}

At this point, we have verified the assumptions of \S $4$ and \S $5$ of \cite{BM2}. It remains to check those of \S $6$. We first note that the Mantovan formula is known for the PEL type Shimura varieties we consider. Indeed this is \cite[Theorem 6.32]{LS2018}.

It remains to check Assumptions $6.2$ and $6.3$ of \cite{BM2}. We record some useful lemmas.
\begin{lemma}{\label{generic}}
Suppose $\pi$ is a discrete automorphic representation of $\mb{GU}(\A)$ contained in an $A$-packet $\Pi$. Suppose further that the infinitesimal character of $\pi_{\infty}$ is sufficiently regular in the sense of \cite[Def. 2.2.10]{KH1}. Then the $A$-parameter associated to $\Pi$ is generic.
\end{lemma}
\begin{proof}
Standard. For instance see \cite[Lemma 4.3.1]{KMSW}. 
\end{proof}
\begin{lemma}{\label{samepacketlem}}
Suppose $\pi$ is a discrete automorphic representation of $\mb{GU}(\A)$ contained in an $A$-packet and such that $\pi_{\infty}$ has sufficiently regular infinitesimal character. Then this is the unique $A$-packet containing $\pi_{\infty}$. Moreover, if $\widetilde{\pi}$ is another discrete automorphic representation of $\mb{GU}(\A)$ such that $\widetilde{\pi}_{\infty}$ has sufficiently regular infinitesinal character and such that $\pi^{\infty} \cong \widetilde{\pi}^{\infty}$ then $\pi$ and $\widetilde{\pi}$ are in the same $A$-packet.
\end{lemma}
\begin{proof}
Suppose that $\pi$ belongs to two $A$-packets with associated $A$-parameters $ (\Dot{\psi}_1, \chi_1) $ and $ (\Dot{\psi}_2, \chi_2) $. Since $ \chi_1 $ and $\chi_2$ correspond to the central character of $ \pi $, they are equal. We need to show that $ \Dot{\psi}_1 $, $ \Dot{\psi}_2 $ are also equal. At almost all finite unramified places $v$ where $ \pi_v $ is unramified, the localizations $ (\Dot{\psi}_1)_v $ and $ (\Dot{\psi}_2)_v $ are equal. Indeed, our sufficiently regular assumption implies that these parameters are generic. Following \cite[pg 189]{C.P.Mok}, these local parameters factor through $ \prescript{L}{}{M} $ where $M$ is the minimal Levi subgroup of $U_{E_v/\Q_v}(n)$ and correspond to the same spherical parameter of $ M $ (for more details, see \cite{C.P.Mok}, page $189$). This implies that $\Dot{\psi}^n_1$ and $ \Dot{\psi}^n_2 $ give rise to the same Hecke string. Then, by \cite{JS}, \cite[Theorem 4.3]{Art4} we see that $ \Dot{\psi}_1 $ and $\Dot{\psi}_2$ are equal. 
It is clear that the second statement also follows from exactly the same argument.
\end{proof}

Before verifying Assumptions on $6.2$ and $6.3$ of \cite{BM2}, we need to understand the effect of an unramified twist on the $\Mant_{\GU, b, \mu}$ map. Let $c : \GU(\Q_p) \longrightarrow \Q_p^{\times} $ be the similitude factor character. For $b$ non-basic, the group $\J_b(\Q_p)$ is an inner form of a Levi subgroup $\mathrm{M}_b(\Qp)$ of $\GU(\Q_p)$. Then the similitude character $c$ restricted to $\mathrm{M}_b(\Qp)$ can be transferred to $\J_b(\Q_p)$. Hence by abuse of language, we also denote $c$ the corresponding character on $\J_b(\Q_p)$. 

\begin{lemma} \phantomsection \label{itm : unramified twisting}
Let $ (E_p, *, V, \langle \cdot | \cdot \rangle, \GU, \mu, b)$ be an unramified unitary Rapoport--Zink PEL datum and suppose $ \omega : \Q_p^{\times} \longrightarrow \overline{\Q}^{\times}_{\ell} $ is an unramified character. Then the following holds in $ \Groth (\GU(\Q_p) \times W_{E_p} ) $
\[
\Mant_{\GU, b, \mu} ( \rho \otimes (\omega \circ c) ) = \Mant_{\GU, b, \mu} ( \rho ) \otimes (\omega \circ c) \otimes (\omega \circ Art^{-1}_{E_p}).
\]
\end{lemma}

\begin{proof}
This lemma is an analogue of \cite[Lemma 4.9]{Shi1} and the same proof applies in our situation. Thus we just briefly give an idea of how to proceed. 

Define a character $\chi$ of $ \J_b(\Q_p) \times \GU(\Q_p) \times W_{E_p} $ by
\[
\chi := (\omega \circ c) \otimes (\omega \circ c) \otimes (\omega \circ Art^{-1}_{E_p}).
\]

Then we prove that there is an isomorphism of $ \overline{\Q}_{\ell} $-vector spaces
\[
H_c^j (\mathcal{M}_{K_p}, \overline{\Q}_{\ell} ) \simeq H_c^j (\mathcal{M}_{K_p}, \overline{\Q}_{\ell} ) \otimes \chi
\]
compatible with the action of $ \J_b(\Q_p) \times ( K_p \backslash \GU(\Q_p) / K_p ) \times W_{E_p} $.

Notice there is a $\J_b(\Q_p)$-equivariant map $ \iota : \mathcal{M}_{K_p} \longrightarrow \Delta := \Hom_{\Z} (X^*(\GU), \Z)$ (\cite[sec. 3.52]{RZ1}) and moreover there is a natural way to define an action of $ \J_b(\Q_p) \times \GU(\Q_p) \times W_{E_p} $ on $\Delta$ such that the map $\iota$ is equivariant with respect to $ \J_b(\Q_p) \times ( K_p \backslash \GU(\Q_p) / K_p ) \times W_{E_p} $ (\cite[remark 2.6.11]{Far1}).

We can prove the lemma by using the fact that $ \chi $ acts trivially on $ ( \J_b(\Q_p) \times ( K_p \backslash \GU(\Q_p) / K_p ) \times W_{E_p} )^1 $ and 
\[
\mathop{\mathrm{lim}}_{\overrightarrow{K_p}} H_c^j (\mathcal{M}_{K_p}, \overline{\Q}_{\ell} )  \simeq c-\ind^{\J_b(\Q_p) \times ( K_p \backslash GU(\Q_p) / K_p ) \times W_{E_p}}_{( \J_b(\Q_p) \times ( K_p \backslash GU(\Q_p) / K_p ) \times W_{E_p} )^1} \Big( \mathop{\mathrm{lim}}_{\overrightarrow{K_p}} H_c^j (\mathcal{M}^{(0)}_{K_p}, \overline{\Q}_{\ell} ) \Big)
\]
where $\mathcal{M}^{(0)}_{K_p}$ is the inverse image of $0$ by $\iota$ and $(\J_b(\Q_p) \times (K_p \backslash \GU(\Q_p) / K_p) \times W_{E_p})^1 $ is the subgroup of $ \J_b(\Q_p) \times (K_p \backslash \GU(\Q_p) / K_p) \times W_{E_p} $ that acts trivially on $\Delta$.

\end{proof}

We can now settle Assumption $6.2$ in the cases we need. Let $\pi_p$ be a representation of $\mb{GU}(\Q_p)$ and $\pi_1$ a discrete automorphic representation of $\mb{GU}(\A)$ such that $(\pi_1)_p \cong \pi_p$. Suppose further that $\pi_1^{\infty}$ appears in either the formula for the cohomology of Igusa varieties or $W^C_{\lambda}$. Then since $V$ has sufficiently regular infinitesimal character, it follows that the same is true of $(\pi_1)_{\infty}$. Now suppose $\pi_2$ is a discrete automorphic representation of $\mb{GU}(\A)$ appearing in either of the above formulas and such that $\pi^{\infty}_1 \cong \pi^{\infty}_2$. We then have by Lemma \ref{samepacketlem} that $\pi_2$ and $\pi_1$ are in the same packet.

We now tackle Assumption 6.3 of \cite{BM2}. For a fixed supercuspidal representation $\pi_p$ with local parameter $\phi_p$, we have the local centralizer group $S_{\phi_p}$. For any global $A$-parameter $\psi$ such that $\psi_p = \phi_p$, we have a natural embedding $S_{\psi} \hookrightarrow S_{\phi_p}$. The formula immediately before Assumption 6.3 of \cite{BM2} includes a sum indexed over a set of representatives $X_{\psi}$ of $\ov{S}_{\psi}$. We must show that we can pick different globalizations, $\psi$, of $\phi_p$ to derive the formula below assumption 6.2 for each element of $S_{\psi_p}$.

Suppose first that $s \in S_{\phi_p}$ projects to the identity element of $\ov{S}_{\phi_p}$. Then by Construction \ref{construction1} we can choose $\psi$ so that the image of $S_{\psi}$ in $S_{\psi_p}$ is $\{ \pm s\}$ and the packet $ \Pi_{\psi_p}(\GU(\Qp), \rho_{\GU}) $ differs from the packet $ \Pi_{\phi_p}(\GU(\Qp), \rho_{\GU}) $ by an unramified twist of the form $\omega \circ c$. Then we simply pick $X_{\psi}$ to contain the unique element of $S_{\psi}$ mapping to $s$. This establishes the formula for $s$ projecting to the identity of $\ov{S}_{\psi_p}$. By Lemma \ref{itm : unramified twisting}, we obtain the formula for $s$ projecting to the identity of $\ov{S}_{\phi_p}$.

Now suppose pick $s \in S_{\phi_p}$ that projects to a non-identity element. By Construction \ref{construction2}, we may choose $\psi$ such that the image of $S_{\psi}$ in $S_{\psi_p}$ is precisely $\{ \pm s, \pm \Id\}$ and the packet $ \Pi_{\psi_p}(\GU(\Qp), \rho_{\GU}) $ differs from the packet $ \Pi_{\phi_p}(\GU(\Qp), \rho_{\GU}) $ by an unramified twist of the form $\omega \circ c$. Choose $X_{\psi}$ to contain the unique elements mapping to $s, \Id$ and denote these $x_s$ and $x_{\Id}$ respectively. Then each side of the formula before Assumption 6.2 for the parameter $ \psi_p $ has two terms indexed by $x_s$ and $x_{\Id}$ respectively. Again, by Lemma \ref{itm : unramified twisting}, we can draw the same formula for $\phi_p$. The $x_{\Id}$ terms are already known to be equal by the previous paragraph. It therefore follows that the $x_s$ terms are equal as well.

This completes the verification of Theorem \ref{AvgFormula}.

\section{Proof of the main theorem}

To prove the Kottwitz conjecture for the groups we consider, we use Theorem \ref{AvgFormula}. 

First of all, we show that
\[
\mathrm{Red}^{\mc{H}^{\mathfrak{e}}}_b(\sum\limits_{\pi^{\mH} \in \Pi_{\psi^{\mH}}} \langle \pi^{\mH}, s_{\psi^{\mH}} \rangle \pi^{\mH})) = 0
\]
for $b$ non basic, $(\mH, s, \eta)$ an elliptic endoscopic datum of $\GU$ and $\psi$ a supercuspidal parameter.

Indeed, the parameter $\psi^{\mH}$ is again a supercuspidal $L$-parameter. In particular, the representations $\pi^{\mH}$ are supercuspidal. Now by definition we have
\begin{equation*}
\mathrm{Red}^{\mc{H}^{\mf{e}}}_b =\sum\limits_{X^{\mf{e}}_{J_b}} \overline{\delta^{\frac{1}{2}}_{P(\nu_b)}} \otimes \mathrm{Trans}^{\mH_{\M_b}}_{\J_b} \mathrm{Jac}^{\mH}_{P(\nu)^{op}}
\end{equation*}

As $b$ is non-basic, the group $\J_b$ is an inner form of a proper Levi subgroup of $\GU$. Suppose that $P(\nu)^{op}$ = $\mH$. In this case $\mH$ equals $\mH_{\M}$ and is isomorphic to an endoscopic group of $\J_b$. This is a contradiction because by the classification of the endoscopic groups of $\GU$ and its Levi subgroups, we know that the elliptic endoscopic groups of $\GU$ are not endoscopic groups of any proper Levi subgroup of $\GU$. We conclude that $P(\nu)^{op}$ is a proper parabolic subgroup of $\mH$ so that 
\[
\mathrm{Red}^{\mc{H}^{\mathfrak{e}}}_b(\sum\limits_{\pi^{\mH} \in \Pi_{\psi^{\mH}}} \langle \pi^{\mH}, s_{\psi^{\mH}} \rangle \pi^{\mH})) = 0,
\]
as desired. 

Now, for $b$ basic, the main formula of Theorem \ref{AvgFormula} becomes
\begin{equation*}
  \mathrm{Mant}_{\GU,b, \mu}( \mathrm{Trans}^{\mH}_{\J_b} (\sum\limits_{\pi^{\mH} \in \Pi_{\psi^{\mH}}} \langle \pi^{\mH}, 1 \rangle \pi^{\mH} ))=\sum\limits_{\rho} \, \sum\limits_{\pi \in \Pi_{\psi}(\GU, \varrho)} \langle \pi, \eta(s) \rangle \frac{\mathrm{tr}(\eta(s) \mid V_\rho)}{\dim \rho} [\pi][\rho \otimes | \cdot |^{-\langle \rho_{\GU}, \mu \rangle}].   
\end{equation*}
Simplifying the left-hand side using the endoscopic character identities gives
\begin{equation*}
  \mathrm{Mant}_{\GU,b, \mu}(  \sum\limits_{\pi_{\J_b} \in \Pi_{\psi}(\J_b, \varrho_b)} \langle \pi_{\J_b}, \eta(s) \rangle \pi_{\J_b} )=\sum\limits_{\rho} \, \sum\limits_{\pi \in \Pi_{\psi}} \langle \pi, \eta(s) \rangle \frac{\mathrm{tr}(\eta(s) \mid V_\rho)}{\dim \rho} [\pi][\rho \otimes | \cdot |^{-\langle \rho_{\GU}, \mu \rangle}].   
\end{equation*}
Now, fix $\pi_{\J_b} \in \Pi_{\psi}(\J_b, \varrho_b)$ and multiply the above equation by $\langle \pi_{\J_b}, \eta(s) \rangle^{-1}$. Then one can check that both sides only depend on the projection $\ov{\eta(s)} \in \ov{\mathcal{S}}_{\psi}$. We then average over  $\ov{\mathcal{S}}_{\psi}$. This gives equality between
\begin{equation*}
  \mathrm{Mant}_{\GU,b, \mu} \left(\frac{1}{|\ov{\mathcal{S}}_{\psi}|} \sum\limits_{s \in \ov{\mathcal{S}}_{\psi}} \sum\limits_{\pi'_{\J_b} \in \Pi_{\psi}(\J_b)} \langle \pi_{\J_b}, s \rangle^{-1} \langle \pi'_{\J_b}, s \rangle \pi'_{\J_b} \right)
  \end{equation*}
  and
  \begin{equation*}
  \frac{1}{|\ov{\mathcal{S}}_{\psi}|}\sum\limits_{s \in \ov{\mathcal{S}}_{\psi}}\sum\limits_{\rho} \, \sum\limits_{\pi \in \Pi_{\psi}(\GU, \varrho)} \langle \pi_{\J_b}, s \rangle^{-1}\langle \pi, s \rangle \frac{\mathrm{tr}(s \mid V_\rho)}{\dim \rho} [\pi][\rho \otimes | \cdot |^{-\langle \rho_{\GU}, \mu \rangle}].   
\end{equation*}
Now, for any irreducible representation $\chi$ of $\ov{\mathcal{S}}_{\psi}$, we have $\frac{1}{|\ov{\mathcal{S}}_{\psi}|}\sum\limits_{s \in \ov{\mathcal{S}}_{\psi}} \chi(s)$ is $1$ if $\chi$ is trivial and $0$ otherwise. Hence we get the equality
\begin{equation*}
  \mathrm{Mant}_{\GU,b, \mu}( \pi_{\J_b} )=
  \frac{1}{|\ov{\mathcal{S}}_{\psi}|}\sum\limits_{s \in \ov{\mathcal{S}}_{\psi}}\sum\limits_{\rho} \, \sum\limits_{\pi \in \Pi_{\psi}(\GU, \varrho)} \langle \pi_{\J_b}, s \rangle^{-1}\langle \pi, s \rangle \frac{\mathrm{tr}(s \mid V_\rho)}{\dim \rho} [\pi][\rho \otimes | \cdot |^{-\langle \rho_{\GU}, \mu \rangle}].   
\end{equation*}
We now isolate the term for a fixed $\pi_{\GU} \in \Pi_{\psi}(\GU, \varrho)$ and representation $\rho$. It is
\begin{equation*}
     \frac{1}{|\ov{\mathcal{S}}_{\psi}|}\sum\limits_{s \in \ov{\mathcal{S}}_{\psi}} \langle \pi_{\J_b}, s \rangle^{-1}\langle \pi_{\GU}, s \rangle \frac{\mathrm{tr}(s \mid V_\rho)}{\dim \rho} [\pi_{\GU}][\rho \otimes | \cdot |^{-\langle \rho_{\GU}, \mu \rangle}],
\end{equation*}
which equals
\begin{equation*}
    \frac{\dim \Hom_{\ov{\mathcal{S}}_{\psi}}(\iota_{\mf{w}}(\pi_{\J_b}) \otimes \iota_{\mf{w}}(\pi_{\GU})^{\vee}, V_{\rho})}{\dim \rho}[\pi_{\GU}][\rho \otimes | \cdot |^{-\langle \rho_{\GU}, \mu \rangle}].
\end{equation*}
This equals
\begin{equation*}
    [\pi_{\GU}][ \Hom_{\ov{\mathcal{S}}_{\psi}}(\iota_{\mf{w}}(\pi_{\J_b}) \otimes \iota_{\mf{w}}(\pi_{\GU})^{\vee}, V_{\rho}) \otimes | \cdot |^{-\langle \rho_{\GU}, \mu \rangle}].
\end{equation*}
Hence summing over $\rho$, we get
\begin{equation*}
  \mathrm{Mant}_{\GU,b, \mu}( \pi_{\J_b} )=
 \sum\limits_{\pi_{\GU} \in \Pi_{\psi}(\GU, \varrho)} [\pi_{\GU}][ \Hom_{\ov{\mathcal{S}}_{\psi}}(\iota_{\mf{w}}(\pi_{\J_b}) \otimes \iota_{\mf{w}}(\pi_{\GU})^{\vee}, r_{-\mu} \circ \psi) \otimes | \cdot |^{-\langle \rho_{\GU}, \mu \rangle}].   
\end{equation*}
In conclusion we have proven
\begin{theorem}[Kottwitz Conjecture]{\label{Kottwitzconjintext}}
For irreducible admissible representations $\pi_{\J_b}$ of $\J_b(\Q_p)$ with supercuspidal $L$-parameter $\psi$, we have the following equality in $\Groth(\G(\Q_p) \times W_{E_{\mu}})$:
\begin{equation*}
  \mathrm{Mant}_{\G, b, \mu}( \rho )=
 \sum\limits_{\pi_{\GU} \in \Pi_{\psi_{\rho}}(\G)} [\pi_{\GU}][ \Hom_{\ov{\mathcal{S}}_{\psi_{\rho}}}(\iota_{\mf{w}}(\pi_{\J_b}) \otimes \iota_{\mf{w}}(\pi_{\GU})^{\vee}, r_{-\mu} \circ \psi) \otimes | \cdot |^{-\langle \rho_{\G}, \mu \rangle}].
\end{equation*}
\end{theorem}

\printbibliography

\end{document}